\documentclass[11pt,a4paper]{article}
\usepackage[utf8]{inputenc}
\usepackage[T1]{fontenc}

\usepackage{amsmath}
\usepackage{amsfonts}
\usepackage{amssymb}
\usepackage{amsthm}
\usepackage{upgreek}
\usepackage[final]{changes}

\newcommand{\stkout}[1]{\ifmmode\text{\sout{\ensuremath{#1}}}\else\sout{#1}\fi}
\setdeletedmarkup{\stkout{#1}}

\usepackage{mathtools}

\usepackage{dsfont} 
\usepackage{cleveref}
\usepackage{calc}

\usepackage{pgf,tikz}
 \usepgflibrary{fpu}

\usepackage{filecontents}
\usetikzlibrary{arrows,positioning,patterns}

\crefformat{equation}{(#2#1#3)}

\newtheorem{theorem}{Theorem}[section]
\newtheorem{proposition}[theorem]{Proposition}
\newtheorem{lemma}[theorem]{Lemma}
\newtheorem{corollary}[theorem]{Corollary}

\theoremstyle{definition}

\newtheorem{remarkx}[theorem]{Remark}
\newenvironment{remark}
  {\pushQED{\qed}\remarkx}
  {\popQED\endremarkx}
\newtheorem{assumptionx}[theorem]{Assumption}
\newenvironment{assumption}
  {\pushQED{\qed}\assumptionx}
  {\popQED\endassumptionx}

\numberwithin{equation}{section}

\def\writefig#1 #2 #3 {\rlap{\kern #1 truecm
\raise #2 truecm \hbox{\protect{\small #3}}}}

\newcommand{\Ku}{K^{u}}
\newcommand{\meta}[1]{\mathcal{M}_{#1}}

\definecolor{purple}{rgb}{.8,0,.8}

\tikzset{
    partial ellipse/.style args={#1:#2:#3}{
        insert path={+ (#1:#3) arc (#1:#2:#3)}
    }
}

\input{racc.sty}

\begin{document}

\newcommand{\setName}{\meta{k}}
\newcommand{\setA}{A} 
\newcommand{\mm}{m} 
\newcommand{\cUn}{c_{1}}
\newcommand{\rUn}{r_{1}}
\newcommand{\cDeux}{c_{2}}%
\newcommand{\lamb}{\lambda}%
\newcommand{\nr}{r_i} 

\newcommand{\setAUn}{A_1}
\newcommand{\setADeux}{A_2}
\newcommand{\normalind}[1]{\mathds{1}_{\{#1\}}}
\newcommand{\bigind}[1]{\mathds{1}_{\bigl\lbrace {#1} \bigr\rbrace
}}
\newcommand{\Bigind}[1]{\mathds{1}_{\Bigl\lbrace {#1} \Bigr\rbrace
}}

\newcommand{\ev}{\lambda} 
\newcommand{\kk}{{k}} 
\newcommand{\kkMUn}{{k-1}} 
\newcommand{\kkMDeux}{{k-2}}

\newcommand{\bigabs}[1]{\bigl\vert{#1}\bigr\vert}
\newcommand{\bigexpecim}[2]{\E_{\mskip1.5mu #1}\bigl\{#2\bigr\}}


%
\title{Spectral theory 
for random Poincar\'e maps}
\author{Manon Baudel, Nils Berglund}
\date{}

\maketitle

\begin{abstract}
\noindent
We consider stochastic differential equations, obtained by adding
weak Gaussian white noise to ordinary differential equations admitting $N$
asymptotically stable periodic orbits. We construct a discrete-time,
continuous-space Markov chain, called a random Poincar\'e map, which encodes
the metastable behaviour of the system. We show that this process admits
exactly $N$ eigenvalues which are exponentially close to $1$, and provide
expressions for these eigenvalues and their left and right eigenfunctions in
terms of committor functions of neighbourhoods of periodic orbits. The
eigenvalues and eigenfunctions are well-approximated by principal eigenvalues
and quasistationary distributions of processes killed upon hitting some of these
neighbourhoods. The proofs rely on Feynman--Kac-type representation formulas
for eigenfunctions, Doob's $h$-transform, spectral theory of compact operators,
and a recently \added{discovered} detailed-balance property satisfied by committor
functions. 
\end{abstract}

\added{\leftline{\small{\it Date.\/} November 15, 2016. 
Revised version, April 20, 2017.}}
\leftline{\small 2010 {\it Mathematical Subject Classification.\/} 
60J60,   
60J35    
(primary), 
34F05,   
45B05    
(secondary)
}


%
\section{Introduction} 
\label{sec:intro} 

A very useful method to analyse the dynamics of ordinary differential equations
(ODEs) admitting one or several periodic orbits consists in introducing a 
submanifold of codimension $1$, which is transversal to the flow. Successive
intersections of orbits with this submanifold are described by an iterated map,
called a \emph{first-return map} or \emph{Poincar\'e map}. This map has proved
extremely useful for a number of reasons. First, it replaces a $d$-dimensional
ODE by a $(d-1)$-dimensional map, which is often easier to visualise. Perhaps
more importantly, it simplifies the stability analysis of periodic orbits,
because it allows to get rid of neutral transversal directions. Furthermore,
Poincar\'e maps simplify the classification of bifurcations of periodic orbits,
since the problem is reduced to the easier one of classifying bifurcations of
fixed points of maps. 

When noise is added to an ODE, it becomes a stochastic differential equation
(SDE). SDEs with multiple periodic orbits appear in many applications, such as
enzyme reaction models~\cite{moran1984onset,citri1988systematic}, 
neuron dynamics~\cite{Ditlevsen_Greenwood_13,Hopf_Loch_Thieu_16} and
related piecewise deterministic Markov processes~\cite{Duarte_Locherbach_Ost}.
A natural analogue of the Poincar\'e map in this situation was introduced
in~\cite{HitczenkoMedvedev}, and further analysed in~\cite{HitczenkoMedvedev1},
by Hitczenko and Medvedev who called it \emph{Poincar\'e map of randomly
perturbed periodic motion}, or \emph{random Poincar\'e map} for short. Random
Poincar\'e maps have already proved useful in several applications: they allowed
to study interspike interval statistics in the stochastic FitzHugh--Nagumo
equations~\cite{BerglundLandon}, the first-passage location through an unstable
periodic orbit in planar SDEs~\cite{berglund2014noise}, and mixed-mode
oscillation patterns in systems admitting a folded-node
singularity~\cite{Berglund_Gentz_Kuehn_2015}. 

Mathematically, a random Poincar\'e map is described by a discrete-time,
continuous-space Markov chain. If the ODE perturbed by noise admits $N\geqs2$
stable periodic orbits, and the noise intensity $\sigma$ is weak, the Markov
chain will tend to spend very long time intervals in small neighbourhoods of the
periodic orbits, with occasional transitions between these neighbourhoods. This
kind of behaviour is known as \emph{metastability}. 

The metastable dynamics of SDEs has been studied on the level of exponential
asymptotics by Freidlin and Wentzell~\cite{FreidlinWentzell_book}, using  the
theory of large deviations. In the particular case where the original ODE
derives from a potential and the noise is homogeneous and isotropic, the
perturbed system's invariant measure is known explicitly, and the dynamics is
reversible with respect to this measure. Reversibility greatly simplifies the
analysis of the system. In particular, the potential-theoretic approach
developed by Bovier, Eckhoff, Gayrard and Klein in~\cite{BEGK,BGK} for SDEs
yields very precise estimates on metastable transition times and small
eigenvalues of the generator, which are governed by the so-called
Eyring--Kramers formula. See for instance~\cite{Berglund_Kramers} for a recent
review, and the monograph~\cite{Bovier_denHollander_book} for a comprehensive
account of the potential-theoretic approach. 

A drawback of the potential-theoretic approach to metastability is that it has
so far only been developed in the reversible case. Systems admitting periodic
orbits are, however, strongly non-reversible. Recently, there have been a few
attempts to derive Eyring--Kramers-like formulas for non-reversible systems. For
instance, in~\cite{Lu_Nolen_2015} Lu and Nolen provided expressions for
transition times and reactive times in terms of committor functions (that is,
probabilities to hit a set $A$ before a set $B$), based on the transition-path
theory introduced by E and Vanden--Eijnden~\cite{E_VandeEijnden_2006}.
In~\cite{Bouchet_Reygner_2016}, Bouchet and Reygner formally derived an
Eyring--Kramers law for a class of non-reversible systems admitting an isolated
saddle, based on WKB asymptotics. Furthermore, in~\cite{Landim_Seo_2016}, Landim
and Seo obtained an Eyring--Kramers formula for certain non-reversible
\added{random walks} for which the invariant measure is
explicitly known, \added{using two variational formulae for the capacity.
In~\cite{Landim_Mariani_Seo_2017} Landim, Mariani, and Seo 
provide a sharp estimate for the transition times between two different wells 
for a class of non-reversible diffusion processes (again with known invariant
measure).}
 Despite these promising results, a full theory providing sharp asymptotics for metastable
transition times for general non-reversible systems has yet to be developed. 

Fortunately, it turns out that some central ideas in~\cite{BGK}, concerning the
spectral properties of the generator, do in fact not require any
potential-theoretic tools. The key assumption is that the metastable states can
be ordered in a particular way, from most stable to least stable, forming a
so-called \emph{metastable hierarchy}. Furthermore, it has become apparent that
the small eigenvalues of the diffusion and the corresponding eigenfunctions are
strongly connected to principal eigenvalues and quasistationary distributions
(QSDs) of certain related processes. See for
instance~\cite{Bianchi_Gaudilliere_2016} for the case of reversible Markovian
jump processes,~\cite{Champagnat_Villemonais_16} for birth-and-death processes
and related population models, and~\cite{GLPN_16} for the case of reversible
diffusions. Principal eigenvalues and QSDs are much easier to determine
numerically than arbitrary eigenvalues and eigenfunctions.

The aim of the present work is to derive spectral information on random
Poincar\'e maps, associated with non-reversible SDEs obtained by perturbing ODEs
admitting $N\geqs2$ asymptotically stable periodic orbits. Indeed, discrete-time
continuous-space Markov chains are amenable to Fredholm theory, showing that
transition probabilities can be represented as sums of projectors on invariant
subspaces multiplied by eigenvalues. Our main result,
Theorem~\ref{thm:eigenvalues}, shows that for sufficiently small noise, the
random Poincar\'e map admits exactly $N$ eigenvalues exponentially close to $1$,
\added{which are all real}, while all remaining eigenvalues are bounded away
from $1$. Theorems~\ref{thm:right_eigenfunctions}
and~\ref{thm:left_eigenfunctions} provide expressions for the associated right
and left eigenfunctions. All these quantities are expressed in terms of
committor functions associated with small neighbourhoods of the stable periodic
orbits. Furthermore, we show that the eigenvalues and left eigenfunctions are
well approximated by principal eigenvalues and QSDs of processes killed upon
hitting some of these neighbourhoods. Therefore our results provide links
between spectral properties of the random Poincar\'e map and quantities that are
accessible to numerical simulations. 

\added{The spectral decomposition that we obtain can be interpreted as showing
that on long timescales, the dynamics of the system can be described by an
$N$-state Markov chain. The $N$ states correspond to the $N$ stable periodic
orbits, and one-step transition probabilities between different states are
exponentially small. In particular, the metastable hierarchy assumption implies
that there are $N-1$ timescales of the form $\e^{H_i/\sigma^2}$, with
$H_1>H_2>\dots>H_{N-1}>0$. The time needed to reach 
\added{the union of the $k$ first periodic orbits} starting from the $k+1^{\textrm{st}}$ orbit is of order
$\e^{H_k/\sigma^2}$, while the $k^{\textrm{th}}$ eigenvalue of the random
Poincar\'e map behaves like $1-\e^{-H_k/\sigma^2}$. Note that this is
compatible with~\cite[Theorem~7.3, Chapter~6]{FreidlinWentzell_book}, which
states that the generator of the diffusion admits $N-1$ eigenvalues with
exponentially small real parts, of order $-\e^{-H_k/\sigma^2}$. A new feature
of our results is that they concern the eigenvalues of the discrete-time
Markov chain instead of the continuous-time generator, and that we are able to
prove that these eigenvalues are real. Apart from this relation interpretable
in terms of metastable transition times, the general link between
eigenvalues of the discrete-time and continous-time generators is not yet fully
understood (except in trivial cases where the dynamics transversal to periodic
orbits is completely decoupled from the phase dynamics).}

To obtain these results, we combine various techniques developed in prior works.
One of them is the representation of eigenfunctions in terms of Laplace
transforms of hitting times of well-chosen sets, already present in~\cite{BGK}.
Another key idea is the fact, discovered in~\cite{betz2016multi}, that committor
functions of not necessarily reversible Markov chains satisfy a kind of detailed
balance condition. We also rely on perturbation theory for compact linear
operators (see e.g.~\cite{Kato1980,gohberg2013classes}), Doob's $h$-transform,
\added{which is linked to the theory of quasi-stationary distributions as
reviewed in~\cite{Chetrite_Touchette_14}}, as well as sample-path estimates for
SDEs which were developed
in~\cite{berglund2002GeoPerSDE,Berglund_Gentz_book,berglund2014noise}. 

The remainder of this work is organised as follows. In Section~\ref{sec:setup},
we define precisely the kind of SDEs that we are going to consider, provide
a construction of their random Poincar\'e maps, and define the spectral
decomposition. Section~\ref{sec:results} contains the main results of the work.
In Section~\ref{sec:outline}, we provide an outline of the main steps of the
proofs. Subsequent sections are dedicated to technical details of the proofs. 
Section~\ref{sec:KZBi} contains estimates of the spectral gap and principal
eigenfunction of the process killed upon leaving a neighbourhood of a periodic
orbit. 
In Section~\ref{sec:diluted}, we show that the random Poincar\'e
map can be approximated by a finite-rank operator by providing estimates for the
operator norm of their difference. 
The spectral properties of the finite-rank operator are
described in Section~\ref{sec:operators}.  
Section~\ref{sec:sample_paths} provides sample-path estimates needed to apply
the bounds on operator norms, while Section~\ref{sec:lastSteps} contains the
proofs of the main results. Appendix~\ref{app:Doob} recalls some properties of
Doob's $h$-transform\added{, whereas Appendix~\ref{app:Floquet} recalls some 
results on Floquet theory.}

\medskip

\noindent
\textit{Notations:} Unless otherwise specified, $\norm{\cdot}$ denotes the
supremum norm of a function or a linear operator. The indicator function of an
event \added{or set} $A$  is denoted $\mathds{1}_A$ . The symbol 
$\id$ is used for the identity operator as well as the identity matrix. 

\medskip

\noindent
\textit{Acknowledgements:} The authors wish to thank Luc Hillairet for
useful advice on spectral-theoretical aspects, \added{and two anonymous
referees for their numerous constructive comments on the first version of the
manuscript, which allowed to substantially improve its readability.}

\section{Set-Up}
\label{sec:setup} 
\subsection{Deterministic system}
\label{ssec:setup-det} 

Let  $\cD_0 \subset \real^{d+1}$ be an open, connected set and let $f \in
\mathcal{C}^2\pth{\cD_0, \real^{d+1}}$. We consider the $(d+1)$-dimensional
deterministic ordinary differential equation (ODE) given by
\begin{equation}\label{EqODE}
\dot{z} = f\pth{z}\;.
\end{equation}

\begin{assumption}[Invariant domain]
\label{ass:domain} 
There exists a  bounded, open connected set $\cD\subset\cD_0$ which is
positively invariant under the flow of~\eqref{EqODE}.
\end{assumption}

This assumption ensures that for all $z\in\cD$ the flow $\varphi_t(z)$ is
defined for all $t\geqs0$. Recall that the image
$\setsuch{\varphi_t(z)}{t\geqs0}$ is called the (positive) \emph{orbit} of $z$.
The \emph{$\omega$-limit set} of $z$ is the set of accumulation points of
$\varphi_t(z)$ as $t\to\infty$. If $\varphi_t(z)$ is defined for all $t\leqs0$,
its set of accumulation points as $t\to-\infty$ is called the
\emph{$\alpha$-limit set} of $z$. A \emph{heteroclinic connection} between two
sets $A,B\subset\R^{d+1}$ invariant under the flow is an orbit admitting $A$ as
$\alpha$-limit set and $B$ as $\omega$-limit set. 

Recall that $\Gamma$ is a \emph{periodic orbit} of period $T>0$ of~\eqref{EqODE}
if there exists a periodic function $\gamma : \real \rightarrow \cD$ of minimal
period $T$ such that 
\begin{equation}
\dot{\gamma}\pth{t} = f\pth{\gamma\pth{t}} 
\qquad 
\forall t \in \real\;.
\end{equation}
Then $\Gamma$ is simply the image $\setsuch{\gamma(t)}{t\in[0,T)}$ of $\gamma$.
The periodic orbit is called \emph{linearly asymptotically stable} if all
Floquet multipliers of the linearised system $\dot\xi = \partial_z
f(\gamma(t))\xi$ are strictly smaller than $1$ in modulus. A periodic orbit is
\emph{linearly unstable} if it admits at least one Floquet multiplier of modulus
strictly larger than $1$. 

\begin{assumption}[Limit sets]
\label{ass:limit_sets} 
There are finitely many $\omega$-limit sets in $\cD$. They include $N\geqs2$
distinct linearly asymptotically stable periodic orbits
$\Gamma_1,\dots,\Gamma_N$. All other $\omega$-limit sets in $\cD$ are either
linearly unstable stationary points, or linearly unstable periodic orbits. 
Furthermore, there exists a smooth orientable $d$-dimensional manifold
$\Sigma\subset\cD$ with boundary $\partial\Sigma\subset\partial\cD$, such that
for all $x\in\Sigma$, $f(x)$ is not tangent to $\Sigma$ (transversality). Each
stable periodic orbit $\Gamma_i$ intersects $\Sigma$ at exactly one point
$x^\star_i$. In addition, there are no heteroclinic connections between
unstable orbits or between unstable orbits and unstable stationary points.
\end{assumption}

\added{Note that $\cD$ is not required to be simply connected: it can have the
shape of a solid torus containing all periodic orbits in its interior
(cf.~Fig.~\ref{figRandomPoincareMap}).} 
The deterministic \emph{Poincar\'e map} associated with
$\Sigma$ is then the map $\Pi:\Sigma\to\Sigma$ defined by 
\begin{equation}
\label{eq:first_return} 
 \Pi(x) = \varphi_\tau(x) 
 \qquad\text{where }
 \tau = \inf\setsuch{t>0}{\varphi_t(x)\in\Sigma}\;.
\end{equation} 
We will always implicitly assume that $\tau<\infty$ for almost all $x\in\Sigma$.
In other words, except perhaps for a set of initial conditions of zero
Lebesgue measure, orbits starting on $\Sigma$ always return to $\Sigma$ in a
finite time. 

We will denote by
\begin{equation}
\label{eq:Kj} 
\cA_j = \Bigsetsuch{x\in\Sigma}{\lim_{n\to\infty}
\Pi^n(x)=x^\star_j} 
\end{equation} 
the basin of attraction of the orbit $\Gamma_j$. The $\cA_j$ are open, disjoint
subsets of $\Sigma$, and the union of their closures is equal to $\Sigma$. 

\begin{remark}
Note that Assumption~\ref{ass:limit_sets} rules out the existence of any other
$\omega$-limit sets than periodic orbits and unstable stationary points.
We have formulated the assumption in this way for simplicity. In fact, what we
really need is that for each $\omega$-limit set other than the $\Gamma_i$,
noise added to the system is likely to move sample paths away from these sets
in a time which is negligible with respect to typical transition times between
the $\Gamma_i$. 

Furthermore, we believe that the absence of heteroclinic connections is not
required. We only need that a sample path starting near an unstable
$\omega$-limit set reaches the neighbourhood of a stable periodic orbit after a
negligible time. 
\end{remark}

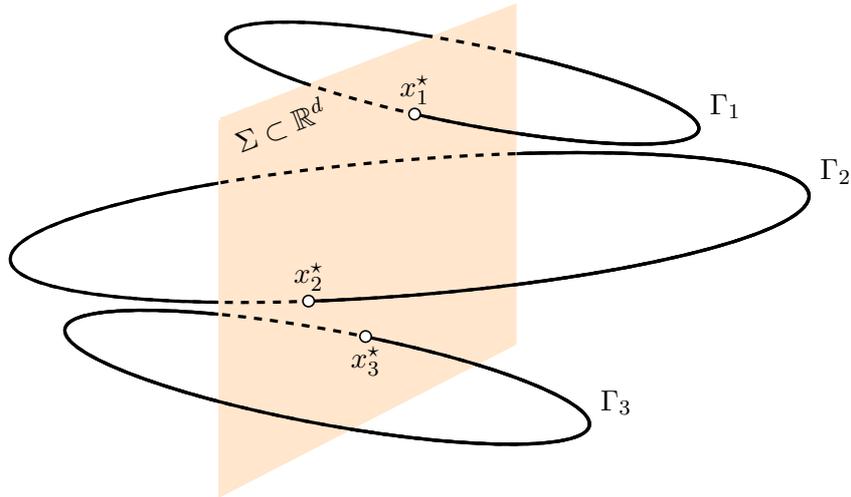
\begin{figure}
\begin{center}
\begin{tikzpicture}[scale=0.5]
\coordinate  (SigmaPGH) at (-3.3,4.4); 
\coordinate  (SigmaPGB) at (-3.3,-5.5); 
\coordinate  (SigmaPDB) at (4.4,-1.5); 
\coordinate  (SigmaPDH) at (4.4,7.4); 

\coordinate  (centreOrbiteA) at (3.05,5.39);
\coordinate  (centreOrbiteB) at (-0.51,-2.41);
\coordinate  (centreOrbiteC) at (1.66,1.56);

\coordinate (intersectionSectionPoincareA) at (1.79, 4.57);
\coordinate (intersectionSectionPoincareB) at  (-1,-0.39);
\coordinate (intersectionSectionPoincareC) at  (0.5, -1.33);

\def\petitAxeEllipseA{2.13*0.5};
\def\petitAxeEllipseB{2.5*0.5};
\def\petitAxeEllipseC{3.59*0.5};

\def\angleA{-11.26};
\def\angleB{-10.55};
\def\angleC{4.7};

\def\grandAxeEllipseA{12.67*0.5};
\def\grandAxeEllipseB{14.03*0.5};
\def\grandAxeEllipseC{21.08*0.5};

\draw [very thick,rotate around={\angleA:(centreOrbiteA)}] (centreOrbiteA)
[partial ellipse=0:360:\grandAxeEllipseA cm  and \petitAxeEllipseA  cm]
node[above right] {$\Gamma_1$};
\draw [very thick,rotate around={\angleB:(centreOrbiteB)}] (centreOrbiteB)
[partial ellipse=0:360:\grandAxeEllipseB cm  and \petitAxeEllipseB  cm]
node[above right] {$\Gamma_3$};
\draw [very thick,rotate around={\angleC:(centreOrbiteC)}] (centreOrbiteC)
[partial ellipse=0:360:\grandAxeEllipseC cm and \petitAxeEllipseC  cm]
node[above right] {$\Gamma_2$};

   \draw[fill=orange!20,draw=orange!20,line width = 2pt] (SigmaPGH) --  (SigmaPGB) -- (SigmaPDB)
-- (SigmaPDH) -- (SigmaPGH);
       \draw[fill=orange!20,draw=orange!20,line width = 2pt] (SigmaPGH) -- (SigmaPDH) node
[pos=0.17, below, sloped] {$\Sigma \subset \real^{d}$};

\draw [very thick,dashed,rotate around={\angleA:(centreOrbiteA)}]
(centreOrbiteA) [partial ellipse=0:360:\grandAxeEllipseA cm  and
\petitAxeEllipseA  cm];
\draw [very thick,rotate around={\angleA:(centreOrbiteA)}] (centreOrbiteA) [partial ellipse=-100:79:\grandAxeEllipseA cm  and \petitAxeEllipseA  cm];
\draw [very thick,dashed,rotate around={\angleB:(centreOrbiteB)}]
(centreOrbiteB) [partial ellipse=0:360:\grandAxeEllipseB cm  and
\petitAxeEllipseB  cm];
\draw [very thick,rotate around={\angleB:(centreOrbiteB)}] (centreOrbiteB) [partial ellipse=-242:83:\grandAxeEllipseB cm  and \petitAxeEllipseB  cm];
\draw [very thick,dashed,rotate around={\angleC:(centreOrbiteC)}]
(centreOrbiteC) [partial ellipse=0:360:\grandAxeEllipseC cm and
\petitAxeEllipseC  cm];
\draw [very thick,rotate around={\angleC:(centreOrbiteC)}] (centreOrbiteC) [partial ellipse=-105:73:\grandAxeEllipseC cm and \petitAxeEllipseC  cm];

\draw [semithick,color=black,fill=white,fill opacity=1]
(intersectionSectionPoincareA) circle (0.15cm) node[above] {$x^\star_1$};
\draw [semithick,color=black,fill=white,fill opacity=1]
(intersectionSectionPoincareB) circle (0.15cm) node[above] {$x^\star_2$};
\draw [semithick,color=black,fill=white,fill opacity=1]
(intersectionSectionPoincareC) circle (0.15cm) node[below] {$x^\star_3$};


\end{tikzpicture}
\caption{Sketch of Poincaré map for a deterministic system admitting several
stables periodic orbits. }
\end{center}
\end{figure}

\subsection{Stochastic system}
\label{ssec:setup-stoch} 

We turn now to random perturbations of the ODE~\eqref{EqODE}, given by It\^o
stochastic differential equations (SDEs) of the form 
\begin{equation}
 \label{eq:SDE}
 \6z_t = f(z_t)\6t + \sigma g(z_t) \6W_t\;.
\end{equation} 
Here $W_t$ denotes a $k$-dimensional standard Wiener process on a probability
space $(\Omega,\cF,\fP)$ with $k\geqs d+1$, while
$g\in\cC^1(\cD_0,\R^{(d+1)\times k})$, and $\sigma>0$ is a small parameter. We
will denote by $Z_t^{z}$, or simply $Z_t$, the solution of~\eqref{eq:SDE}
starting in $z$ at time $0$. The corresponding probability is written
$\probin{z}{\cdot}$, and expectations with respect to $\probin{z}{\cdot}$ are
denoted $\expecin{z}{\cdot}$. The \emph{infinitesimal generator} of the
diffusion process is the second-order differential operator 
\begin{equation}
 \label{eq:generator}
 \cL = \sum_{i=1}^{d+1} f_i(z) \frac{\partial}{\partial z_i}
 + \frac{\sigma^2}{2} \sum_{i,j=1}^{d+1} D_{ij}(z) \frac{\partial^2}{\partial
z_i\partial z_j}
\end{equation} 
where $D(z) = g(z)\transpose{g(z)}$ denotes the \emph{diffusion matrix}. 

\begin{assumption}[Ellipticity]
\label{ass:ellipticity} 
There exist constants $c_+ > c_- > 0$ such that 
\begin{equation}
 c_- \norm{\xi}^2 \leqs 
 \pscal{\xi}{D(z)\xi}
 \leqs c_+ \norm{\xi}^2
\end{equation}
for all $z\in\cD$ and all $\xi\in\R^{d+1}$. 
\end{assumption}

We recall a few elements from the large-deviation theory for SDEs developed by
Freidlin and Wentzell~\cite{FreidlinWentzell_book}. Given a finite time
interval $[0,T]$ and a continuous function $\gamma:[0,T]\to\cD$, one defines
a \emph{rate function} by 
\begin{equation}
\label{eq:rate_function}
I_{[0,T]}(\gamma) = 
\begin{cases}
\dfrac12 \displaystyle\int_0^T \transpose{(\dot\gamma_s - f(\gamma_s))}
D(\gamma_s)^{-1} (\dot\gamma_s - f(\gamma_s)) \,\6s
& \text{if $\gamma\in H^1$,}\\
+\infty
& \text{otherwise\;.}
\end{cases}
\end{equation}
Roughly speaking, the probability of a sample path of~\eqref{eq:SDE} tracking
$\gamma$ behaves like $\e^{-I_{[0,T]}(\gamma)/\sigma^2}$ in the limit
$\sigma\to0$ (see~\cite{FreidlinWentzell_book} for details). 
If $x^\star$ belongs to one of the $\Gamma_i$ and $y\in\cD$, we define the
\emph{quasipotential} 
\begin{equation}
 \label{eq:quasipotential}
 V(x^\star,y) = \inf_{T>0}\inf_{\gamma:x^\star\to y} I_{[0,T]}(\gamma)\;,
\end{equation} 
where the second infimum runs over all continuous paths $\gamma$ such that
$\gamma_0=x^\star$ and $\gamma_T=y$. It is easy to see that if $y_1$ and $y_2$
belong to the same periodic orbit, then $V(x^\star,y_1) = V(x^\star,y_2)$.
Indeed one can connect $y_1$ to $y_2$ at zero cost by tracking the
deterministic flow, so that $V(y_1,y_2)=0$, and similarly one has
$V(y_2,y_1)=0$. Thus for $1 \leqs i\neq j\leqs N$, the quantity 
\begin{equation}
 \label{eq:Hij}
 H(i,j) = V(x^\star_i,x^\star_j) 
\end{equation} 
measures the cost of reaching the $j$th periodic orbit from the $i$th periodic
orbit in arbitrary time. If $i\not\in A\subset\set{1,\dots,N}$ it
will be convenient to use the notation 
\begin{equation}
 \label{eq:HiA}
 H(i,A) = \min_{j\in A} H(i,j)
\end{equation} 
for the cost of reaching any of the orbits in $\bigcup_{j\in A}\Gamma_j$. 
The following non-degeneracy assumption will greatly simplify the spectral
analysis. 

\begin{assumption}[Metastable hierarchy]
\label{ass:hierarchy} 
There exists a constant $\theta>0$ such that the stable periodic orbits
$\Gamma_1, \dots, \Gamma_N$ can be ordered in such a way that if one writes
$M_j = \set{1,\dots, j}$, then 
\begin{equation}
 \label{eq:metastable_hierarchy}
 H(j,M_{j-1}) \leqs \min_{i<j} H(i,M_j\setminus\set{i}) - \theta\;.
\end{equation} 
We say that the orbits are in \emph{metastable order}, and
write $\Gamma_1 \prec \Gamma_2 \prec \dots \prec \Gamma_N$. 
\end{assumption}

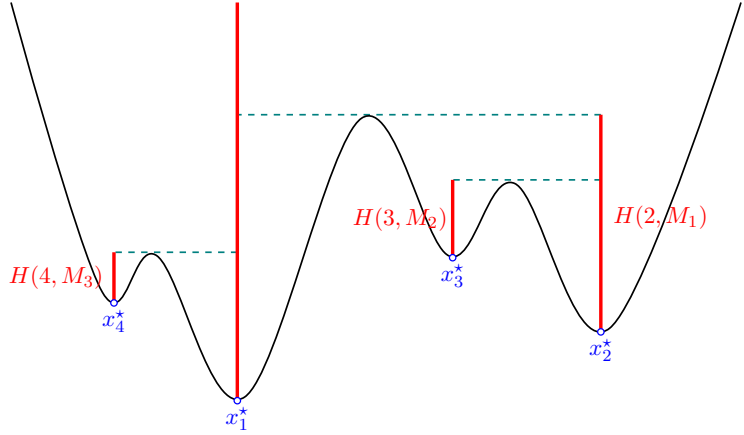
\begin{figure}[tb]
\begin{center}
\scalebox{0.8}{
\begin{tikzpicture}[>=stealth',main node/.style={draw,circle,fill=white,minimum
size=3pt,inner sep=0pt},x=1.2cm,y=1.2cm
]



\draw[teal,thick,dashed] (0.58,2.95) -- (-4.4,2.95);
\draw[teal,thick,dashed] (0.58,2.05) -- (-1.45,2.05);
\draw[teal,thick,dashed] (-4.4,1.05) -- (-6.09,1.05);


\draw[black,thick] plot[smooth,tension=.6]
  coordinates{(2.5,4.5) (0.7,0) (-0.6,2) (-1.5,1) (-2.7,2.9) (-4.3,-0.95)
(-5.5,1) (-6.25,0.5) (-7.5,4.5)};
  

\draw[red,ultra thick] (0.58,-0.05) -- 
node[red,right,xshift=0.05cm,yshift=0.1cm] {$H(2,M_1)$} (0.58,2.95); 

\draw[red,ultra thick] (-1.45,0.975) -- 
node[red,left,xshift=0.07cm] {$H(3,M_2)$} (-1.45,2.05); 

\draw[red,ultra thick] (-4.4,-1) -- (-4.4,4.5);

\draw[red,ultra thick] (-6.09,0.35) -- 
node[red,left,xshift=-0.02cm] {$H(4,M_3)$} (-6.09,1.05); 


\node[main node,blue,fill=white,semithick] (2) at (0.58,-0.05) {}; 
\node[blue] [below of=2,yshift=0.7cm] {$x^\star_2$};

\node[main node,blue,fill=white,semithick] (3) at (-1.45,0.975) {}; 
\node[blue] [below of=3,yshift=0.7cm] {$x^\star_3$};

\node[main node,blue,fill=white,semithick] (1) at (-4.4,-1) {}; 
\node[blue] [below of=1,yshift=0.7cm] {$x^\star_1$};

\node[main node,blue,fill=white,semithick] (4) at (-6.09,0.35) {}; 
\node[blue] [below of=4,yshift=0.7cm] {$x^\star_4$};



\end{tikzpicture}
}
\end{center}
\vspace{-5mm}
\caption[]{In cases where there exists a global potential $U$, such that
$V(x^\star_i,x^\star_j) - V(x^\star_j,x^\star_i) = U(x^\star_j) - U(x^\star_i)$
for all $i, j$, the metastable hierarchy can be determined by a graphical
construction. In this example, the $x^\star_i$ have already been labelled in
such a way that~\eqref{eq:metastable_hierarchy} is satisfied, so that $\Gamma_1
\prec \Gamma_2 \prec \Gamma_3 \prec \Gamma_4$.}
\label{fig:disconnectivity} 
\end{figure}

The metastable order can be determined in the following way. First
one computes,
for each $i$, the minimal cost $H(i,M_N\setminus\set{i})$ for reaching another
orbit from $\Gamma_i$. If the minimum $\min_{1\leqs i\leqs N}
H(i,M_N\setminus\set{i})$ is reached in a unique $i$, then this $i$ will be
relabelled $N$. The procedure is then reiterated with the other $N-1$ orbits,
discarding the $N$th orbit, until all orbits have been ordered.
Figure \ref{fig:disconnectivity} illustrates the procedure in case the quasipotential
derives from a global potential $U$ (i.e., such that $V(x^\star_i,x^\star_j) - 
V(x^\star_j,x^\star_i) = U(x^\star_j) - U(x^\star_i)$ for all $i, j$), which is
not the case for a generic irreversible system. 

\added{The metastable hierarchy assumption is related to the concept of
$W$-graphs used by Wentzell in~\cite{Wentzell_72_matrices} in the case of
finite matrices, and shown in~\cite[Theorem~7.3,
Chapter~6]{FreidlinWentzell_book} to determine the real parts of exponentially
small eigenvalues of the generator of a diffusion. The $W$-graphs can be used
without the metastable hierarchy assumption to determine the relevant
exponential timescales, but if this assumption holds then the $W$-graph
algorithm becomes particularly simple, since only the edges from vertex $j$ to
one vertex in $M_{j-1}$ contribute. See 
also}~\cite{Cameron_Vanden-Eijnden_2014,CameronGan2016} for recent results
\added{based on $W$-graphs} on how to determine the metastable hierarchy
efficiently in case $N$ is large.

Finally, we will need a type of recurrence assumption, because so far we have
not assumed much on the behaviour of the diffusion outside the set $\cD$. In
particular, the solutions of the SDE~\eqref{eq:SDE} may not even exist globally
in time. In fact, we will consider two slightly different situations, which
however can be treated in a uniform way. 

\begin{assumption}[Confinement property]
\label{ass:confinement} 
One of the two following situations holds. 
\begin{enumerate}
\renewcommand{\theenumi}{\Alph{enumi}}
\item 	Either there exist a \emph{Lyapunov function} $V\in\cC^2(\cD_0,\R_+)$
such that $\norm{V(z)}\to\infty$ as $z\to\partial\cD_0$ (or as
$\norm{z}\to\infty$ in case $\cD_0$ is unbounded) satisfying  
\begin{equation}
 \label{eq:recurrence1}
 (\cL V)(z) \leqs -c + d\ind{z\in\cD}
 \qquad
 \forall z\in\cD_0
\end{equation} 
for some constants $c>0$ and $d\geqs0$. 

\item 	Or  
\begin{equation}
 \label{eq:recurrence2}
 \bar V(\partial\cD) := 
 \min_{1\leqs i\leqs N} \inf_{y\in\partial\cD} V(x^\star_i,y)
 \geqs \max_{i\neq j} H(i,j) + \theta'
\end{equation} 
for a constant $\theta'>0$.
\qed
\end{enumerate}
\renewcommand{\qed}{}
\end{assumption}

\added{By~\cite[Theorem~4.2]{Meyn_Tweedie_1993b}, variant A implies that 
the process $\set{Z_t}_{t\geqs0}$ is positive Harris recurrent. Recall that a
process is Harris recurrent if there exists a $\sigma$-finite measure $\mu$ such
that the first-hitting time of a set $A$ is almost surely finite whenever
$\mu(A)>0$. 
Such a process admits an essentially unique invariant measure $\pi$, and is
called positive Harris recurrent if $\pi$ can be normalised to be a probability
measure. The ellipticity assumption~\ref{ass:ellipticity} implies that the
restriction of $\pi$ to $\cD$ is absolutely continuous with respect to Lebesgue
measure. Furthermore,}
\cite[Theorem~4.3]{Meyn_Tweedie_1993b}, applied with $f=1$, shows that the
first-hitting time $\tau_\cD$ of $\cD$ satisfies 
\begin{equation}
\label{eq:tau_D} 
 \expecin{z}{\tau_\cD} \leqs \frac{1}{c} V(z)
\end{equation} 
for all $z\in\cD_0$. 

\begin{remark}
If $V$ is a quadratic form, then we have $\cL V = \pscal{f}{\nabla V} +
\Order{\sigma^2}$. Thus a quadratic deterministic Lyapunov function satisfying
$\pscal{f}{\nabla V} \leqs -cV$ outside $\cD$ may already fulfil
Condition~\eqref{eq:recurrence1} if $\sigma$ is small enough.
\end{remark}

Variant B of Assumption~\ref{ass:confinement} says that it should be
harder to reach the boundary $\partial\cD$ of $\cD$ than to make any transition
between periodic orbits. In that situation, we are going to consider the process
conditioned on staying in $\cD$. Doob's $h$-transform
(cf.~Appendix~\ref{app:Doob}) will allow us to relate the spectral properties of
the conditioned process with those of the process killed upon leaving
$\cD$. Both processes are not influenced by what happens outside $\cD$, so that
global existence of solutions is not required.



\subsection{Random Poincar\'e map}
\label{ssec:setup-poincare} 

We now define a discrete-time process recording successive intersections of
sample paths with the surface of section $\Sigma$. The following basic estimate
shows that solutions starting in $\cD$ will hit $\Sigma$ almost surely after a
finite time \added{(and thus return to $\Sigma$ infinitely often)}. 

\begin{proposition}
\label{prop:tau_sigma} 
Let $\tau_\Sigma = \inf\setsuch{t>0}{Z_t\in\Sigma}$. There exist constants
$\sigma_0, M>0$ such that for all $\sigma < \sigma_0$, 
\begin{equation}
\label{eq:tau_Sigma} 
 \sup_{z\in\cD} \bigexpecin{z}{\tau_\Sigma} \leqs
M\log(\sigma^{-1})\;.
\end{equation} 
\end{proposition}

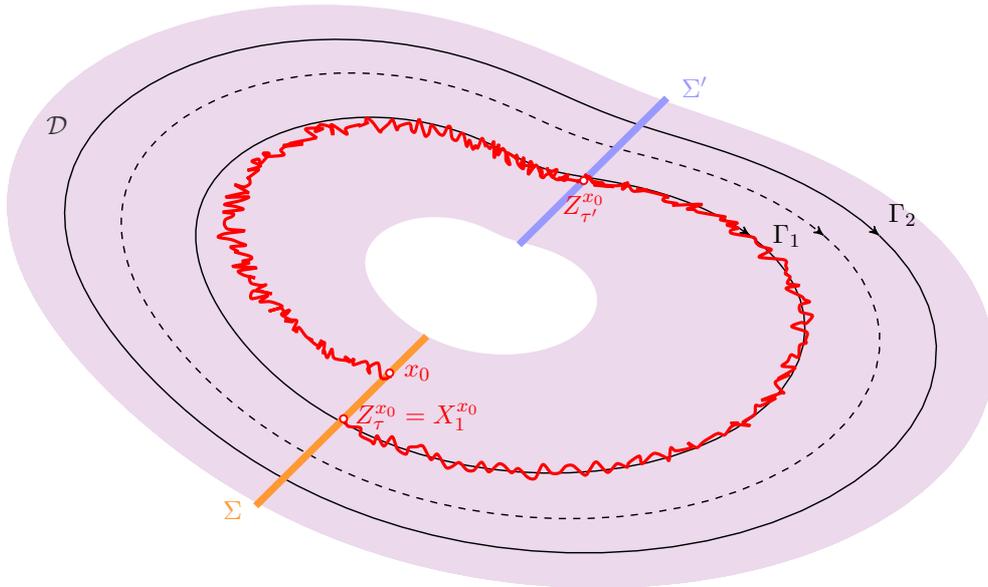
\begin{figure}
\begin{center}
\scalebox{0.9}{
\begin{tikzpicture}[>=stealth',main
node/.style={draw,circle,fill=white,minimum
size=3pt,inner sep=0pt},scale=0.9,
x=1.5cm,y=0.9cm,
declare function={
rho(\x) = 3*(1 + 0.4*sin(\x r) + 0.2*sin(2*\x r));
rhoext(\x)=  3*(1 + 0.4*sin(\x r) + 0.3*sin(2*\x r))+2;
rhoint(\x)=  1.15*(1 + 0.3*sin(\x r) + 0.2*sin(2*\x r));
}]
\draw[fill,violet!15,line width=2,thick,domain=0:10,samples=150] plot 
(xy polar cs:angle=-\x r,radius= {rhoext(\x)});
\draw[fill,white,line width=0,thick,domain=0:2*pi,samples=100] plot 
({rhoint(\x)*cos(\x r)},{-0.6-rhoint(\x)*sin(\x r)});
\draw[black,semithick,domain=0:2*pi,samples=100,->] plot (xy polar
cs:angle=-\x r,radius= {rho(\x)}) node[right,xshift=2mm] {$\Gamma_1$};
\draw[black,semithick,domain=0:2*pi,samples=100,->] plot (xy polar
cs:angle=-\x r,radius= {0.3*rho(\x)+0.7*rhoext(\x)}) node[above right]
{$\Gamma_2$};
\draw[black,dashed,semithick,domain=0:2*pi,samples=100,->] plot (xy
polar cs:angle=-\x r,radius= {0.6*rho(\x)+0.4*rhoext(\x)});
\draw[line width =3,orange!80]
(-2.35,{-2.35*5/3-1}) -- (-0.5,{-5/6-1});
\node[orange!80] at (-2.6,-5) {$\Sigma$};
\draw[line width =3,blue!40]
(2.1,{2.1*5/3-1}) -- (0.5,{5/6-1});
\node[blue!40] at (2.4,2.7) {$\Sigma'$};
\pgfmathsetseed{37}
\draw[red,line join=bevel,line width =1.3,domain=1.9:8.24,samples=330,smooth]
plot (xy polar cs:angle=-\x r+0.2*rand,radius= {(1-exp(-0.65*
\x))*rho(\x)+0.15*rand});
\node[main node,red,fill=white,thick] at (-1.4,-3.33333) {};
\node[red] at (-0.6,-3.3) {$Z_{\tau}^{x_0}=X_1^{x_0}$};
\node[main node,red,fill=white,thick] at (1.2,1) {};
\node[red] at (1.2,0.5) {$Z_{\tau'}^{x_0}$};
\node[main node,red,fill=white,thick] at (-0.9,-2.5) {};
\node[red] at (-0.6,-2.5) {$x_0$};
\node[black!80] at (-4.5,2) {$\cD$};
\end{tikzpicture}
}
\caption{Sketch of a random Poincaré map when the process starts in the basin of
attraction of a stable periodic orbit $\Gamma_1$.}
\label{figRandomPoincareMap}
\end{center}
\end{figure}
Consider first the case where variant A of Assumption~\ref{ass:confinement}
holds (for variant B, see Section~\ref{ssec:doob}). 
\deleted{Sentences removed}
%
Assume we have chosen a parametrisation of $\Sigma$ by a variable $x\in\R^d$. By
a slight abuse of notation, we will denote the domain of $x$ by $\Sigma$ as
well. For an initial condition $X_0\in\Sigma$, we would like to study the
sequence $(X_1, X_2, \dots)$ of successive intersections of the sample path
$(Z^{X_0}_t)_{t\geqs0}$ with $\Sigma$. We cannot proceed exactly as in the
deterministic case, because $\tau$ defined as in~\eqref{eq:first_return} is
equal to $0$ almost surely, due to the irregularity of Brownian paths. This
problem is cured quite easily, however. One can for instance introduce a second
manifold $\Sigma'\subset\cD$ which does not intersect $\Sigma$, such that
$\partial\Sigma'\subset\partial\cD$ and the vector field $f$ is transversal to 
$\Sigma'$ as well. Then \added{setting $\tau_{0}=0$,}
 it suffices to set for each $n\in \added{\N_0}$ 
\begin{align}
\nonumber
 \tau'_{n+1} &= \inf\setsuch{t> \added{ \tau_n} }{Z^{X_0}_t\in\Sigma'}\;, \\
 \tau_{n+1} &= \inf\setsuch{t>\tau'_{n+1}}{Z^{X_0}_t\in\Sigma}\;,
\end{align} 
and to define $X_{n+1}$ to be the $x$-coordinate of $Z^{\added{X_0}}_{\tau_{n+1}} \in
\Sigma$, see Figure \ref{figRandomPoincareMap}.

The strong Markov property implies that the law of $X_{n+1}$ given $X_n$ is
independent of $n$ and of all $X_m$ with $m<n$, that is, $(X_n)_{n\geqs0}$ forms
a time-homogeneous Markov chain. Since each
$X_n$ can be seen as the first-exit location from a bounded set, results
from~\cite{BenArous_Kusuoka_Stroock_1984} show that the law of $X_n$ has a
continuous density. We thus obtain a continuous-space Markov kernel $K$ with
continuous density $k$, defined by 
\begin{equation}
\label{eq:def_K} 
 K(x,A) = \bigprobin{x}{X_1 \in A} = \int_A k(x,y)\6y 
\end{equation} 
for any $x\in\Sigma$ and any Borel set $A\subset\Sigma$. We will denote $n$-fold
iterates of $K$ by 
\begin{equation}
 K^n(x,A) = \bigprobin{x}{X_n \in A} = \int_A k^n(x,y)\6y 
\end{equation} 
where the time-$n$ transition densities $k^n$ can be determined recursively by
the Chapman--Kolmogorov equation 
\begin{equation}
 k^{n+1}(x,y) = \int_\Sigma k^n(x,z) k(z,y) \6z\;.
\end{equation} 
The Markov kernel $K$ induces two Markov semigroups in the standard way.
Namely, with any bounded measurable test function $f:\Sigma\to\R$, we associate
the function 
\begin{equation}
 (Kf)(x) 
 = \int_\Sigma k(x,y) f(y) \6y
  = \bigexpecin{x}{f(X_1)}\;.
\end{equation} 
Furthermore, with any (signed) Borel measure $\mu$ on $\Sigma$ with density
$m$, we associate the measure 
\begin{equation}
\label{eq:semigroup_mu} 
 (\mu K)(\6y)  
 = \biggpar{\int_\Sigma m(x) k(x,y)\6x} \6y
 = \probin{\mu}{X_1\in\6y}\;. 
\end{equation} 
Since in what follows, all measures will have densities, we will often use the
same symbol for a measure and its density, and write $(mK)(y)$ for the integral
appearing on the right-hand side of~\eqref{eq:semigroup_mu}. 

\subsection{Spectral decomposition}
\label{ssec:setup-spectral} 

Since $\Sigma$ is bounded and $k$ is continuous, $K$ is a compact operator
(cf.~\cite[Section~VI.5]{Reed_Simon_I}), which implies that \added{the behaviour
of its large iterates} can be described by Fredholm theory. In particular, the
Riesz--Schauder theorem~\cite[Theorem~V1.15]{Reed_Simon_I} states that $K$ has
discrete spectrum, with all eigenvalues except possibly $0$ having finite
multiplicity. The eigenvalues are roots of the Fredholm determinant, first
introduced in~\cite{fredholm1903classe}, which is well-defined since $k$ is
bounded.

Let us denote by $(\lambda_j)_{j\in\N_0}$ the eigenvalues of $K$, ordered by
decreasing modulus, and by $\pi_j$ and $\phi_j$ the left and right
eigenfunctions respectively. That is, 
\begin{equation}
 (\pi_j K)(x) = \lambda_j \pi_j(x) 
 \qquad\text{and}\qquad
 (K\phi_j)(x) = \lambda_j \phi_j(x)
\end{equation} 
for all $j\in\N_0$. We can normalise the eigenfunctions in such a way that 
\begin{equation}
\label{eq:orthogonality} 
 \pi_i \phi_j := 
 \int_\Sigma \pi_i(x)\phi_j(x)\6x = \delta_{ij}\;.
\end{equation} 
In this way, the kernels $\phi_i(x)\pi_i(y)$ are projectors on invariant
subspaces of $K$. If the set of eigenfunctions is complete, 
and all nonzero eigenvalues have multiplicity $1$, then we have
the spectral decomposition 
\begin{equation}
\label{eq:spectral_dec} 
 k(x,y) = \sum_{i\geqs0} \lambda_i \phi_i(x)\pi_i(y)\;,
\end{equation} 
which entails the very useful property  
\begin{equation}
\label{eq:spectral_dec_n} 
 k^n(x,y) = \sum_{i\geqs0} \lambda_i^n \phi_i(x)\pi_i(y)\;.
\end{equation} 
A similar spectral decomposition holds if there are eigenvalues of multiplicity
higher than~$1$, except that there may be nontrivial Jordan blocks. In what
follows, we will show that $K$ is close to a finite-rank operator, defined by a
sum with $N$ terms. Therefore, the completeness of the set of all
eigenfunctions will not be an issue. 

Jentzsch's extension of the Perron--Frobenius theorem~\cite{Jentzsch1912} shows
that $\lambda_0$ 
is real, positive and
simple, and that the eigenfunctions
$\pi_0(x)$ and $\phi_0(x)$ can be taken real and positive. Since in our case,
$K$ is a stochastic Markov kernel (i.e.~$K(x,\Sigma)=1$ for all $x\in\Sigma$),
we have in fact $\lambda_0=1$
, while $\phi_0(x)$ can be taken identically equal
to $1$, and $\pi_0(x)$ is an invariant density, which
by~\eqref{eq:orthogonality} is normalised to be a probability density. Under the
spectral-gap condition $\abs{\lambda_1}<1$, the iterates $k^n(x,y)$ will
converge to
$\pi_0(y)$ for all $x$. 


\subsection{Process killed upon leaving a subset $A$}
\label{ssec:killed} 

Given a Borel set $A\subset\Sigma$, several processes related to
$(X_n)_{n\geqs0}$ will play an important r\^ole in what follows. The simplest
one is the process $(X^A_n)_{n\geqs0}$ killed
upon leaving $A$. Its kernel $K_A$ has density
\begin{equation}
 k_A(x,y) = k(x,y) \ind{x\in A, y\in A}\;.
\end{equation} 
This is in general a substochastic Markov process ($K_A(x,A) < 1$), which can be
turned into a stochastic Markov process by adding to $A$ an absorbing cemetery
state $\partial$. We denote its eigenvalues by $\lambda^A_i$ and its left and
right eigenfunctions by $\pi^A_i(x)$ and $\phi^A_i(x)$. The largest eigenvalue
$\lambda^A_0$ is still real, positive and simple, but in general smaller than
$1$. It is called the \emph{principal eigenvalue} of the process.
Following~\cite{BGK}, we call $\phi^A_0$ the (right) \emph{principal
eigenfunction} of $(X^A_n)_{n\geqs0}$. The normalised left
eigenfunction $\pi^A_0$ is called the \emph{quasistationary distribution}
(QSD) of the killed process. Under the spectral gap condition $\abs{\lambda^A_1}
< \lambda^A_0$, it satisfies 
\begin{equation}
 \pi^A_0(B) 
 = \lim_{n\to\infty} \bigpcondin{x}{X^A_n\in B}{X^A_n\in A}
 = \lim_{n\to\infty} \bigpcondin{x}{X_n\in B}{\tau^+_{A^c} > n}
\end{equation} 
for any Borel set $B\subset A$, independently of $x\in A$. 
Here $\tau^+_{A^c} = \inf\setsuch{n\geqs1}{X_n\not\in A}$ denotes the first-exit
time of the original process from $A$. 
A useful property of the QSD is that for the process $(X_n)_{n\geqs0}$ one has 
\begin{equation}
 \bigprobin{\pi^A_0}{\tau^+_{A^c} = n} = (\lambda^A_0)^{n-1} (1-\lambda^A_0)
 \qquad\forall n\geqs1\;,
\end{equation} 
that is, the first-exit time from $A$ has a geometric distribution with success
probability $(1-\lambda^A_0)$. In particular, we have 
\begin{equation}
 \bigexpecin{\pi^A_0}{\tau^+_{A^c}} = \frac{1}{1-\lambda^A_0}\;.
\end{equation} 

\begin{remark}
\added{
Note that if an eigenvalue $\lambda$ of the original process satisfies the
lower bound}
\begin{equation}
\added{\abs{\lambda}\geqs\supSur{x \in A }{\Prcx{x}{X_1 \in A}}\;,}
\end{equation}
\added{then the principal eigenvalue of the process killed upon leaving $A$ satisfies}
\begin{equation}
\added{\lambda^A_0 \leqslant \abs{\lambda}\; ,}
\end{equation}
\added{because $\lambda^A_0  = \Prcx{\pi^A_0 }{X_1 \in A} $.}
\end{remark}

\subsection{Trace of the process on $A$}
\label{ssec:trace} 

A second important process is the trace $(X_n)\vert_A$, which describes the
process monitored only while it visits $A$. This is still a Markov process, with
kernel 
\begin{equation}
 K\vert_A(x,\6y) = \bigprobin{x}{X_{\tau^+_A}\in\6y}\;,
\end{equation} 
where $\tau^+_A = \inf\setsuch{n\geqs1}{X_n\in A}$ denotes the first-return time
to $A$. The density of $\added{K\vert_A}$ is thus given by 
\begin{equation}
 k\vert_A(x,y) = \sum_{n\geqs1} \bigprobin{x}{\tau^+_A=n} k^n(x,y)\ind{x\in A,
y\in
A}\;.
\end{equation} 
If $R_{A^c}(1;z_1,\6z_2) = [\id - K_{A^c}]^{-1}(z_1,\6z_2)$ denotes the
resolvent at $1$ of the kernel killed upon leaving $A^c$, then
$\added{k\vert_A}$ can also
be written 
\begin{equation}
 k\vert_A(x,y) = \Bigbrak{k(x,y) + \int_{A^c}\int_{A^c} k(x,z_1)
R_{A^c}(1;z_1,\6z_2)k(z_2,y)\6z_1}\ind{x\in A, y\in A}\;.
\end{equation} 
One of the key points of our analysis will be to characterise the process
monitored only while visiting a neighbourhood of a well-chosen subset of
the stable periodic orbits.

\subsection{Process conditioned on staying in $A$}
\label{ssec:doob} 

The last important kernel describes the process $(\bar X^A_n)_{n\geqs0}$
conditioned on remaining in $A$ forever, and is defined by 
\begin{equation}
 \bar K_A(x,B) 
 = \lim_{n\to\infty}\bigpcondin{x}{X^A_1\in B}{X^A_n\in A}
 = \lim_{n\to\infty}\bigpcondin{x}{X_1\in B}{\tau^+_{A^c} > n}
\end{equation} 
for all Borel sets $B\subset A$. It can be constructed using Doob's
$h$-transform. 

\begin{proposition}[Doob $h$-transform]
\label{prop:Doob} 
Assume the spectral gap condition $\abs{\lambda^A_1} < \lambda^A_0$. Then the
density
of $\bar K_A$ is given by 
\begin{equation}
\label{eq:kAbar} 
 \bar k_A(x,y) = \frac{1}{\lambda^A_0} \frac{\phi^A_0(y)}{\phi^A_0(x)}
k_A(x,y)\;.
\end{equation} 
Furthermore, the eigenvalues and eigenfunctions of $\bar K_A$ are given by 
\begin{equation}
\label{eq:Doob_eigenfunctions} 
 \bar\lambda^A_n = \frac{\lambda^A_n}{\lambda^A_0}\;, 
 \qquad
 \bar\pi^A_n(x) = \pi^A_n(x)\phi^A_0(x) 
 \qquad\text{and}\qquad 
 \bar\phi^A_n(x) = \frac{\phi^A_n(x)}{\phi^A_0(x)}\;.
\end{equation} 
\end{proposition}

This is a standard result, which is closely related to what is called
\emph{ground state transformation} in quantum physics. For the reader's
convenience, we give a short proof in Appendix~\ref{app:Doob}.
Integrating~\eqref{eq:kAbar} over $y\in A$, we see immediately that $\bar K_A$
is a stochastic Markov kernel. Its principal eigenvalue $\bar\lambda^A_0$ is
indeed
equal to $1$, and the corresponding right eigenfunction is identically equal to
$1$. Proposition~\ref{prop:Doob} shows that the spectral properties of $K_A$ and
$\bar K_A$ determine one another, provided one knows the principal eigenvalue
$\lambda^A_0$ and the corresponding right eigenfunction $\phi^A_0$. 

We finally discuss the situation where variant B of
Assumption~\ref{ass:confinement} holds. In that case, we may consider the
process $Z_t^\cD$ killed upon leaving $\cD$. Proceeding exactly as above, we can
define a continuous-space Markov kernel $K_\cD$ describing the distribution of
first-hitting points of $\Sigma$ after visiting $\Sigma'$. Because of the
killing,  $K_\cD$ is a substochastic kernel. However, Doob's $h$-transform
allows us to define a stochastic kernel $\bar K_\cD$ of the process conditioned
on staying in $\cD$ forever. Proposition~\ref{prop:Doob} then allows us to
deduce spectral properties of $K_\cD$ from those of $\bar K_\cD$, provided we
manage to control the principal eigenvalue and right eigenfunction of $K_\cD$.

\section{Results}
\label{sec:results} 

We can now state the main results of this work. Throughout, we \added{require} 
Assumptions~\ref{ass:domain}, \ref{ass:limit_sets}, \ref{ass:ellipticity},
\ref{ass:hierarchy} and~\ref{ass:confinement} to hold. If variant A of
the confinement assumption~\ref{ass:confinement} holds, all results concern the
kernel $K$ defined in~\eqref{eq:def_K}. In case of variant B, they concern the
kernel $\bar K_\cD$ of the Doob-transformed process introduced just above. 

For $i=1,\dots,N$, we let $B_i \subset \Sigma$ be \added{the closure of a
neighbourhood of $x^\star_i$, contained in a ball centred in  
$x^\star_i$ and of radius $\delta>0$. Here $\delta$ is assumed} to be small
enough for each $B_i$ to be contained in the basin of attraction $\cA_i$ of
$x^\star_i$ (cf.~\eqref{eq:Kj}) \added{and such that the 
deterministic Poincar\'e map maps $B_i$ strictly into itself (such a ball
exists since the orbit $\Gamma_i$ is asymptotically stable)}.
 For $1\leqs k\leqs N$ we define the metastable
neighbourhood  
\begin{equation}
 \cM_k = \bigcup_{i=1}^k B_i\;.
\end{equation} 
For a Borel set $A\subset\Sigma$ we denote by $\tau_A =
\inf\setsuch{n\geqs0}{X_n\in A}$ its first-hitting time of $A$ and by $\tau^+_A
= \inf\setsuch{n\geqs1}{X_n\in A}$ the first-return time of the process to $A$.
If $A$ and $B$ are disjoint, \replaced{an}{and} important r\^ole will be played by the
\emph{committor functions}  
\begin{equation}
 \bigprobin{x}{\tau_A < \tau_B}
 \qquad\text{and}\qquad  
 \bigprobin{x}{\tau^+_A < \tau^+_B}
\end{equation} 
of hitting $A$ before $B$. Note that these functions are identical whenever 
$x\not\in A\cup B$, while $\probin{x}{\tau_A < \tau_B}$ has value $1$ in 
$A$ and $0$ in $B$. A rough bound on committor functions can be obtained by
large-deviation theory.

\begin{proposition}
\label{Prop:LargeDeviation}
For any $\eta>0$, there exists $\delta_0>0$ such that if $\delta < \delta_0$,
then for any $1\leqs i, j\leqs N$ with $i\neq j$, any \added{non-empty open set}
$A \subset \cA_j$ and all $x\in B_i$, one has 
\begin{equation}
 -H(i,j) - \eta \leqs 
 \lim_{\sigma\to0} \sigma^2 \log \bigprobin{x}{\tau^+_A<\tau^+_{B_i}}  
 \leqs - H(i, j) + \eta\;.
\end{equation}
\end{proposition}

This bound indicates that for $x\in B_i$ and $A\subset\cA_j$, the
committor $\smash{\probin{x}{\tau^+_A<\tau^+_{B_i}}}$ is exponentially small, of
order $\smash{\e^{-H(i,j)/\sigma^2}}$. We will see below that some committor
functions can be more precisely estimated in terms of certain principal
eigenfunctions. 

\subsection{Eigenvalue estimate}
\label{res:eigenvalues} 

Fix a small constant $\eta\in(0,\theta)$ and set $\theta^-=\theta-\eta$\added{, where $\theta$  is given by the metastable hierarchy assumption~\ref{ass:hierarchy}}. In all
results given below, it is always implicitly understood that there exists a
$\sigma_0>0$, depending on $\eta$, such that the claims hold for all $\sigma <
\sigma_0$. We will not repeat this condition in what follows. 

\begin{theorem}[Eigenvalue estimates]
\label{thm:eigenvalues} 
The $N$ largest eigenvalues of $K$ are real and positive and satisfy 
\begin{align}
\nonumber
 \lambda_0 &= 1\;, \\
 \lambda_k &= 1 - \bigprobin{\mathring{\pi}^{B_{k+1}}_0}{\tau^+_{\cM_k} <
\tau^+_{B_{k+1}}}
 \bigbrak{1 + \Order{\e^{-\theta_k/\sigma^2}}}
 && \text{for $1\leqs k\leqs N-1$\;,}
 \label{eq:lambdak} 
\end{align} 
where $\mathring{\pi}^{B_{k+1}}_0$ is a probability measure concentrated
on $B_{k+1}$ and $\theta_k = H(k+1,M_k)/2 - \eta$. Furthermore, there exists
$c>0$ such that 
\begin{equation}
 \abs{\lambda_k} \leqs \rho := 1 - \frac{c}{\log(\sigma^{-1})}
 \qquad \text{for all $k\geqs N$\;.} 
\label{eq:gap_lambdaN} 
\end{equation}  
Finally, the principal eigenvalue of the process killed upon hitting $\cM_k$
satisfies 
\begin{equation}
1 - \lambda_0^{\cM_k^c} = (1-\lambda_k)\bigbrak{1 +
\Order{\e^{-\theta_k/\sigma^2}}}
\label{eq:principal_eigenvalue} 
\end{equation} 
for $1\leqs k\leqs N-1$.
\end{theorem}

The probability measure $\mathring{\pi}^{B_{k+1}}_0$ has an explicit
definition: it is the QSD of the trace process
$(\smash{X^{B_{k+1}}_n})\vert_{\cM_{k+1}}$, monitored only while visiting
$\cM_{k+1}$ and killed upon hitting $\cM_k$ (which is equivalent to the
trace process leaving $B_{k+1}$). \added{Note that this process is not the same
as (the trace of) the process killed when leaving $B_{k+1}$, meaning that
taking the trace and killing do not commute.}
Regardless of the precise definition
of \added{the probability measure $\smash{\mathring{\pi}^{B_{k+1}}_0}$},
Proposition~\ref{Prop:LargeDeviation} shows that 
\begin{equation}
 \lambda_k = 1 - \Order{\e^{-(H(k+1,M_k)-\eta)/\sigma^2}}
 \qquad \added{\text{for $k=1,\dots,N-1$\;,}}
\end{equation} 
where $\eta$ can be chosen as small as one likes. 
The main interest of this estimate is that the spectral
decomposition~\eqref{eq:spectral_dec_n} becomes 
\begin{equation}
\label{eq:spectral_dec_N} 
 k^n(x,y) = \sum_{i=0}^{N-1} \lambda_i^n \phi_i(x)\pi_i(y)
 + \Order{\rho^n}\;,
\end{equation} 
which is dominated by the $N$ first terms as soon as $n \gg
\log(\added{{\rho}^{-1}})$. 
Since the $N$ first eigenvalues are exponentially close to $1$, the first $N$
terms of the sum decrease very slowly, highlighting the metastable behaviour of
the system. 

The proof of Theorem~\ref{thm:eigenvalues} relies on two main ingredients. In a
first step, we show that for each $k\leqs N-1$, the kernel of the process
monitored only while visiting $\cM_{k+1}$ can be described by a finite-rank
operator, given by a stochastic matrix $P$ with elements 
\begin{equation}
P_{ij}
= \probin{\mathring{\pi}^{B_i}_0}{X_{\tau_{\cM_{k+1}}} \in B_j}
= \probin{\mathring{\pi}^{B_i}_0}{\tau_{B_j} < \tau_{\cM_{k+1} \setminus B_j}} 
\;. 
\end{equation} 
In a second step we use the metastable hierarchy assumption to show that the
largest eigenvalue of $\id - P$ is close to the indicated committor functions.

If variant B of the confinement assumption~\eqref{ass:confinement} holds, then
the following result shows that Theorem~\ref{thm:eigenvalues} essentially holds
also for the process killed upon leaving $\cD$. \added{It can be seen as a
generalisation to the case $N>1$ of the result in~\cite{Wentzell_72_eigenvalue}
by Wentzell, which estimates the principal eigenvalue of the generator of a
diffusion killed upon leaving a domain containing a stable equilibrium point as
sole attractor.}

\begin{proposition}
The principal eigenvalue of the chain killed upon leaving $\cD$ satisfies 
\begin{equation}
 \lambda_0^\cD 
 = 1 - \probin{\mathring{\pi}^{B_1}_0}{\tau^+_\partial < \tau^+_{B_1}}
\bigbrak{1 +
\Order{\e^{-\theta_0/\sigma^2}}}
 = 1 - \Order{\e^{-\bar V(\partial\cD)/\sigma^2}}\;,
\end{equation} 
where $\partial$ denotes the cemetery state,
$\mathring{\pi}^{B_1}_0$ is a probability measure concentrated on $B_1$ and 
$\theta_0=\inf_{y\in\partial\cD}V(x^\star_1,y)/2-\eta$. 
\end{proposition}
\begin{proof}
The proof is the same as the proof of Theorem~\ref{thm:eigenvalues}, except
that one adds a fictitious ball $B_0$ with zero boundary conditions,
representing
the cemetery state $\partial$. 
\end{proof}
 
Indeed, when using Proposition~\ref{prop:Doob} to compute the eigenvalues of
the killed process, Condition~\eqref{eq:recurrence2} ensures that the
corrections to the eigenvalues $\bar \lambda^\cD_k$ are negligible for
$k=1,\dots,N-1$.

\subsection{Right eigenfunctions}
\label{res:right_eigenfunctions} 

For the spectral decomposition~\eqref{eq:spectral_dec_N} to be useful, it is
desirable to also have a control on the $N$ first right and left
eigenfunctions. We start by giving a result on the right
eigenfunctions~$\phi_k$, which is close in spirit to~\cite[Theorem~1.3]{BGK}. 

\begin{theorem}[Right eigenfunctions]
\label{thm:right_eigenfunctions} 
The $N$ first right eigenfunctions of $K$ can be taken real. They satisfy 
$\phi_0(x)=1$ for all $x\in\Sigma$, while for $k=1, \dots, N-1$,
\begin{equation}
\label{eq:right_eigenfunction} 
 \phi_k(x) = \bigprobin{x}{\tau_{B_{k+1}} < \tau_{\cM_k}}
 \bigbrak{1+\Order{\e^{-\theta^-/\sigma^2}}}
 + \Order{\e^{-\theta^-_k/\sigma^2}}
  \qquad \forall x\in\Sigma\;,
\end{equation} 
where $\theta^-_k = \min\set{\theta^-,\theta_k}$. 
Furthermore, the right principal eigenfunction of the process killed upon first
hitting $\cM_ k$ satisfies  
\begin{equation}
\label{eq:phi_0} 
 \phi_0^{\cM_k^c}(x) = \bigprobin{x}{\tau_{B_{k+1}} < \tau_{\cM_k}}
\bigbrak{1+\Order{\e^{-\theta^-/\sigma^2}}}+ \Order{\e^{-\theta^-_k/\sigma^2}}
 \qquad \forall x\in\cM_k^c
\end{equation} 
for $k=1, \dots, N-1$.
\end{theorem}

If $x$ is in the basin of attraction $\cA_i$ of $B_i$, then the committor
$\probin{x}{\tau_{B_i}<\tau_A}$ is exponentially close to $1$ whenever $A$ is
not in $\cA_i$. This shows that to leading order, 
\begin{itemize}
\setlength\itemsep{-1mm}
\item 	if $x\in\cA_i$ for $1\leqs i\leqs k$, then $\phi_k(x)$ is exponentially
small; 
\item 	if $x\in\cA_{k+1}$, then $\phi_k(x)$ is exponentially close to $1$; 
\item 	if $x\in\cA_j$ for $j>k+1$, then $\phi_k(x)$ is exponentially close to
$1$ if it is easier to reach $B_{k+1}$ than $\cM_k$ from $B_j$, and
exponentially small otherwise. 
\end{itemize}

In the case where variant B of Assumption~\ref{ass:confinement} holds, the
following result together with Proposition~\ref{prop:Doob} show that the same
expressions for eigenfunctions hold, except perhaps close to the boundary of
$\Sigma$. 

\begin{remark}
The proof actually yields a more precise estimate of the eigenfunctions, of the
form  
\begin{equation}
\label{eq:right_eigenfunction_iterated} 
 \phi_k(x) = \bigprobin{x}{\tau_{B_{k+1}} < \tau_{\cM_k}}
 \bigbrak{1+\Order{\e^{-\theta^-/\sigma^2}}}
 + \sum_{i=1}^k \bigprobin{x}{\tau_{B_i} < \tau_{\cM_{k+1}\setminus
 B_i}} \rho_{ki}
\end{equation} 
for $1 \leqs k \leqs N-1$, where 
\begin{equation}
 \rho_{ki} = 
 - \frac{\bigprobin{\mathring{\pi}^{B_i}_0}{\tau^+_{B_{k+1}} <
 \tau^+_{\cM_k}}}{\bigprobin{\mathring{\pi}^{B_{k+1}}_0}{\tau^+_{\cM_k} <
 \tau^+_{B_{k+1}}}}
 + \Order{\e^{-2\theta^-/\sigma^2}}
 = \Order{\e^{-\theta^-/\sigma^2}}\;.
\end{equation}
Higher-order expansions are also available. 
This expression may contain more information than~\eqref{eq:right_eigenfunction}
if the leading term in~\eqref{eq:right_eigenfunction} is exponentially small.
Note that at least some of the coefficients $\rho_{ki}$ are negative, which is
consistent with the orthogonality relation~\eqref{eq:orthogonality}. 
\end{remark}

\begin{proposition}
\label{prop:phi_0_D} 
The principal eigenfunction of the chain killed upon leaving $\cD$ satisfies 
\begin{equation}
 \phi_0^\cD(x) 
 = \probin{x}{\tau_{B_1} < \tau_\partial}
\bigbrak{1+\Order{\e^{-\theta^-/\sigma^2}}}
 \qquad \forall x\in\Sigma\;.
\end{equation} 
Thus $\phi_0^\cD(x) = 1-\Order{\e^{-\theta^-/\sigma^2}}$ whenever $x$ is
bounded away from $\partial\Sigma$. 
\end{proposition}

\subsection{Left eigenfunctions}
\label{res:left_eigenfunctions} 

If the kernel $K$ were reversible, that is, if $\pi_0(x)k(x,y) =
\pi_0(y)k(y,x)$ were true for any $x,y\in\Sigma$, then it would be immediate to
obtain the left eigenfunctions. Indeed, it is straightforward to check that they
would be given by $\pi_k(x) = \pi_0(x)\phi_k(x)$. Since we do not assume
reversibility, we have to find another way to determine the left
eigenfunctions. 

In~\cite{betz2016multi}, the authors obtained that first-return times of
finite-state space Markov chains satisfy the remarkable identity
$\pi_0(x) \probin{x}{\tau^+_y < \tau^+_x} = \pi_0(y) \probin{y}{\tau^+_x <
\tau^+_y}$, even if the chain is not reversible.
The following result shows that a similar property holds in our case.
The proof which, arguably, is even more elementary than the one given
in~\cite{betz2016multi}, is given in
Section~\ref{subSection:Computation_Eigenfunctions}.

\begin{proposition}
\label{prop:pi0} 
For any disjoint Borel sets $A_1, A_2 \subset \Sigma$ one has 
\begin{equation}
 \int_{A_1} \pi_0(x) \bigprobin{x}{\tau^+_{A_2} < \tau^+_{A_1}} \6x 
 = \int_{A_2} \pi_0(x) \bigprobin{x}{\tau^+_{A_1} < \tau^+_{A_2}} \6x\;.
\end{equation} 
The same relation holds when each $\tau^+_{A_i}$ is replaced by the
$n^{\text{th}}$ return time $\tau^{+,n}_{A_i}$ to $A_i$. 
\end{proposition}

Applying this result with $A_1=\cM_N$, $A_2=\Sigma\setminus\cM_N$, and using the
fact that the $\Gamma_i$ are the only attractive limit sets, we obtain that
$\pi_0$ is concentrated in $\cM_N$, in the sense that there exists $\kappa>0$,
depending on the size $\delta$ of the $B_i$, such that 
\begin{equation}
\label{eq:pi0_MN} 
 \frac{\pi_0(\Sigma\setminus\cM_N)}{\pi_0(\cM_N)}
 = \Order{\e^{-\kappa/\sigma^2}}\;.
\end{equation} 
Furthermore, for any compact $D_j$ such that $B_j\subset D_j \subset \cA_j$,
one has 
\begin{equation}
 \frac{\pi_0(D_j\setminus B_j)}{\pi_0(D_j)}
 = \Order{\e^{-\kappa/\sigma^2}}
\end{equation} 
where $\kappa>0$ may depend on $D_j$. Similar bounds hold for the QSDs
$\pi_0^{\cM_k^c}$ and the other left eigenfunctions. The essential information
is thus contained in the integrals of these measures over the sets $B_j$, which
are described by the following result. 

\begin{theorem}[Left eigenfunctions]
\label{thm:left_eigenfunctions} 
The invariant distribution satisfies 
\begin{equation}
\label{eq:left_ef_1} 
 \pi_0(B_1) = 1 - \Order{\e^{-\kappa/\sigma^2}}\;, 
 \qquad 
 \pi_0(B_j) = \Order{\e^{-\theta^-/\sigma^2}} 
 \quad \text{for $j = 2,\dots,N$\;.}
\end{equation} 
Similarly, the QSDs $\pi_0^{\cM_k^c}$ of the process killed upon first hitting
$\cM_k$ satisfy 
\begin{equation}
\label{eq:left_ef_2} 
 \pi_0^{\cM_k^c}(B_{k+1}) = 1 - \Order{\e^{-\kappa/\sigma^2}}\;, 
 \qquad 
 \pi_0^{\cM_k^c}(B_j) = \Order{\e^{-\kappa/\sigma^2}} 
 \quad \text{for $j = k+2,\dots,N$\;.}
\end{equation} 
Furthermore, the left eigenfunction $\pi_k$ satisfies 
\begin{equation}
\label{eq:left_ef_3} 
 \pi_k(B_j) = 
 \begin{cases}
  - \displaystyle 
  \frac{\bigprobin{\mathring{\pi}^{B_{k+1}}_0}{\tau^+_{B_j} <
 \tau^+_{\cM_{k+1}\setminus B_j}}}
 {\bigprobin{\mathring{\pi}^{B_{k+1}}_0}{\tau^+_ {\cM_k} <
 \tau^+_{B_{k+1}}}}
  \bigbrak{1 + \Order{\e^{-\theta^-/\sigma^2}}}
  + \Order{\e^{-\theta_k/\sigma^2}}
  & \text{for $1\leqs j \leqs k$\;,} \\
  \pi_0^{\cM_k^c}(B_j) \bigbrak{1 + \Order{\e^{-\theta^-/\sigma^2}}} 
  + \Order{\e^{-\theta_j/\sigma^2}}
  & \text{for $j \geqs k+1$\;.} 
 \end{cases}
\end{equation} 
\end{theorem}

This result shows in particular that
\begin{itemize}
\setlength\itemsep{-1mm}
\item 	$\pi_k(B_{k+1})$ is exponentially close to $1$;
\item 	if $k+1 < j \leqs N$, then $\pi_k(B_j)$ is exponentially small;
\item 	if $1\leqs j\leqs k$, then $\pi_k(B_j)$ is
negative, which is consistent with the orthogonality
relation~\eqref{eq:orthogonality}; it can be close to $-1$ or exponentially
small, depending on whether $B_j$ is the easiest ball in $\cM_k$ to reach from
$B_{k+1}$ or not. 
\end{itemize}

\added{In the case where variant B of Assumption~\ref{ass:confinement} holds,
combining Propositions~\ref{prop:phi_0_D} and~\ref{prop:Doob} it is immediate
to see that the conclusions of Theorem~\ref{thm:left_eigenfunctions} still hold
true. }
\begin{remark}
Using Proposition~\ref{prop:pi0}, either for the sets $B_1$ and $B_{k+1}$ or
for the sets $\cM_k$ and $B_{k+1}$, one can obtain more precise estimates for
the invariant distribution, namely the relations
\begin{align}
\nonumber
\pi_0(B_{k+1}) &=
 \frac{\bigprobin{\mathring{\pi}^{B_1}_0}{\tau^+_{B_{k+1}} <
 \tau^+_{B_1}}}{\bigprobin{\mathring{\pi}^{B_{k+1}}_0}{\tau^+_{B_1} <
 \tau^+_{B_{k+1}}}}
\bigbrak{1 + \Order{\e^{-\theta^-/\sigma^2}}}\;,  \\
\pi_0(B_{k+1}) &= \sum_{j=1}^k \pi_0(B_j) 
 \frac{\bigprobin{\mathring{\pi}^{B_j}_0}{\tau^+_{B_{k+1}} <
 \tau^+_{\cM_k}}}{\bigprobin{\mathring{\pi}^{B_{k+1}}_0}{\tau^+_{\cM_k} <
 \tau^+_{B_{k+1}}}}
\bigbrak{1 + \Order{\e^{-\theta^-/\sigma^2}}}
\end{align}
which hold for $1\leqs k\leqs N-1$. 
The second expression, while more complicated, has the merit of making it
obvious that $\pi_0(B_{k+1}) = \Order{\e^{-\theta^-/\sigma^2}}$, as a
consequence of Assumption~\ref{ass:hierarchy}. 

Similar expressions hold for the Doob-conditioned distributions 
$\bar\pi_0^{\cM_k^c}$, which immediately imply expressions for the QSDs via
Proposition~\ref{prop:Doob} and the expression~\eqref{eq:phi_0} of the right
principal eigenfunctions. 
\end{remark}

\subsection{Link between eigenvalues and expected return times}
\label{res:return_times} 

If the initial condition is distributed according to $\pi_0^{\cM_k^c}$, then it
follows directly from the properties of QSDs that $\tau_{\cM_k}$ has a
geometric distribution, with expectation 
\begin{equation}
 \bigexpecin{\pi_0^{\cM_k^c}}{\tau_{\cM_k}} 
 = \frac{1}{1 - \lambda_0^{\cM_k^c}}
 = \frac{1+\Order{\e^{-\theta_k/\sigma^2}}}{1 - \lambda_k}
 = \frac{1+\Order{\e^{-\theta_k/\sigma^2}}}
{\bigprobin{\mathring{\pi}^{B_{k+1}}_0}{ \tau^+_{\cM_k} <
 \tau^+_{B_{k+1}}}}\;.
\end{equation} 
Combining this fact with the bounds we obtained on the QSDs $\pi_0^{\cM_k^c}$,
it is not hard to obtain the following link between expected hitting times and
eigenvalues. 

\begin{theorem}[Expected hitting times]
\label{thm:expectations} 
There exists a constant $\kappa>0$, depending on the size $\delta$ of the $B_j$,
such that for every $k\in\set{1, \dots, N-1}$ one has, for any $x\in B_{k+1}$, 
\begin{equation}
  \bigexpecin{x}{\tau_{\cM_k}} 
 = \frac{1+\Order{\e^{-\kappa/\sigma^2}}}{1 - \lambda_k}
 = \frac{1+\Order{\e^{-\kappa/\sigma^2}}}
{\bigprobin{\mathring{\pi}^{B_{k+1}}_0}{\tau^+_{\cM_k} <
 \tau^+_{B_{k+1}}}}\;.
\end{equation} 
\end{theorem}


\subsection{Discussion of computational aspects}
\label{res:numerics} 

Our results provide sharp relations between eigenvalues and eigenfunctions of
the random Poincar\'e map, committor functions between and expected
first-hitting times of neighbourhoods of periodic orbits, and principal
eigenvalues, eigenfunctions and QSDs of processes killed when hitting these
sets. One limitation, compared to results in the reversible case, is that we do
not have sharp asymptotics for the prefactors of these quantities as in the case
of the Eyring--Kramers formula. However, some of them are accessible to
numerical methods. 

Computing eigenvalues and eigenfunctions of a continuous-space linear operator
by discretisation is possible, but costly, especially in high space dimension.
By contrast, principal eigenvalues, eigenfunctions and QSDs are much cheaper to
compute, since it is sufficient to simulate the process conditioned on
survival, starting with an arbitrary initial distribution.

There also exist powerful methods allowing to compute committor functions in
certain situations, such as adaptive multilevel splitting, see for instance
\cite{CerouGuyader07,AristoffLelievre_etal13,BrehierLelievreRousset15}. The fact
that the expressions~\eqref{eq:lambdak} for eigenvalues depend on committors
with respect to a QSD is not a problem, since we find that the spectral gap of
the associated process is \added{at least} logarithmically large in $\sigma$, so that
whatever the initial distribution, this QSD can be sampled in a relatively short
time.

 \section{Outline of the proof}
\label{sec:outline} 
As described in \Cref{ssec:setup-spectral}, in order to quantify transitions
between periodic orbits, our main objective is to solve the eigenvalue problem 
\begin{equation}\label{EqGenEigenvaluePb}
\pth{K \phi}\pth{x} = \e^{-u}  \phi\pth{x}
\end{equation}
for the discrete-time, continuous-state space kernel $K$.
We will start by exhibiting some general properties of this problem.

\subsection{Continuous-space, discrete-time Markov chains}

Let $\Sigma \subset \real^d$ be a bounded set equipped with the Borel
$\sigma$-algebra $\mathcal{B}\pth{\Sigma}$. Consider a \added{positive Harris
recurrent} discrete-time Markov chain $\pth{X_n}_{n\geqs 0}$ on the continuous
state space $\Sigma$ and let $K$ be the associated Markovian kernel having
density $k>0$ with respect to Lebesgue measure, i.e.,
\begin{equation}
K(x, \6y) = k(x, y) \6y \;.
\end{equation}
Given a Borel set $\setA \subset \Sigma$,  we introduce the first hitting time
and first return time 
\begin{align}
\nonumber
\tau_{\setA}(x) &= \inf\braces{n \geqslant 0, X_n \in {\setA}} \;,\\
\tau^+_{\setA}(x) &= \inf\braces{n \geqslant 1, X_n \in {\setA}} \;,
\end{align}
where $x$ denotes the initial condition. When the initial condition is clear
from the context, then we simply write $\tau_{\setA}$, $\tau^+_{\setA}$.
Note that $\tau^+_{\setA}\pth{x} = \tau_{\setA}\pth{x}$ for $x \in A^c= \Sigma
\bs {\setA}$, whereas $0 = \tau_{\setA}\pth{x} < 1 \leqs\tau^+_{\setA}\pth{x}$
if $x \in {\setA}$.
\added{If $A$ has positive Lebesgue measure}, due to the positive
\added{Harris} recurrence assumption on the Markov chain \added{and the fact
that $K$ has positive density}, the stopping times $\tau_{\setA}$ and
$\tau^+_{\setA}$ are almost surely finite.
To ease notation, we introduce
\begin{equation}
\Espc{A}{\cdot} = \supSur{x \in A}{\Espc{x}{ \cdot}}\;,\qquad \added{ \Prcx{A}{\cdot} = \supSur{x \in A}{\Prcx{x}{ \cdot}}}  
\end{equation}
 We also introduce the $n^{\text{th}}$ return
time defined inductively by
\begin{equation}
\tau^{+,n}_{\setA} = \inf\bigsetsuch{n >
\tau^{+,n-1}_{\setA}}{X_n \in {\setA}}\;,
\end{equation}
with $\tau^{+,1}_{\setA} = \tau^{+}_{\setA}$.

We recall the following result on existence of Laplace transforms, see
e.g.~\cite[Lemma~5.1]{berglund2014noise}.
\begin{lemma}
Consider a positive recurrent Markov chain with state space $\Sigma$. The
Laplace transform of the first hitting time $\bigexpecin{x}{\e^{u
\tau_{\setA}}}$ and the Laplace transform of the first return time
$\smash{\bigexpecin{x}{\e^{u \tau^+_{\setA}}}}$ are analytic in $u$ for $u$ such
that
\begin{equation}\label{EqLaplaceTransform}
\supSur{x \in A^c}{\Prcx{x}{X_1 \in A^c}}<
\abs{\e^{-u}} \;.
\end{equation}
\end{lemma}

Following ideas from the potential-theoretic approach to
metastability~\cite{BEGK,BGK}, we are going to study a Dirichlet
boundary value problem to solve the eigenvalue problem. Given a set
$\setA\subset\Sigma$, $u \in \complex$
and a \added{bounded} measurable function $\overline{\phi} : A \rightarrow \real$, we want
to find a (bounded) function ${\phi}^u$ which satisfies
\begin{align}
\nonumber
 \pth{ K {\phi}^{u} } \pth{x}  &= \e^{-u} {\phi}^{u} (x), & x& \in \setA^c\;,\\
{\phi}^{u} \pth{x} &= \overline{\phi} \pth{ x }, &   x &\in {\setA}\;.
\label{EqDirichletBoundaryValueProblem}
\end{align}
Solutions of such a Dirichlet problem admit a probabilistic representation 
in terms of Laplace transforms.

\begin{proposition}[Feynman--Kac type relation]
\label{prop:Feynman-Kac} 
For $u$ such that \Cref{EqLaplaceTransform} is satisfied, the unique solution of
the Dirichlet boundary value problem \Cref{EqDirichletBoundaryValueProblem}
is given by 
\begin{equation}
{\phi}^{u}\pth{x} = \bigexpecin{x}{\e^{u \tau_{\setA}} 
\overline{\phi}\pth{X_{\tau_{\setA}}}}\;.
\end{equation}
\end{proposition}
\begin{proof}
First, let us check that the proposed function solves the boundary problem.
This is obvious for $x \in {\setA}$, since in that case $\tau_A =0$, so that
$\bigexpecin{x}{\e^{u \tau_{\setA}} 
\overline{\phi}\pth{X_{\tau_{\setA}}}} = \overline{\phi}\pth{x}$.
For $x \in A^c$, splitting the expectation defining $ \pth{ K
{\phi}^{u} } \pth{x}$ according to the location of $X_1$, we get
\begin{align}
\nonumber
 \pth{ K {\phi}^{u} } \pth{x}  &= \bigexpecin{x}{ \bigexpecin{X_1}{\e^{u
\tau_{\setA}} \overline{\phi}\pth{X_{\tau_{\setA}}}} \ind{X_1 \in
{\setA} }} +  \bigexpecin{x}{ \bigexpecin{X_1}{\e^{u \tau_{\setA}}
\overline{\phi}\pth{X_{\tau_{\setA}}}} \ind{X_1 \in  A^c
}}\\
\nonumber
& = \bigexpecin{x}{ \overline{\phi}\pth{X_{1}} \ind{X_1 \in {\setA} }} + 
\bigexpecin{x}{ \e^{u \pth{\tau_{\setA} - 1}}
\overline{\phi}\pth{X_{\tau_{\setA}}} \ind{X_1 \in  A^c
}}\\
 &=\e^{-u} \bigexpecin{x}{\e^{u \tau_{\setA}} 
\overline{\phi}\pth{X_{\tau_{\setA}}}}\;.
 \end{align}
 This shows that $\bigexpecin{x}{\e^{u \tau_{{\setA}}} 
\overline{\phi}\pth{X_{\tau_{{\setA}}}}}$ is an admissible solution for all
$x \in \Sigma$.
\added{Uniqueness follows from the Fredholm alternative. Let us assume
by contradiction that two functions $f$ and $g$ solve the Dirichlet boundary
value problem with $f \neq g$. Then} 
\begin{align}
\nonumber
\added{ \pth{ (\id - \e^u K ) \pth{f-g} } \pth{x}  }&\added{{}= 0\;,} 
&\added{x}& \added{{}\in \setA^c\;,}\\
\added{(f-g)(x)} &\added{{}= 0\;,} 
&   \added{x}&\added{{}\in {\setA}}\;.
\label{EqKindDirichletBoundaryValueProblem} 
\end{align}
\added{The contradiction comes from the fact that under Condition
\eqref{EqLaplaceTransform}, $ \norm{\e^{ u}{ K_{A^c} } } < 1 $, so that we can 
apply \cite[Theorem 8.1]{gohberg2013classes}. In particular $\pth{\id - \e^u
K_{A^c}}$  is invertible and $f  \equiv g$.}
\end{proof}

\deleted{A remark has been removed.}

The solution of the boundary value
problem~\eqref{EqDirichletBoundaryValueProblem} allows us to define a (non
Markov) kernel on $\setA$.

\begin{corollary}\label{corEquivalenceEigenvaluePb}
Let $\Ku$ be the kernel defined on $\added{A \times\mathcal{B}(A)}$ by
\begin{equation}
\Ku\pth{x, \6y} = \biggexpecin{x}{\e^{u\pth{ \tau^+_{\setA}-1}}
\bigind{X_{\tau^+_{\setA}}  \in \6y}}\;.
\end{equation}
For $u$ verifying \Cref{EqLaplaceTransform}, the eigenvalue problem on $\Sigma$
\begin{equation}
\pth{K {\phi}^u}\pth{x} = \e^{-u}  {\phi}^u\pth{x}
\end{equation}
 is equivalent to the eigenvalue problem on $\setA$ given by
\begin{equation}
\label{eq:Kuphibar}  
 \bigpar{\Ku \overline{\phi}^{u}}\pth{x} = \e^{-u}  \overline{\phi}^{u}\pth{x}
\end{equation}
where $\overline{\phi}^{u}\pth{x} = {\phi}^{u}\pth{x}$ for all $x \in \setA$.
\end{corollary}
\begin{proof}
Let $\pth{\e^{-u}, {\phi}^{u}}$ be a couple of eigenvalue, eigenfunction for the
Markov kernel $K$. Then splitting the integral equation according to $X_1$ and
inserting the previous solution in the second term of the right-hand side, we
have
\begin{align}
\nonumber
\e^{-u} {\phi}^{u}\pth{x}=\pth{K {\phi}^u}\pth{x} 
&= \bigexpecin{x}{{\phi}^u\pth{X_1}
\ind{X_1 \in \setA}} +  \bigexpecin{x}{{\phi}^u\pth{X_1} \ind{X_1 \in A^c}}\\
\nonumber
&= \bigexpecin{x}{ {\phi}^u(X_{\tau^+_{\setA}}) \normalind{ \tau^+_{\setA}= 1 }}
+ \bigexpecin{x}{\Espc{X_1}{\e^{u \tau_{{\setA}}} 
{\phi}^u\pth{X_{\tau_{{\setA}}}}} \normalind{\tau^+_{\setA} > 1}}\\
\nonumber
&=  \bigexpecin{x}{\e^{u(\tau^+_{\setA}-1)} 
{\phi}^u(X_{\tau^+_{\setA}}) \added{\normalind{ \tau^+_{\setA}= 1 } }} +
\bigexpecin{x}{\e^{u\pth{
\tau^+_{\setA}-1}}  {\phi}^u(X_{\tau^+_{\setA}}) \normalind{\tau^+_{\setA} >
1}}\\
&= \bigexpecin{x}{\e^{u(\tau^+_{\setA}-1)}  {\phi}^u(X_{\tau^+_{\setA}}) }
= \pth{\Ku {\phi}^u}\pth{x}\;.
\end{align}
Since for $x \in A$, we have ${\phi}^u \pth{x} = 
\overline{\phi}^{u}\pth{x}$, this proves that~\eqref{eq:Kuphibar} holds.

On the other hand, if we know a couple $({\e^{-u},\overline{\phi}^{u}})$ of
eigenvalue and eigenfunction for the kernel $\Ku$, we introduce the
function
\begin{equation}
{\phi}^u\pth{x} = \Espc{x}{\e^{u \tau_{\setA} }
\overline{\phi}^u\pth{X_{\tau_{\setA}}}}\;.
\end{equation}
Note that ${\phi}^u\pth{x} = \overline{\phi}^u\pth{x}$ for $x \in \setA$.
By the previous proposition, ${\phi}^u$ satisfies the eigenvalue equation
with eigenvalue $\e^{-u}$.
\end{proof}
In the sequel, we will forget the notation $\overline{\phi}$ since $
\overline{\phi}^{u} = \restriction{ {\phi}^u}{x \in \setA} $.

\subsection{Choice of the set to reduce the eigenvalue problem}
\label{subSection:ChoiceSetA}

Thanks to \Cref{corEquivalenceEigenvaluePb}, we have reduced the eigenvalue
problem on $\Sigma$ to an eigenvalue problem on a subset $\setA$ of $\Sigma$,
which has yet to be defined. We now discuss the choice of $\setA$. Under the
metastable hierarchy assumption~\ref{ass:hierarchy}, we expect that there will
be $N$ eigenvalues exponentially close to one, and a gap between the
$N^{\text{th}}$ eigenvalue and the remaining part of the spectrum (we recall
that eigenvalues are ordered by decreasing modulus).

The general idea of the proof is to first choose a set $\setA$ which is well
suited to estimating $\lambda_{N-1}$, the ${N}^{\text{th}}$ eigenvalue of the
kernel, and also to obtain a rough estimate of the $N-1$ largest eigenvalues.
Then we take another set $\setA$ in order to estimate $\lambda_{N-2}$, and
obtain a rough estimate of the $N-2$ largest eigenvalues, and so on up to
$\lambda_{1}$. The way to estimate one of the ${N}$ largest eigenvalues is
based  on approximations of the kernel $\Ku$ and is explained is the next
subsection.

To estimate the $N^{\text{th}}$ eigenvalue,  we are going to choose $\setA :=
\meta{N}$. Note that intuitively for such a choice of set,
\Cref{EqLaplaceTransform} is not restrictive. Indeed, due to the  attraction of
$\meta{N}$, starting outside the union of the neighbourhoods of the stable
periodic orbits, the probability that the first return point is still outside
this neighbourhood should be very small (possibly replacing $K$ by a suitable
iterate $K^m$). So this should allow us to estimate the
$N^{\text{th}}$ eigenvalue of the kernel $\Ku$. 

Next, to estimate the $\pth{N-1}^{\text{st}}$ eigenvalue $\lambda_{N-2}$, we
will study the eigenvalue problem on $\setA := \meta{N-1}$. It follows that for
all $u$ such that
\begin{equation}
\underset{x \in \meta{N-1}^c}{\sup} \Prcx{x}{X_1 \in 
\meta{N-1}^c} <\abs{\e^{-u}}\; ,
\end{equation} 
the original eigenvalue problem \Cref{EqGenEigenvaluePb} is equivalent to an
eigenvalue problem on $\meta{N-1}$.
Note that the Laplace transform conditions given by \Cref{EqLaplaceTransform}
satisfies
\begin{equation}
\underset{x \in \meta{N}^c}{\sup} \Prcx{x}{X_1 \in  
\meta{N}^c} < \underset{x \in \meta{N-1}^c}{\sup} \Prcx{x}{X_1 \in
\meta{N-1}^c}\;.
\end{equation}
Therefore when solving the eigenvalue equation defined on $\cM_{N-1}$ with
kernel $K^{u,\pth{N-1}}$ we find eigenvalues which are greater in modulus  than
the eigenvalues of the kernel $K^{u,\pth{N}}$ defined on $\cM_{N}$.
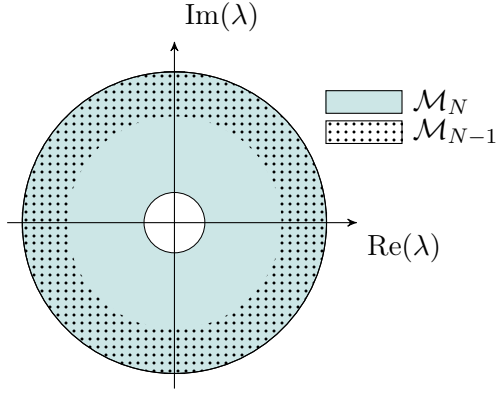
\begin{figure}
\centering
\begin{tikzpicture}[scale=2,>=stealth']
\draw [fill=teal!20] (0,0.) circle (1cm);
\draw [pattern=dots] (0,0.) circle (1cm);
\draw [fill=teal!20, color=teal!20] (0,0.) circle (0.7cm);
\draw [line width=0.1pt,color=black,fill=white] (0,0) circle (0.2cm);
 \node[draw=black,fill=teal!20,shape=rectangle,
       minimum width=1cm,minimum height=0.3cm,anchor=east](mn) at (1.5,0.8){};
 \node[anchor=west](tmn) at (1.5,0.8){$\cM_{N}$};
  \node[draw=black,pattern=dots,shape=rectangle,
       minimum width=1cm,minimum height=0.3cm,anchor=east](mn) at (1.5,0.6){};
 \node[anchor=west](tmn) at (1.5,0.6){$\cM_{N-1}$};
\draw[->,color=black] (-1.1,0.) -- (1.2,0.) node[below
right]{$\text{Re}\pth{\lambda}$};
\draw[->,color=black] (0.,-1.1) -- (0.,1.2) node[above
right]{$\text{Im}\pth{\lambda}$};
\draw[shift={(-1,0)},color=black] (0pt,2pt) -- (0pt,-2pt);
\draw[shift={(1,0)},color=black] (0pt,2pt) -- (0pt,-2pt);
\draw[shift={(0,1)},color=black] (2pt,0pt) -- (-2pt,0pt);
\draw[shift={(0,-1)},color=black] (2pt,0pt) -- (-2pt,0pt);
\end{tikzpicture}
\caption{Set of possible eigenvalues $\lambda$ for the eigenvalue problems
defined on $\cM_{N}$ and $\cM_{N-1}$.}
\end{figure}

In general, to estimate the $k^{\text{th}}$ eigenvalue, we will study
the eigenvalue problem \Cref{EqDirichletBoundaryValueProblem} with $\setA =
\cM_{k+1}$. Therefore we have to study the spectral properties of the kernel
$K^{u,\pth{k+1}}$ defined on $\meta{k+1}$ by
\begin{equation}
K^{u,\pth{k+1}}\pth{x,\6y}  = \biggexpecin{x}{\e^{u(\tau^+_{\cM_{k+1}}-1)}
\bigind{X_{\tau^+_{\cM_{k+1} }} \in \6y} }\;.
\end{equation}
To ease notation, we will simply write $K^{u}$ and keep in mind that the kernel
depends on $k$ through its domain of definition.

\subsection{Eigenvalue problem on a union of metastable sets}
\label{subSection:EigenvaluePbUnionMetastable}

Reducing our eigenvalue problem on $\cM_{k+1}$ is convenient because the
kernel is defined where we expect to have information from the deterministic
part of the system due to the attraction of the stable periodic orbits. However,
the introduced kernel does not have a nice probabilistic interpretation since it
depends on the spectral parameter $u$. 
To circumvent this problem, we are going to introduce a new parameter $\varv$
and solve the system of two coupled equations
\begin{align}
\nonumber
\pth{\Ku \phi^u}\pth{x} & = \varv \phi^u\pth{x} \\
\varv & = \e^{-u}\;.
\label{eq:coupled} 
\end{align}
%
In addition, instead of studying the kernel $\Ku$ (or its iterate
$\pth{\Ku}^m$), we are going to approximate it by a kernel having a nicer
probabilistic interpretation and for which we can easily  obtain the spectrum.
The justification for using such an approximation is given by the continuity of
eigenvalues of bounded linear operators \cite{gohberg2013classes}. Indeed, let
$K^{\star}: Y \rightarrow Y$ be a bounded linear operator acting on the Banach
space $Y$. The following classical theorem describes what happens to parts of
the spectrum $\uGs\pth{K^{\star}}$ if the operator $K^{\star}$ is
subjected to a small perturbation.

\begin{proposition}[{\cite[Proposition 4.2]{gohberg2013classes}}]
\label{propContinuityEigenvalue}
Let 
$\uGs$ be a finite set of eigenvalues of finite type of $K^{\star}$, and let
$\cC$ be a contour around $\uGs$ which separates $\uGs$ from
$\uGs\pth{K^{\star}} \bs \uGs$. Then there exists $\epsilon>0$ such
that for any operator $\Ku$ on $X$ with $\norm{K^{\star}-\Ku} < \epsilon$ the
following holds true: $\uGs\pth{\Ku} \cap \cC =\emptyset$,
the part of $\uGs\pth{\Ku}$ inside $\cC$ is a finite set of eigenvalues of
finite type and, if we denote $m\pth{\lambda ; K}$ the algebraic multiplicity of
the eigenvalue $\lambda$ for the operator $K$, then
\begin{equation}
\underset{\lambda \textrm{ inside } \cC}{\sum} m\pth{\lambda ; K^u} = 
\underset{\lambda \textrm{ inside } \cC}{\sum} m\pth{\lambda ; K^{\star}}\;.
\end{equation}
\end{proposition}

The theory also provides bounds on $\epsilon$, cf.\
Section~\ref{sec:operators} for details. 

Two approximations are going to be made. Firstly, because we are looking for
eigenvalues of $\Ku$ that are close to $1$, i.e.\ $u$ close to zero, we can
compare the kernels $\Ku$ and $K^{0}$. Note that $K^0$ is a Markov kernel
defined  on $\cM_{k+1} \times \mathcal{B}\pth{\cM_{k+1}}$ and given by
\begin{equation}
K^{0}\pth{x, \6y} = \Bigprobin{x}{X_{\tau^+_{\cM_{k+1}} }  \in \6y} \; .
\end{equation}
This is exactly the kernel of the trace process $X\vert_{\cM_{k+1}}$ introduced
in Section~\ref{ssec:trace}.  
For the second approximation, we introduce the kernel $K^{\star}$ given by
\begin{equation}
K^{\star}\pth{x, \6y} =  \displaystyle \sum_{i=1}^{k+1} \ind{x \in B_i} 
\int_{B_i} \mathring{\pi}^{B_i}_0\pth{z} \ K^0\pth{z, \6y} \6z
\end{equation}
where $\mathring{\pi}^{B_i}_0$ is the quasistationary distribution of the
process described by the kernel $K^0_{B_i}$  (see \Cref{ssec:killed}). Note that
the kernel $K^{\star}$ is of finite rank, since it is a finite sum of products
of two functions, one of which depends on its first argument only. 

In order to obtain sharper estimates, instead of considering the Markov chain at
each return time to the set $\cM_{k+1}$ on the Poincaré map, we will study the
diluted chain given by the $m^{\text{th}}$ iterate of the kernel, where $m$ may
depend on $\sigma$. It is clear that $\e^{-u}$ is an eigenvalue of $\Ku$ if and
only if $\e^{-u \mm}$ is an eigenvalue of the kernel $\npar{\Ku}^{m}$. We also
introduce the $m$-fold iterates $\npar{K^{0}}^m$ and $\npar{K^{\star}}^m$.

We will prove in \Cref{Ssection:KUKZ} the following bound on the norm of the
difference between $\npar{K^{u}}^m$ and $\npar{K^{0}}^m$.

\begin{proposition}[Proposition~\ref{Prop:NormKuKZIteratedk}]
\label{Prop:NormKuKZIterated}
For all real  $u$ verifying the Laplace condition given by
\eqref{EqLaplaceTransform} \added{with $A=
\cM_{k+1}$}, and such that $\pth{1-\e^{-u}}
\bigexpecin{\cM_{k+1}^c}{\tau^+_{\cM_{k+1}} }<1$, we have
\begin{equation}
\norm{\pth{\Ku}^m  - \pth{K^{0}}^m} \leqslant \Biggpar{ 1 +
\frac{\pth{1-\e^{-u}} \bigexpecin{\cM_{k+1}}{\tau^+_{\cM_{k+1}}-1} }{ 1-
\pth{1-\e^{-u}} \bigexpecin{\cM_{k+1}^c}{ \tau^+_{\cM_{k+1}} }}}^m  -1 \; .
\label{Eq:NormKUKZIterated}
\end{equation}
%
\end{proposition}

The expectations appearing in this bound will be estimated in
Section~\ref{sec:sample_paths}. We will also obtain in
Section~\ref{ssec:K0Kstar} the following bound on the norm of the difference
between $\npar{K^{0}}^m$ and $\npar{K^{\star}}^m$.

\begin{proposition}[Proposition~\ref{Prop:NormK0KStark}]
\label{Prop:NormK0KStar}
For all $m\in\N$, the norm of the difference between the 
 iterates of $K^0$ and $K^{\star}$ satisfies the bound 
\begin{equation}
\norm{\pth{K^{0}}^m  - \pth{K^{\star}}^m} \leqslant  \supSur{1 \leqslant i
\leqslant {k+1}} {R_i}\;,
\end{equation}
where 
\begin{align}
\nonumber
R_i ={} &\norm{\mathring{\phi}^{B_i}_0-1} +2 (\mathring{\ev}^{B_i}_1)^m + 2
\frac{1- (\mathring{\ev}^{B_i}_1)^m }{1-\mathring{\ev}^{B_i}_1}
\Prcx{B_{i}}{\tau^+_{\cM_{k+1} \bs  B_i}  < \tau^+_{B_i}} \\
&{}+ m \pth{m-1} \Prcx{B_{i}}{\tau^+_{\cM_{k+1} \bs B_i}  < \tau^+_{B_i}}
\Prcx{\cM_{k+1} \bs B_i}{\tau^+_{B_i} <\tau^+_{\cM_{k+1} \bs B_i  }  }\; .
\end{align}
\end{proposition}

The quantities $\mathring{\ev}^{B_i}_k$ and $\mathring{\phi}^{B_i}_0$ appearing
in this estimate are related to the trace process $K^0_{B_i}$ killed upon
leaving $B_i$. In Section~\ref{sec:KZBi}, we will derive bounds on the
oscillation of the principal eigenfunction $\smash{\mathring{\phi}^{B_i}_0}$ and
the spectral gap $\smash{|\mathring{\ev}^{B_i}_1|/\mathring{\ev}^{B_i}_0}$.
Together with the metastable hierarchy assumption, these bounds show that
$\norm{\pth{K^{0}}^m - \pth{K^{\star}}^m}$ is small for all $m$ such that
\begin{equation}
m \Prcx{B_{k+1}}{\tau^+_{\cM_{k}} < \tau^+_{B_{k+1}}} \ll 1\; .
\end{equation}
The difference $\pth{\Ku}^m  - \pth{K^{0}}^m$ is also small under this
condition. 

Thanks to these approximations, we have reduced our eigenvalue problem to a much
simpler one. Since the kernel $\npar{K^\star}^m$ is of finite rank, it admits
exactly $N$ eigenvalues.  Furthermore, solving the eigenvalue problem for
$\npar{K^\star}^m$ is now equivalent to solving a system of linear algebraic
equations. 

\begin{proposition}[\Cref{CorEigenvalueProp}]
For $ 0 \leqslant i \leqslant k$, we denote by ${\ev}^{\star}_i$ the
eigenvalues of $K^{\star}$ labelled by decreasing order.  The smallest
eigenvalue ${\ev}^{\star}_k$ of $K^{\star}$ is real and simple. It
satisfies 
\begin{equation}
\label{eq:lambdastark} 
\abs{{{\ev}^{\star}_k} - \pth{1 - \Bigprobin{
\mathring{\pi}_0^{B_{k+1}} }{X_{\tau^{+}_{\cM_{k+1}}}  \in \cM_{k}}}} 
\leqslant 2 \maxSur{1\leqslant l \leqslant k}{ \Bigprobin{
\mathring{\pi}_0^{B_l} }{X_{\tau^{+}_{\cM_{k+1}}}  \in \cM_{k+1} \bs
B_l} } \;.
\end{equation}
The $k$ remaining eigenvalues satisfy for all $0 \leqslant i < k$
\begin{equation}
\abs{1 - {\ev}^{\star}_{i}} \leqslant  4 \maxSur{1\leqslant l
\leqslant k}{ \Bigprobin{ \mathring{\pi}_0^{B_l}
}{X_{\tau^{+}_{\cM_{k+1}}} \in \cM_{k+1} \bs B_l}} \;.
\end{equation}
\end{proposition}

Theorem~\ref{thm:eigenvalues} then follows essentially by combining the
estimate~\eqref{eq:lambdastark} with the bound on $\norm{\pth{\Ku}^m  -
\pth{K^\star}^m}$ implied by the two previous propositions, for an appropriate
choice of $m$. Details are given in Section~\ref{sec:lastSteps}.  

\subsection{Computation of eigenfunctions}
\label{subSection:Computation_Eigenfunctions}

Once eigenvalues have been estimated, determining the associated left and right
eigenfunctions is relatively easy. The expressions for right eigenfunctions
$\phi_k$ are essentially consequences of the Feynman--Kac representation given
in Proposition~\ref{prop:Feynman-Kac}. As for the left eigenfunctions, a
crucial tool is the following result.  

\begin{lemma}
\label{lem:left_eigenfunction} 
For any left eigenfunction $\pi_k$ \added{of the kernel $K$} associated to the eigenvalue $e^{-u_k}$, and
for any $B\subset A \subset \Sigma$, we have 
\begin{equation}
 \int_A \pi_k(x) K^{u_k}(x,B) \6x  
:= \int_A \pi_k(x)  \bigexpecin{x}{ \e^{u_k (\tau^+_{A}-1)} \normalind{
\tau^+_{B} < \tau^+_{A \bs B} }} \6x = \e^{-u_k}\pi_k(B)\;.
\label{Eq:LeftEigenvectorCommitork}
\end{equation}
\end{lemma}
\begin{proof}
Consider the function 
$ h^u(x) = \bigexpecin{x}{\e^{u\tau_A} \normalind{\tau_B < \tau_{A\bs B}}}$. 
Note that $h^u(x) = \ind{x\in B}$ whenever $x\in A$, while a similar
argument as in Proposition~\ref{prop:Feynman-Kac} yields 
\begin{equation}
 (K h^u)(x) 
 = \bigexpecin{x}{h^u(X_1)}
 = K^u(x,B)\;.
\end{equation} 
It follows that 
\begin{align}
\nonumber
 \int_A \pi_k(x) K^{u_k}(x,B) \6x 
 &= \int_\Sigma \pi_k(x) (Kh^{u_k})(x)\6x 
 - \int_{\Sigma\bs A} \pi_k(x) K^{u_k}(x,B) \6x \\
\nonumber
 &= \e^{-u_k} \int_\Sigma \pi_k(x) h^{u_k}(x) \6x
 - \e^{-u_k} \int_{\Sigma\bs A} \pi_k(x) h^{u_k}(x) \6x \\
 &= \e^{-u_k} \int_A \pi_k(x) \ind{x\in B} \6x
 = \e^{-u_k} \pi_k(B)\;.
\end{align}
In the second line, we have used the eigenvalue equation $\pi_k K =
\e^{-u_k}\pi_k$ and the fact that in $\Sigma \bs A$, $\tau_A=\tau^+_A$ and
$\tau_B=\tau^+_B$, and thus $\Ku(x,B)=\e^{-u}h^u(x)$. 
\end{proof}

\begin{proof}[{\sc Proof of Proposition~\ref{prop:pi0}}]
Applying \eqref{Eq:LeftEigenvectorCommitork} for the left eigenfunction $\pi_0$
associated to the eigenvalue $1$ and with any disjoint $A_1, A_2$ such that $A_1
\cup A_2 = A$, we have
\begin{equation}
\pi_0\pth{A_1} = \int_{A_1 \cup A_2} \pi_0\pth{x}
\bigprobin{x}{\tau^+_{A_1} < \tau^+_{A_2}} \6x.
\end{equation}
Decomposing the domain of the integral into $A_1$ and $A_2$, and using the fact
that for all $x$, $\bigprobin{x}{\tau^+_{A_1} < \tau^+_{A_2}} =
1-\bigprobin{x}{\tau^+_{A_2} < \tau^+_{A_1}}$, we immediately get the result.
\end{proof}

\section{Spectral properties of $K^0_{B_i}$}
\label{sec:KZBi}
 Recall that we have denoted $K^{\star}$ the kernel on $\cM_{k}
\times \mathcal{B}\pth{\cM_{k}}$ defined by
\begin{equation}
K^{\star}\pth{x, \6y} =  \displaystyle \sum_{i=1}^{k}  \ind{x \in B_i} 
\int_{B_i} \mathring{\pi}^{B_i}_0\pth{z}\bigprobin{z}{X_{\tau^{+}_{\cM_{k}}}
\in \6y}\6z\;, 
\end{equation}
where $\smash[b]{\mathring{\pi}^{B_i}_0}$ is the quasistationary distribution of
the process described by the kernel $K^0_{B_i}$.  Also recall that
$\smash[b]{K^0_{B_i}}$ is the kernel associated to the trace process
$(X_n)\vert_{\meta{k}}$ killed upon
leaving $B_i$. \added{
To remind us that we are not looking at the process described by the
kernel defined on $\Sigma$ but at the process}
\begin{equation}
\added{
\bigpar{X_n\vert_{\cM_k}}_{n\geqs0}
 = \Bigpar{X_{\tau^{+,n}_{\cM_k}}}_{\! n\geqs0}\;,}
\end{equation} 
\added{i.e., the trace of the original process at the return times to}\added{  $\cM_{k} =
\bigcup_{i=1}^{k} B_i$}\added{, we use the symbol}{ $\mathring{\phantom{a}}$.} 
\added{Since we study the killed process, we also follow the notations introduced in Section \ref{ssec:killed}, by denoting its eigenvalues by $\smash[b]{\mathring{\ev}^{B_i}_j}$ and its left
and right eigenfunctions by $\smash[b]{\mathring{\pi}^{B_i}_j\pth{x}}$ and
$\smash[b]{\mathring{\phi}^{B_i}_j\pth{x}}$ respectively.}
The principal eigenvalue $\smash{\mathring{\ev}_0^{B_{i}}}$ of this kernel is
given by
\begin{equation}
\label{eq:lambda0_qed} 
\mathring{\ev}_0^{B_{i}}= \Prcx{\mathring{\pi}_0^{B_i} }{ \tau^{+}_{B_i} < \tau^{+}_{\cM_{k} \bs B_i} }\; .
\end{equation} 
Using the spectral decomposition, we can introduce the function $g\pth{x,y}$  such that the density of $K^0_{B_i}$ satisfies
\begin{equation}
k^{0}_{B_i}\pth{x,y} = \mathring{\ev}^{B_i}_0 \braces{ \mathring{\pi}^{B_i}_0\pth{y}  \mathring{\phi}^{B_i}_0\pth{x}  + \frac{\mathring{\ev}^{B_i}_1}{\mathring{\ev}^{B_i}_0} g\pth{x,y}}\; .
\label{Eq:SpectralDecompositionKZBI}
\end{equation}
Note that due to orthogonality of eigenfunctions (see \eqref{eq:orthogonality}),
\begin{equation}
\int_{B_i} g\pth{x,y} \mathring{\phi}^{B_i}_0\pth{y} \6y  =0 \; , \qquad
\int_{B_i}\mathring{\pi}^{B_i}_0\pth{x}  g\pth{x,y} \6x  =0\; .
\end{equation}
It follows that
\begin{equation}
\pth{k^{0}_{B_i}}^m\pth{x,y} = \bigpar{\mathring{\ev}^{B_i}_0}^m \braces{
\mathring{\pi}^{B_i}_0\pth{y}  \mathring{\phi}^{B_i}_0\pth{x}  +
{\pth{\frac{\mathring{\ev}^{B_i}_1}{\mathring{\ev}^{B_i}_0}}}^m g^m\pth{x,y}}\;.
\label{Eq:SpectralDecompositionKZBIIterated}
\end{equation}
In addition $g$ has spectral radius $1$. 

\subsection{Spectral gap estimate}
%

\begin{proposition}[Adapted from {\cite[Proposition 5.5]{berglund2014noise}}]
\label{prop:spectral_gap} 
Assume that for some $n \in \N$, the density of the $n$-fold iterated kernel
${\pth{k^{0}_{B_i}}}^n$ satisfies a uniform positivity condition, i.e., there
exists $L(n)>1$ such that
\begin{equation}
\label{eq:uniform_positivity} 
\infSur{x_0 \in B_i}{(k^0_{B_i})^n\pth{x_0,y} } \leqslant 
(k^0_{B_i})^n\pth{x,y}\leqslant L\pth{n} \added{\infSur{x_0 \in
B_i}{(k^0_{B_i})^n\pth{x_0,y} }}\ \qquad \forall x, y \in B_i \;.
\end{equation}
Then $\theta = \absN{\mathring{\ev}_1}/\mathring{\ev}_0$ satisfies
\begin{equation}
\label{eq:spectralgap} 
{\theta}^n \leqslant L\pth{n} - \frac{\infSur{x \in B_i}{  \Bigprobin{x}{
\tau^{+,n}_{B_i} < \tau^{+}_{\cM_{k} \bs B_i}  }  }}
{ \bigpar{\mathring{\ev}^{B_i}_0}^{n}}\; .
\end{equation}
\end{proposition}
\begin{proof}
To ease notation, we prove the result for $n=1$, but one can show that it is still true for all $n\geqslant 2$. 
For any $l \geqslant 1$, the eigenvalue equation for $\mathring{{\ev}}^{B_i}_l$
and the orthogonality relation \eqref{eq:orthogonality} of the eigenfunctions
$\mathring{\phi}_l^{B_i}$ and $\mathring{{\pi}}_l^{B_i}$ give 
\begin{align}
\nonumber
\mathring{{\ev}}^{B_i}_l \mathring{\phi}_l^{B_i}\pth{x} &= \int_{B_i}
k^0_{B_i}\pth{x,y} \mathring{\phi}^{B_i}_l\pth{y}  \6y\;, \\
0 &= \int_{B_i} \mathring{{\pi}}^{B_i}_0\pth{y}
\mathring{\phi}^{B_i}_l\pth{y}\6y \; .
\end{align}
For any $\kappa>0$, we thus obtain
\begin{equation}
\mathring{{\ev}}^{B_i}_l \mathring{\phi}_l^{B_i}\pth{x}  = \int_{B_i} \crochets{
k^0_{B_i}\pth{x,y}  - \kappa  \mathring{{\pi}}^{B_i}_0\pth{y} }  
\mathring{\phi}^{B_i}_l\pth{y} \6y \; .
\end{equation}
Let us denote by $x_0$ the point in $B_i$ where $ \mathring{\phi}^{B_i}_l\pth{y}$ reaches its supremum. Evaluating the last equation in $x_0$ we obtain
\begin{equation}
\bigabs{\mathring{{\ev}}^{B_i}_l   } \leqslant \int_{B_i} \abs{
k^0_{B_i}\pth{x_0,y}  - \kappa  \mathring{{\pi}}^{B_i}_0\pth{y} }  
  \6y \; .
\end{equation}
Remark that for all $y\in B_i$,
\begin{equation}
\mathring{{\ev}}^{B_i}_0\mathring{{\pi}}^{B_i}_0 \pth{y}
= \int_{B_i}\mathring{{\pi}}^{B_i}_0 \pth{x} k^0_{B_i}\pth{x,y} \added{\6x} 
\added{\geqslant \infSur{x \in B_i}{k^0_{B_i}\pth{x,y} }}\;.
\end{equation}
Taking $\kappa =\mathring{{\ev}}^{B_i}_0L\pth{1}$, we can remove the absolute
value and write
\begin{align}
\nonumber
\bigabs{\mathring{{\ev}}^{B_i}_l   }& \leqslant \int_{B_i}
\Bigbrak{\mathring{{\ev}}^{B_i}_0 L\pth{1} 
\mathring{{\pi}}^{B_i}_0\pth{y}  -  \added{\infSur{x \in B_i}{ k^0_{B_i}\pth{x,y}}}}   \6y \\
&= \mathring{{\ev}}^{B_i}_0L\pth{1} - \added{\infSur{x \in B_i}{\Prcx{x}{{\tau^{+}_{B_i}} <
\tau^{+}_{\cM_{k} \bs B_i} }}}\;,
\end{align}
which proves~\eqref{eq:spectralgap} for $n=1$. 
\end{proof}
 
The two following results based on Harnack inequalities~\cite{Gilbarg_Trudinger}
will enable us to prove that $n$ and $L(n)$ satisfying the uniform positivity
condition~\eqref{eq:uniform_positivity} exist.

\begin{lemma}[{\cite[Lemma 5.7]{berglund2014noise}}]
For any set $\cD_{\added{1}}$ such that its closure satisfies $\bar{\cD_{\added{1}}} \subset \cD$, there exists a constant $C$, independent of $\sigma$, such that
\begin{equation}
\frac{\supSur{x \in \cD_{\added{1}}} k^0_{B_i}\pth{x,y}}{\infSur{x \in \cD_{\added{1}}}  k^0_{B_i}\pth{x,y}} \leqslant \e^{C / \sigma^2} 
\end{equation}
for all $y \in \partial D$.
\label{Lemma:Harnack1}
\end{lemma}

\begin{lemma}[{\cite[Lemma 5.8]{berglund2014noise}}]
\label{Lemma:Harnack2}
Let $\cB_{r}\pth{x}$ denote the ball of radius $r$ centred in $x$, and let
$\cD_{\added{1}}$ be such that its closure satisfies $\bar{\cD}_{\added{1}} \subset \cD$. Then for
any $x_0 \in \cD_{\added{1}}, y \in \partial \cD$, and $\eta >0$, one can find a constant
$r=r\pth{y,\eta}$, independent of  $\sigma$, such that
\begin{equation}
\supSur{x \in \cB_{r \sigma^2}\pth{x_0}}{ k^0_{B_i}\pth{x,y}} \leqslant \pth{1+\eta} \infSur{x \in \cB_{r \sigma^2}\pth{x_0}}{ k^0_{B_i}\pth{x,y}}\; .
\end{equation}
\end{lemma}

\begin{proposition}
\label{prop:coupling} 
For $x_1, x_2 \in B_i$, define the integer stopping time 
\begin{equation}
\label{eq:def_N} 
N = N\pth{x_1,x_2} = \inf \braces{n \geqslant 1 : \abs{\hat{{X}}_n^{x_2} -
\hat{{X}}_n^{x_1}}\leqslant r_{\eta} \sigma^2  }\;,
\end{equation}
where $\hat{{X}}_n^{x_0}$ denotes the Markov chain with transition kernel
$K^{0}_{B_i}(x_0,\6y)/K^{0}_{B_i}(x_0,B_i)$ (i.e.\ the Markov chain conditioned
to stay in $B_i$) and initial condition $x_0$, and $r_{\eta}$ is the constant of
Lemma \ref{Lemma:Harnack2}. 
Let
\begin{equation}
\rho_{n} = \supSur{x_1,x_2 \in B_i}{\prob{N\pth{x_1,x_2}>n}} \; .
\end{equation}
Then for any  $n \geqslant 2$, and any $\eta >0$, 
the transition kernel ${(K^{0}_{B_i})}^{n}(x,\6y)$ fulfils a uniform positivity
condition with constant $L(n)$ satisfying 
\begin{equation}
\label{eq:rho_n} 
L(n) \leqslant \frac{ 1+ \eta + \rho_{n-1} \e^{C/\sigma^2} }{\infSur{x
\in B_i}{  \Prcx{x}{ \tau^{+,n}_{B_i} < \tau^{+}_{\cM_{k} \bs B_i}  }  }  } \; ,
\end{equation}
where $C$ does not depend on $\sigma$.
\end{proposition}

\begin{proof}
Thanks to \cite[Proposition 5.9]{berglund2014noise},
we obtain that
\begin{equation}
\supSur{x \in B_i}{ \frac{(k^0_{B_i})^n\pth{x,y}}{(K^0_{B_i})^n\pth{x,B_i}} }
\leqslant  \infSur{x \in B_i}{
\frac{(k^0_{B_i})^n\pth{x,y}}{(K^0_{B_i})^n\pth{x,B_i}} } \pth{ 1 + \eta +
\rho_{n-1} \e^{C/\sigma^2} } \qquad \forall y \in B_i \; .
\end{equation}
The result is then immediate.
\end{proof}

\subsection{Oscillations of the principal right eigenfunction}

\begin{proposition}
\label{prop:oscillation_phi0} 
Assume that $(k^0_{B_i})^n$ satisfies the uniform positivity
condition~\eqref{eq:uniform_positivity} for some $n \in\N$.
Then there exists $M>0$, such that the normalised principal
right eigenfunction of $K^0_{B_i}$ satisfies
\begin{equation}
\norm{ \mathring{\phi}^{B_i}_0 - 1 } \leqslant M L(n)^2 \supSur{x \in
B_i} {\abs{1- \frac{ \Prcx{x}{{\tau^{+,n}_{B_i}} < \tau^{+}_{\cM_{k} \bs B_i}}}
{\bigpar{\mathring{\ev}_0^{B_i}}^n}}}  \; .
\end{equation}
\end{proposition}

\begin{proof}
The uniform positivity condition implies that we can apply
\cite[Theorem 3, Lemma 3]{Birkhoff1957}, which tells us that for any bounded
measurable function $f : B_i \rightarrow \R$, there exists a constant $M\pth{f}$
such that for all $m\in\N$, 
\begin{equation}
\norm{{(K^0_{B_i})}^{nm} f - (\mathring{\ev}_0^{B_i})^{nm}
(\mathring{\pi}^{B_i}_0 f) 
\mathring{\phi}^{B_i}_0  }   \leqslant M\pth{f} \varrho^m
(\mathring{\ev}_0^{B_i})^{nm}
\norm{\mathring{\phi}^{B_i}_0} \; ,
\end{equation}
where $\varrho <1$. Inspecting the proofs in~\cite{Birkhoff1957} shows that 
$\varrho$ satisfies $\varrho \leqslant 1- 1/L(n)^2$.
Taking $f(x) = {1}$, it follows that
\begin{equation}
\bigabs{ 
(K^0_{B_i})^{nm}(x,B_i) - 
(\mathring{\ev}_0^{B_i})^{nm} \mathring{\phi}^{B_i}_0\pth{x}      } \leqslant
M(1) \varrho^m (\mathring{\ev}_0^{B_i})^{nm} \norm{\mathring{\phi}^{B_i}_0}
\; .
\end{equation}
Dividing by $(\mathring{\ev}_0^{B_i})^{nm}$ 
and using the spectral
decomposition~\eqref{Eq:SpectralDecompositionKZBIIterated}, 
we get 
\begin{equation}
\abs{\int_{B_i}  
{\pth{\frac{\mathring{\ev}^{B_i}_1}{\mathring{\ev}^{B_i}_0}}}^{n m} g^{n
m}\pth{x,y} \6y} = 
\abs{\frac{ (K^0_{B_i})^{nm}(x,B_i)}
{\bigpar{\mathring{\ev}_0^{B_i}}^{nm}} - \mathring{\phi}^{B_i}_0\pth{x} }
\leqslant
M(1) \varrho^m \norm{\mathring{\phi}^{B_i}_0}\;. 
\label{Eq:BirkApplication}
\end{equation}
Since $\varrho<1$, taking the limit $ m \rightarrow \infty$, we obtain 
\begin{equation}
\mathring{\phi}^{B_i}_0\pth{x} 
= \lim_{m \rightarrow \infty}
\frac{(K^0_{B_i})^{nm}(x,B_i)}
{\bigpar{\mathring{\ev}_0^{B_i}}^{nm}} 
= \lim_{m \rightarrow \infty}
\frac{ \Prcx{x}{{\tau^{+,nm}_{B_i}} < \tau^{+}_{\cM_{k} \bs B_i}}}
{\bigpar{\mathring{\ev}_0^{B_i}}^{nm}} \; .
\end{equation}
Let $(h_m)_{m\geqs0}$ be the sequence of bounded measurable functions in $B_i$
defined by $h_0 =1$, and 
\begin{equation}
h_{m+1} \pth{x} = \frac{1}{\bigpar{\mathring{\ev}_0^{B_i}}^n} \int_{B_i}
(k^0_{B_i})^n\pth{x,y} h_{m}\pth{y} \6y\; ,
\end{equation}
so that for all $m$
\begin{equation}
h_m \pth{x} = \frac{(K^0_{B_i})^{nm}(x,B_i)}
{\bigpar{\mathring{\ev}_0^{B_i}}^{nm} } \; .
\end{equation}
We can now use a telescopic series to estimate 
\begin{align}
\nonumber
{1- \mathring{\phi}^{B_i}_0 \pth{x}}&= {h_0\pth{x} - \lim_{m \rightarrow \infty }{h_m}\pth{x} }\\
\nonumber
&=  \sum_{m=0}^\infty \bigbrak{h_m\pth{x}- h_{m+1}\pth{x}}  \\
&= \sum_{m=0}^\infty \int_{B_i}  {  \frac{{(k^0_{B_i})}^{n m}\pth{x,y}}{
\bigpar{\mathring{\ev}_0^{B_i}}^{nm} } \bigbrak{h_0 - h_1(y)} \6y } \; .
\end{align}
Since $\displaystyle\int_{B_i}\mathring{\pi}^{B_i}_0\pth{x} \bigbrak{h_0-
h_1\pth{x}} \6x= 0$, the spectral decomposition
\eqref{Eq:SpectralDecompositionKZBIIterated} and 
\eqref{Eq:BirkApplication} yield
\begin{align}
\nonumber
\norm{1 - \mathring{\phi}^{B_i}_0 } & \leqslant  
\sup_{x\in B_i} \sum_{m=0}^\infty \abs{\int_{B_i}  
{\pth{\frac{\mathring{\ev}^{B_i}_1}{\mathring{\ev}^{B_i}_0}}}^{n m} g^{n
m}\pth{x,y} \6y} \, \norm{h_0- h_1}      \\
&\leqslant \sum_{m=0}^\infty   M(1) \varrho^{m}
\norm{\mathring{\phi}^{B_i}_0} \norm{h_0- h_1}\; .
\end{align}
Since $\sum_m \varrho^{m} \leqs L(n)^2$ and $h_1(x) =
(\mathring{\ev}^{B_i}_0)^{-n} (K^0_{B_i})^n(x,B_i)$, the result follows.
\end{proof}

\section{Estimates on operators norms}
\label{sec:diluted} 
The aim of this section is to show that the kernel $K^{u}$ (or its $m$-fold
iterates) defined on $\cM_{{k}}$ by
\begin{equation}
\label{eq:defKu} 
K^{u}\pth{x,\6y}  = \Biggexpecin{x}{\e^{u( \tau^+_{\meta{k}}-1)}
\bigind{X_{\tau^+_{\meta{k}}}  \in \6y}}\; ,
\end{equation}
can be approximated by a finite-rank operator $K^{\star}$ (or its $m$-fold
iterates) given by 
\begin{equation}
K^{\star}\pth{x,\6y}  = \sum_{i=1}^{k} \ind{x \in B_i}  \int_{B_i}
\mathring{\pi}^{B_i}_0\pth{ x_0 } \ K^0\pth{x_0, \6y}  \6x_0 \; .
\end{equation}
We will first compare $K^u$ to $K^0$,
and then compare $K^0$ to $K^{\star}$ (and similarly for their iterates).

\subsection{Comparison between $\Ku$, $K^0$ and their $m$-fold iterates}
\label{Ssection:KUKZ}
Note that the difference between $\Ku$ and $K^0$ is given by
\begin{equation}
\pth{\Ku  - K^{0}}\pth{x , \6y} =  \Biggexpecin{x}{ \bigpar{\e^{u
(\tau^+_{\cM_{k}}-1)}  - 1} \bigind{X_{\tau^+_{\cM_{k}}} \in \6y} } \; .
\end{equation}
The following proposition enables us to bound the norm of this difference.

\begin{proposition}
\label{Prop:NormKuK0}
For all real  $u$ verifying the Laplace condition given by
\eqref{EqLaplaceTransform} with $A=\cM_k$, and such that $\pth{1-\e^{-u}}
\added{\bigexpecin{\cM_{k}^c}{ \tau^+_{\cM_{k}} }}<1$, we have
\begin{equation}
\norm{\Ku  - K^{0}} \leqslant \frac{\pth{1-\e^{-u}}
\bigexpecin{\cM_{k}}{\tau^+_{\cM_{k}}-1} }{ 1- \pth{1-\e^{-u}}
\bigexpecin{\cM_{k}^c}{ \tau^+_{\cM_{k}} }} \label{Eq:NormKuK0} \; .
\end{equation}
\end{proposition}

\begin{remark}
Note that for real $u$, the two conditions on $u$ can be summarised  as follows:
\begin{equation}
\max\pth{  \bigprobin{\cM_{k}^c}{X_1 \in \cM_{k}^c} ,
\frac{\bigexpecin{\cM_{k}^c}{ \tau^+_{\cM_{k}} } -1}{\bigexpecin{\cM_{k}^c}{
\tau^+_{\cM_{k}} }} } < \e^{-u} \;.
\end{equation}
\end{remark} 

To prove this proposition, we will use the following expression for the
inverse of $\pth{\id- K_{A^c}}$ (which is its resolvent at $z=1$).

\begin{lemma}\label{Prop:EquationBoundaryValueProblemZ}
Assume that there is a set $A \subset \Sigma$ such that
\begin{equation}
\supSur{x \in A^c}{\Prcx{x}{X_1 \in A^c}} <1\; .
\end{equation} 
Then the unique solution of the boundary value problem
\begin{align}
\nonumber
\pth{ \pth{ \id - K } r }\pth{x}  &= g(x), & x& \in A^c
\;,\\
r \pth{x} &= 0, &   x &\in {\setA} \; ,
\label{EqNonDirichletBoundaryValueProblem}
\end{align}
is given by
\begin{equation}
r(x)=\Biggexpecin{x}{ \sum_{n=0}^{\tau_{\setA}-1}  g\pth{X_n} }
\end{equation}
where by convention, the empty sum equals zero.
\end{lemma}
\begin{proof}
First, let us check that the proposed function solves the boundary value
problem. This is obvious for $x \in {\setA}$, since in that case, with the
convention taken for the empty sum, $  \added{\bigexpecin{x}
{\sum_{n=0}^{\tau_{\setA}-1} g\pth{X_n} }}=0$. 
For $x \in A^c$,
\begin{equation}
\pth{\pth{\id - K } r } \pth{x} =    \Espc{x}{ \sum_{n=0}^{\tau_{\setA}-1} 
g\pth{X_n} } - \Espc{x}{  \Espc{X_1}{\sum_{n=0}^{\tau_{\setA}-1}  g\pth{X_n}
}}\; .
\end{equation}
We can split the expectations according to the location of $X_1$, and use the
strong Markov property, to obtain
\begin{align}
\nonumber
&\pth{\pth{\id - K } r } \pth{x} \\
\nonumber
&\;{}= \Espc{x} { \ind{X_1 \in \setA} g\pth{x} } +
\Espc{x}{ \ind{X_1 \in A^c} \sum_{n=0}^{\tau_{\setA}-1}  g\pth{X_n}} 
- \Espc{x}{ \ind{X_1
\in A^c} \Espc{X_1}{\sum_{n=0}^{\tau_{\setA}-1}  g\pth{X_n} }}\\
\nonumber
&\;{}= \Espc{x} { \ind{X_1 \in \setA} g\pth{x} } + \Espc{x}{
\ind{X_1 \in A^c}\sum_{n=0}^{\tau_{\setA}-1}  g\pth{X_n} } -
\Espc{x}{\ind{X_1 \in A^c}\sum_{n=1}^{\tau_{\setA}-1}  g\pth{X_n} }\\
&\;{}= g(x)\; .
\end{align}
This shows that we have an admissible solution for all $x \in \Sigma$. 

Uniqueness is a consequence of the Fredholm alternative. Indeed, since 
\begin{equation}
\norm{ K_{A^c}} \leqslant\supSur{x \in A^c}{\Prcx{x}{X_1 \in A^c}} <1 \; ,
\end{equation}
we can apply \cite[Theorem 8.1]{gohberg2013classes}. In particular, $\pth{\id -
K_{A^c}}$  is invertible.
\end{proof}

\begin{remark}
For $A = {\cM_{k}}$,  since
\begin{equation}
\norm{ K_{\cM_{k}^c}} \leqslant\supSur{x \in \cM_{k}^c}{\Prcx{x}{X_1 \in
\cM_{k}^c}} <1 \; ,
\end{equation}
the assumption of Lemma~\ref{Prop:EquationBoundaryValueProblemZ} is
satisfied.
\end{remark}

\begin{proof}[{\sc Proof of Proposition \ref{Prop:NormKuK0}}]
Note that 
\begin{equation}
\norm{\Ku - K^0} \leqslant \supSur{x \in \cM_{k}}{ \bigexpecin{x}{ \e^{u
(\tau^+_{\cM_{k}}-1)}  - 1 }  }\; .
\end{equation}
Let us assume that this maximum is obtained for $\bar{x} \in \cM_{k}$.
Recognizing the sum of terms of a geometric sequence, we obtain
\begin{align}
\nonumber
 \bigexpecin{\bar{x} }{ \e^{u (\tau^+_{\cM_{k}}-1)}  - 1 } &= \pth{1 -
\e^{-u}} \Biggexpecin{\bar{x}}{ \sum_{n=1}^{\tau^+_{\cM_{k}} \added{-1}  } \e^{u n}} \\
\nonumber
 &= \pth{1 - \e^{-u}} \Biggexpecin{\bar{x}}{ \sum_{n=1}^{\tau^+_{\cM_{k}}-1}
\e^{u \npar{\tau^+_{\cM_{k}} \added{- n}}}}\\ 
 &= \pth{1 - \e^{-u}} \Biggexpecin{\bar{x}}{ \sum_{n=1}^{\tau^+_{\cM_{k}}-1}
\Espc{X_n}{ \e^{u {\tau_{\cM_{k}}}}}}\; .
\end{align}
 We thus  get
\begin{equation}
 \norm{\Ku - K^0} \leqslant \pth{1-\e^{-u}}
\bigexpecin{\cM_{k}}{\tau^+_{\cM_{k}}-1} \bigexpecin{\cM_{k}^c}{ \e^{u
{\tau_{\cM_{k}}}}}\; .
\end{equation}
 Let us now bound the expected value starting from $\cM_{k}^c$.
 Note that $r\pth{x} = \Espc{x}{\e^{u \tau_{\cM_{k}}}} -1$ solves the boundary
value problem
\begin{align}
 \nonumber
\pth{ \pth{\id - K} r} \pth{x} &= \pth{1- \e^{-u}}\Espc{x}{\e^{u
\tau_{\cM_{k}}}} &x \in \cM_{k}^c& \;,\\
r\pth{x} &=0  &x \in \cM_{k} &\; .
\end{align}
Thanks to Lemma~\ref{Prop:EquationBoundaryValueProblemZ}, we have  
\begin{equation}
 r\pth{x} =\Espc{x}{\e^{u \tau_{\cM_{k}}}} -1  
 = \pth{1-\e^{-u}}\Biggexpecin{x}{  \sum_{n=0}^{\tau_{\cM_{k}} -1} 
\bigexpecin{X_n}{\e^{u \tau_{\cM_{k}}}}   }\; .\label{Eq:remainder}
\end{equation}
Introducing $M = \Espc{\cM_{k}^c}{\e^{u \tau_{\cM_{k}}}} $, and taking the
supremum for  $x\in \cM_{k}^c$ in \eqref{Eq:remainder}, we obtain 
\begin{equation}
M -1 \leqslant \pth{1-\e^{-u}}  \bigexpecin{\cM_{k}^c}{  \tau^+_{\cM_{k}}} M\; .
 \end{equation}
Thus if $\pth{1-\e^{-u}}  \bigexpecin{\cM_{k}^c}{  \tau^+_{\cM_{k}}} <1 $, we
have
 \begin{equation}
 M =\bigexpecin{\cM_{k}^c}{\e^{u \tau_{\cM_{k}}}} \leqslant \frac{1}{1-
\pth{1-\e^{-u}}  \bigexpecin{\cM_{k}^c}{  \tau^+_{\cM_{k}}} }\; ,
 \end{equation}
 which gives the result.
\end{proof}

\begin{remark}
Note that in the previous proof, we have obtained the bound
\begin{equation}
\Biggexpecin{\cM_{k}}{ \sum_{n=1}^{\tau^+_{\cM_{k}}-1  } \e^{u n}} \leqslant
\frac{ \bigexpecin{\cM_{k}}{\tau^+_{\cM_{k}}-1}}{1- \pth{1-\e^{-u}}
\bigexpecin{\cM_{k}^c}{ \tau^+_{\cM_{k}} } }\label{Eq:BoundSumUpToTau}\; .
\end{equation}
\end{remark}

We are now going to bound the supremum norm of the difference between the
 iterates of these two kernels. We recall that we want to prove
\begin{proposition}
\label{Prop:NormKuKZIteratedk}
For all real  $u$ verifying the Laplace condition given by
\eqref{EqLaplaceTransform} \added{with $A=\cM_k$}, and such that $\pth{1-\e^{-u}}
\bigexpecin{\cM_{k}^c}{
\tau^+_{\cM_{k}} }<1$, we have
\begin{equation}
\norm{\pth{\Ku}^m  - \pth{K^{0}}^m} \leqslant \Biggpar{ 1 +
\frac{\pth{1-\e^{-u}} \bigexpecin{\cM_{k}}{\tau^+_{\cM_{k}}-1} }{ 1-
\pth{1-\e^{-u}} \bigexpecin{\cM_{k}^c}{ \tau^+_{\cM_{k}} }}}^m  -1 \; .
\label{Eq:NormKUKZIteratedk}
\end{equation}
\end{proposition}

\begin{proof}
To ease notation, we introduce 
$
\smash{\tau^+_m = \tau^{+,m}_{\cM_k}}
$
for the $m^\text{th}$ return time to $\cM_k$. 
Note that the $m^\text{th}$ iterated kernel of $\Ku$ is given by 
\begin{equation}
\pth{\Ku}^m \pth{x , \6y} = \Espc{x}{ \e^{u (  \tau^+_m - m ) } 
\mathds{1}_{\big\{ X_{\tau^+_m} \in   \6y \big\} }}\; .
\end{equation}
Therefore, the norm of the difference between the 
 iterates of $\Ku$ and $K^0$ satisfies 
\begin{equation}
\norm{\pth{\Ku}^m  - \pth{K^{0}}^m} \leqslant  \supSur{x \in \cM_{k}}{\Espc{x}{
\e^{u (\tau^+_m - m)} -1 }}\; .
\end{equation}
As previously,  recognizing the sum of terms of a geometric sequence, we can
bound the norm by 
\begin{equation}
\norm{\pth{\Ku}^m  - \pth{K^{0}}^m} \leqslant \pth{1- \e^{-u}}  \supSur{x \in
\cM_{k}}{\Biggexpecin{x}{ \sum_{n=1}^{\tau^+_m-m} \e^{u n} } }\; .
\end{equation}
We can now split the expected value of the sum as follows:
\begin{equation}
\Biggexpecin{x}{ \sum_{n=1}^{\tau^+_m-m}  \e^{u n}  }  = \Biggexpecin{x}{
\sum_{n=1}^{\tau^+_{1}-1}   \e^{u n}  } + \Biggexpecin{x}{
\sum_{n=\tau^+_{1}}^{\tau^+_m-m} \e^{u n}  }\; .
\end{equation}
Using the strong Markov property for the second term on the
right-hand side we get
\begin{equation}
\Biggexpecin{x}{ \sum_{n=1}^{\tau^+_m-m}   \e^{u n}  }  = \Biggexpecin{x}{
\sum_{n=1}^{\tau^+_{1}-1}  \e^{u n}   }+ \Biggexpecin{x}{
\e^{u \pth{ \tau^+_{1}-1}} \Biggexpecin{X_{\tau^+_{1}}}
{\sum_{n=1}^{\tau^+_{m-1}-(m-1)} \e^{u n}  }} \; .
\end{equation}
%
Denoting for all $m \in \N$
\begin{equation}
t_m =  \Espc{\cM_{k}}{ \sum_{n=1}^{\tau^+_m-m}   \e^{u n}  }\; ,
\end{equation}
we obtain the induction relation
\begin{equation}
t_m \leqslant t_1 + t_{m-1} \pth{1+ \pth{1-\e^{-u}} t_1}\; .
\end{equation}
Thus, the general term can be bounded by
\begin{equation}
t_m \leqslant \frac{\pth{1 + \pth{1- \e^{-u}} t_1}^m-1}{1-\e^{-u}} \; .
\end{equation}
Using the bound found in \eqref{Eq:BoundSumUpToTau} for $t_1$, it follows that
\begin{equation}
\Biggexpecin{  \cM_{k}}{ \sum_{n=1}^{\tau^+_m-m}   \e^{u n}  }  \leqslant
\frac{{\Biggpar{1 + \pth{1- \e^{-u}}  \Biggexpecin{ \cM_{k} }{
\displaystyle\sum_{n=1}^{\tau^+_{1}-1  } \e^{u n}  }}}^m  -1  }{1 - \e^{-u}}\; ,
\end{equation}
which gives the result.
\end{proof}

\subsection{Comparison between $K^0$, $K^{\star}$ and their $m$-fold iterates}
\label{ssec:K0Kstar} 

The aim of this section is to prove the following proposition:
\begin{proposition}
\label{Prop:NormK0KStark}
For all $m\in\N$, the norm of the difference between the 
 iterates of $K^0$ and $K^{\star}$ satisfies the bound 
\begin{equation}
\norm{\pth{K^{0}}^m  - \pth{K^{\star}}^m} \leqslant  \supSur{1 \leqslant i
\leqslant {k}} {R_i} \; ,
\end{equation}
where 
\begin{align}
\nonumber
R_i ={} &\norm{\mathring{\phi}^{B_i}_0-1} +2 \bigabs{\mathring{\ev}^{B_i}_1}^m +
2 \frac{1- \bigabs{\mathring{\ev}^{B_i}_1}^m}
{1-\bigabs{\mathring{\ev}^{B_i}_1}}
\Prcx{B_{i}}{\tau^+_{\cM_{k} \bs  B_i}  < \tau^+_{B_i}} \\
&{}+ m \pth{m-1} \Prcx{B_{i}}{\tau^+_{\cM_{k} \bs B_i}  < \tau^+_{B_i}}
\Prcx{\cM_{k} \bs B_i}{\tau^+_{B_i} <\tau^+_{\cM_{k} \bs B_i  }  }\; .
\end{align}
\end{proposition}
\begin{proof}
Let us first introduce the kernel $\npar{\check{K}}^m$ with density
\begin{equation}
\npar{\check{k}}^m \pth{x,y}= \sum_{i=1}^{k} \ind{x \in B_i}
\npar{\check{k}_i}^m \pth{x,y}
\end{equation}
where for any $x \in B_i$
\begin{multline}
\npar{\check{k}_i}^m \pth{x,y} = {\npar{k^0_{B_i}}}^m\pth{x,y} \\
+ \sum_{j=0}^{m-1} \! \int_{\! \cM_{k} \bs B_i} \! \int_{\! B_i}
{\npar{k^0_{B_i}}}^l\pth{x,z_1} k^0\pth{z_1,z_2} {\npar{k^0_{\cM_{k} \bs
B_i}}}^{m-l-1}\pth{z_2,y} \6z_1 \6z_2\; .
\end{multline}
Note that this kernel describes the process living on $\cM_{k}$ which can only
perform one transition, i.e.,\ starting in $B_i$ the Markov chain either stays
in $B_i$ or makes an excursion to $\cM_{k} \bs B_i$ and stays in this set.
We introduce the notation 
\begin{equation}
\Delta_m = \int_{\cM_{k}} \crochets{{\npar{k^0}}^m\pth{x,y} -
\npar{\check{k}}^m \pth{x,y}} \6y\; .
\end{equation}
We claim that for any $x \in B_i$, for all $m \geqslant 1$
\begin{equation}
\Delta_m \leqslant \frac{1}{2} m \pth{m-1}
\Bigprobin{B_i}{X_{\tau^+_{\cM_{k}}}  \notin B_i}
  \Bigprobin{\cM_{k} \bs B_i}{X_{\tau^+_{\cM_{k}}}  \in  B_i}\; .
\label{Eq:KKcheck}
\end{equation}
Let us prove this claim by induction. Since ${k^0}\pth{x,y} = \check{k}
\pth{x,y}  $ the base case is verified. 
The induction step is based on counting the possible ways to make more than one
transition when considering the $m+1^{\text{st}}$ iterate. At time $m$, either
the process has already made more than two transitions, or the process has made
one transition from $B_i$ to $\cM_{k} \bs B_i$ before time $m$ and made an
excursion from $\cM_{k} \bs B_i$ to $B_i$ at time $m$. Note that in the second
case, there are exactly $m$ different ways to perform such transitions
(depending on the time of the first excursion). It follows that
\begin{equation}
\Delta_{m+1} \leqslant \Delta_m + m \Bigprobin{B_i}{X_{\tau^+_{\cM_{k}}}  \notin
B_i} \Bigprobin{\cM_{k} \bs B_i}{X_{\tau^+_{\cM_{k}}}  \in  B_i}\; ,
\end{equation}
so that the general term indeed satisfies the bound~\eqref{Eq:KKcheck}. 

We can now bound, for all $m$, the norm of the difference between the 
 iterates of $K^0$ and $K^{\star}$, that is
\begin{equation}
\norm{\pth{K^{0}}^m  - \pth{K^{\star}}^m} \leqslant \maxSur{1 \leqslant i
\leqslant {k}}{ \supSur{x \in B_i}{ \int_{\cM_{k}}
\abs{{{\npar{k^0}}}^m\pth{x,y}- {{\npar{k^\star}}}^m\pth{x,y}  }}}  \6y\; .
\end{equation}
The triangle inequality yields 
\begin{align}
\nonumber
\abs{{{\npar{k^0}}}^m\pth{x,y}- {{\npar{k^\star}}}^m\pth{x,y}  }  \leqslant{}&
\abs{ {{\npar{k^0}}}^m\pth{x,y} -  \check{k}^m\pth{x,y}} \\
\nonumber
&{}+ \abs{\check{k}^m\pth{x,y} - \int_{B_i} \mathring{\pi}^{B_i}_0\pth{z}
\check{k}^m\pth{z,y} \6z }\\
&{}+\abs{\int_{B_i} \mathring{\pi}^{B_i}_0\pth{z} \pth{ \check{k}^m\pth{z,y}
-{{\npar{k^0}}}^m\pth{z,y} } \6z}\; .
\end{align}
Integrating over $\cM_{k}$, the first and the last term in the right-hand side
can be bounded using \eqref{Eq:KKcheck}. 
Using the spectral decomposition 
\eqref{Eq:SpectralDecompositionKZBI} of $k^0_{B_i}$, we obtain
\begin{multline}
\int_{B_i} \abs{\check{k}^m\pth{x,y} - \int_{B_i} \mathring{\pi}^{B_i}_0\pth{z}
\check{k}^m\pth{z,y}\6z  } \6y \\
\leqslant 
{\bigpar{\mathring{\ev}^{B_i}_0}}^m \abs{\mathring{\phi}_0^{B_i}\pth{x} -
1}
+2 \abs{ \mathring{\ev}^{B_i}_1}^m \supSur{z \in
B_i}{\abs{\int_{B_i}g^m\pth{z,y}} \6y}\; ,
\end{multline}
(since $\check{k}^m(x,y) = (k^0_{B_i})^m(x,y)$ if $x,y\in B_i$)
and  
\begin{multline}
\int_{\cM_{k} \bs B_i} \mkern-3mu \abs{\check{k}^m\pth{x,y} - \int_{B_i}
\mathring{\pi}^{B_i}_0\pth{z} \check{k}^m\pth{z,y}\6z  } \6y \\
\leqslant
\sum_{l=0}^{m-1} {\bigpar{\mathring{\ev}^{B_i}_0}}^{l}
\abs{\mathring{\phi}_0^{B_i}\pth{x} - 1}\pth{1- \mathring{\ev}^{B_i}_0} 
+ 2 \Bigprobin{B_i}{X_{\tau^+_{\cM_{k}}} \notin B_i}
 {\abs{ \mathring{\ev}^{B_i}_1}}^l \supSur{z \in
B_i}{\abs{\int_{B_i}g^l\pth{z,y} \6y}}\; .
\end{multline}
Regrouping the different terms, we obtain the result.
\end{proof}

\section{Perturbation theory for bounded linear operators}
\label{sec:operators} 
In the previous section, we have shown that the kernel $K^{u}$ 
(or its $m$-fold iterates) defined on $\meta{\kk}$ by
\begin{equation}
K^{u}\pth{x,\6y}  = \Biggexpecin{x}{\e^{u(\tau^+_{\meta{\kk}}-1)}
\bigind{X_{\tau^+_{ \meta{\kk}}}  \in\ \6y}}\; ,
\end{equation}
can be approximated by a finite-rank operator $K^{\star}$ (or its 
iterates). 
We now study the spectral properties of $K^\star$ (and its 
iterates) to deduce the spectral properties of $K^u$.

In the following, to ease notation, we will consider the case $m=1$, but the
results remain true for all $m$ considering the $m^{\text{th}}$ return time to
$\cM_k$.

\subsection{General idea}
\label{ssec:contour_general} 

Let $\uGs\pth{K^\star}$ denote the spectrum of the operator $K^\star$.
For $\uGs$ an isolated part of $\uGs\pth{K^\star}$, we define the Riesz
projection $\Pi_{\uGs}\pth{K^\star}$  by
\begin{equation}
\label{eq:riesz} 
\Pi_{\uGs}\pth{K^\star} = \frac{1}{2 \pi \icx } \int_{\cC} \pth{z \id -
K^{\star}}^{-1} \6z \; ,
\end{equation}
 where we assume that $\cC\subset\C$ is a Cauchy contour (in the resolvent set
of $K^\star$) around $\uGs$.  Recall that $\Pi_{\uGs}$ is a projection
\cite[Lemma 2.1]{gohberg2013classes} and that 
 \begin{equation}
 \Pi_{\uGs\pth{K^\star}}\pth{K^\star}= \id\; .
 \label{Eq:ProjectionOnSigma}
 \end{equation}
We want to know what happens to the spectrum $K^u$ when $K^u$ can be seen as a
perturbation of $K^\star$.
We choose a Cauchy contour $\cC$ in $\complex$ surrounding $\uGs\pth{K^\star}$
and first give a condition that ensures that $\cC$ does not contain any
eigenvalue of  $K^u$. This amounts to checking that $\pth{z \id - K^u}$ is
invertible for all $z \in \cC$.
\begin{proposition}[{\cite[Corollary 8.2]{gohberg2013basic}}]
\label{Prop:InverseOperator}
If $\pth{z \id - K^\star}$ is invertible and $\norm{K^u - K^\star} =
\norm{\pth{z \id - K^\star} - \pth{z \id - K^u}} < {\norm{{\pth{z \id -
K^\star}}^{-1}}^{-1}}$, then  $\pth{z \id - K^u}$ is invertible and
\begin{equation}
\norm{\pth{z \id - K^\star}^{-1} - \pth{z \id - K^u}^{-1} }  \leqslant
\frac{\norm{\pth{z \id - K^\star}^{-1}}^2 \norm{K^u - K^\star}}{1-\norm{\pth{z
\id - K^\star}^{-1}} \norm{K^u - K^\star} }\; .
\label{Eq:InverseOperator}
\end{equation}
\end{proposition}
To ensure that $\pth{z \id - K^u}$ is invertible for all $z \in \cC$, it is
rather natural to require  
\begin{equation}
\norm{K^u - K^\star} 
\leqslant \frac{1}{2} \gamma 
:= \frac12 \min \braces{ \norm{\pth{z \id - K^{\star}}^{-1}}^{-1} \vert z \in
\cC}\;.
\label{Eq:DefGamma}
\end{equation}
Under this assumption, \eqref{Eq:InverseOperator} shows that $\pth{z \id - K^u}$
is invertible for all $z \in \cC$ and that
\begin{align}
\norm{\pth{z \id - K^\star}^{-1} - \pth{z \id - K^u}^{-1} }  
&\leqslant 2 \norm{\pth{z \id - K^\star}^{-1}}^2 \norm{K^u - K^\star}\; .
\label{Eq:NormDiffResolv}
\end{align}
Thus the Riesz projection on the part of $\uGs\pth{K^u}$
inside $\cC$, given by 
\begin{equation}
\Pi = \frac{1}{2 \pi \icx} \int_{\cC} \pth{z \id - K^{u}}^{-1}  \6z\;,
\end{equation} 
is well-defined.
Using \eqref{Eq:ProjectionOnSigma} and \eqref{Eq:NormDiffResolv}, we obtain 
\begin{align}
\nonumber
\norm{\id - \Pi} &= \norm[\bigg]{\frac{1}{2 \pi \icx} \int_{\cC} \pth{z \id -
K^{\star}}^{-1} - \pth{z \id - K^{u}}^{-1}  \6z }\\
\nonumber
& \leqslant \frac{1}{2 \pi} \int_{\cC} \norm{ \pth{z \id - K^{\star}}^{-1} -
\pth{z \id - K^{u}}^{-1} } \6z\\
& \leqslant  \frac{1}{\pi}\int_{\cC} \norm{\pth{z \id -
K^\star}^{-1}}^2 \6z \, \norm{K^u - K^\star}\; .
\end{align}
If $\norm{\id - \Pi} <1$, since $\Pi$ is a projection, it follows that
$\id-\Pi=0$, 
and therefore $\uGs\pth{K^u}$ is inside $\cC$.
Thus a second natural assumption to make in order to control the spectrum of
$K^u$ is that 
\begin{equation}
\label{Eq:DefC}
C := \frac{1}{\pi} \int_{\cC}{\norm{\pth{z \id - K^{\star}}^{-1}}^{2} \6z} \,
 < \frac{1}{\norm{K^\star - K^u}}\; .
\end{equation}
%
This yields the following proposition.

\begin{proposition}[{\cite[Proposition 4.2]{gohberg2013classes}}]
\label{Prop:ContinuityEigenvalueAllSpectrum}
Let $\Omega$ be an open neighbourhood of $\uGs\pth{K^\star}$. Then there exists
$\epsilon>0$
such that $\uGs\pth{K^u} \subset \Omega$  for any operator $K\added{^u}$ with
$\norm{K^\star-K\added{^u}}< \epsilon$.
\end{proposition}

More precisely, the above discussion shows that if 
\begin{equation}
\norm{ K^{u} - K^{\star}}< \min\braces{\frac{1}{2} \gamma , \pth{C+1}^{-1}}\;,
\end{equation}
where $\gamma$ and  $C$ are the quantities introduced in \eqref{Eq:DefGamma} and
\eqref{Eq:DefC}
and $\cC$ is a Cauchy contour that separates a simple eigenvalue of
$K^{\star}$ from the remaining part of its spectrum, 
then $\cC$ also contains a simple eigenvalue for $K^u$. Before estimating the
quantities $\gamma$ and $\cC$, we need to study the spectrum of $K^\star$.

\subsection{\added{Estimation} of the eigenvalues of $K^{\star}$}
\label{SubSect:LocaEigenvalues}

We are interested in the eigenvalues of the finite-rank kernel $K^{\star}$ given
by
\begin{equation}
K^{\star}\pth{x,\6y}=\sum_{i=1}^{\kk} \ind{x \in B_i} \int_{B_i}
\mathring{\pi}^{B_i}_0\pth{x_0} \ K^0\pth{x_0, \6y}   \6x_0\; .
\end{equation}
A kernel has finite rank whenever it can be written as a sum of a finite
number of products of functions of its first argument alone by functions
\replaced{of}{ofn} its second argument alone. 
Because $K^{\star}$ has finite rank, we can associate to it a
${\kk} \times {\kk}$ matrix whose non-zero eigenvalues correspond to the
non-zero eigenvalues of $K^{\star}$. Indeed, non-zero eigenvalues of $K^{\star}$
are solutions of the homogeneous Fredholm equation of the second kind
\begin{equation}\label{Eq:EqValeurPropreKStar}
\ev \phi\pth{x} = \int_{\meta{\kk}} K^{\star}\pth{x, \6y} \phi\pth{y}\; .
\end{equation}
Let us introduce the unknown constants 
\begin{equation}
c_i = \int_{\meta{\kk}} \Bigprobin{ \mathring{\pi}^{B_i}_0}{
X_{\tau^+_{\meta{\kk}}} \in \6x}  \phi\pth{x}
\end{equation}
which depend on the eigenfunction $\phi\pth{x}$.
It follows that
\begin{equation}
\ev \phi\pth{x} = \displaystyle \sum_{i=1}^{\kk} c_i \ind{x \in B_i}\; .
\end{equation}
For $\ev \neq 0$, inserting this expression in \eqref{Eq:EqValeurPropreKStar} we
obtain
\begin{equation}
\displaystyle \sum_{i=1}^{\kk} \ind{x \in B_i} \biggbrak{c_i -  \frac{1}{\ev} 
\int_{\meta{\kk}} \Bigprobin{ \mathring{\pi}^{B_i}_0}{ X_{\tau^+_{\meta{\kk}}}
\in \6y}  \sum_{j=1}^{\kk} c_j \ind{y \in B_j} } = 0\; .
\end{equation}
Writing 
\begin{equation}
P_{ij} = \Bigprobin{ \mathring{\pi}^{B_i}_0}{ X_{\tau^+_{\meta{\kk}}} \in B_j} 
\end{equation}
and since $\pth{ \ind{x \in B_i}}_{1 \leqslant i \leqslant \kk}$ is a set of
linearly independent functions, we obtain the system of linear algebraic
equations
\begin{equation}
\ev c_i = \sum_{j=1}^{\kk}  {P_{ij}} c_j,\qquad 1\leqslant i \leqslant \kk\; .
\end{equation}
It follows that non-zero eigenvalues of $K^{\star}$ correspond to the non-zero
eigenvalues of the matrix $P$. For $0 \leqslant i \leqslant \kkMUn$, we denote
these eigenvalues $\ev_i$. 

Note that the matrix $P$ is a stochastic matrix and due to the Laplace transform
condition \eqref{EqLaplaceTransform}, these eigenvalues should satisfy
\begin{equation}
\supSur{x \in 
\meta{\kk}^c}{\Prcx{x}{X_1 \in 
\meta{\kk}^c}} < \abs{\ev_i}  \leqslant 1
\end{equation}
for all $0 \leqslant i \leqslant \kkMUn$.

Let us examine the structure of the matrix  $P$. Thanks to the large-deviations
estimates of Proposition \ref{Prop:LargeDeviation}, elements on the main
diagonal of $P$ are close to one, whereas off-diagonal elements  are close to
zero. In order to study a matrix where all elements are small we introduce 
$\hat{P} = \id - P$. Its diagonal elements are given by
\begin{equation}
\hat{P}_{ii}= \Bigprobin{ \mathring{\pi}^{B_i}_0}{ X_{\tau^+_{\meta{\kk}}}
\notin {B_i}} = \Bigprobin{ \mathring{\pi}^{B_i}_0}{ X_{\tau^+_{\meta{\kk}}} \in
\cM_k \bs {B_i}}\; .
\end{equation}
Our aim is now to derive spectral properties of the matrix $\hat{P}$. 
\added{Such a problem has been studied by Wentzell~\cite{Wentzell_72_matrices}
using $W$-graphs. Here we use a different approach based on}
block-triangularisation \cite[Section 6.1]{berglund2013eyring}, \added{which
also gives direct access to eigenfunctions}.
We write $\hat{P}$ in the form
\begin{equation}
\hat{P} =
\begin{pmatrix}
\hat{P}_{11}&\hat{P}_{12}\\
\hat{P}_{21}&\hat{a}\\
\end{pmatrix}
\end{equation}
where $\hat{P}_{11} \in \R^{(\kkMUn) \times (\kkMUn) }$, $\hat{P}_{12} \in
\R^{\kkMUn}$, ${\hat{P}_{21}}^{\top} \in \R^{\kkMUn}$ and $\hat{a}  
\in\R$.
We want to prove that there exist matrices $S,T$  in
$\R^{k\times k}$
of the form
\begin{equation}
S=\begin{pmatrix}
\id  & S_{12} \\ 
0 & 1
\end{pmatrix},\qquad T= \begin{pmatrix}
T_{11} & 0\\ 
T_{21} & \alpha
\end{pmatrix}
\end{equation}
with the submatrices having the same dimensions as those of $\hat{P}$ and
verifying
\begin{equation}
\hat{P} S= S T \; . \label{Eq:BlockTriangDecom}
\end{equation}
Following the argument of \cite[Section 6.1]{berglund2013eyring}, if we manage
to prove that 
\begin{equation} \label{Eq:FixedPointEquationMatrixIdMinusP}
\hat{P}_{11} S_{12} - S_{12} \hat{a} - S_{12} \hat{P}_{21} S_{12} + \hat{P}_{12}
= 0
\end{equation}
admits a unique solution, it will follow that $\hat{P}$ is similar to the
block-diagonal matrix $T$, and the eigenvalues of $\hat{P}$ are $\alpha$ and
those of $T_{11}$. \added{Note that $T_{11}, T_{21}$ and $\alpha$ are then given by}
\begin{equation}
\added{T_{11} =  \hat{P}_{12} - S_{12} \hat{P}_{12}\;, 
\quad T_{21} = \hat{P}_{21}\;, 
\quad \alpha = \hat{a} + \hat{P}_{21}S_{12}}\;.
\end{equation}
\added{The fact that~\eqref{Eq:FixedPointEquationMatrixIdMinusP} admits a
unique solution} will be proven using the Banach fixed point theorem. In
the sequel, the matrix norm used is the sup-norm.

\begin{proposition}\label{Prop:TriangUniqueness}
Introduce the notations 
\begin{align}
b &= \maxSur{1 \leqslant l \leqslant \kkMUn }{\Bigprobin{
\mathring{\pi}^{B_{l}}_0 }{ X_{\tau^+_{\meta{\kk}}} \notin {B_l}}} \; ,\\
\hat{a}&= \Bigprobin{ \mathring{\pi}^{B_k}_0 }{X_{\tau^+_{\meta{\kk}}}
\notin {B_k}} \neq 0 \; .
\end{align}
For fixed blocks $\hat{P}_{11}$, $\hat{P}_{12}$, $\hat{P}_{21}$ and  $\hat{a}$,
if
\begin{equation}
\frac{b}{\hat{a}}  
< \frac{1}{8} 
\end{equation} then  \eqref{Eq:FixedPointEquationMatrixIdMinusP} admits a unique
solution. Moreover, this solution satisfies 
\begin{equation}
\norm{  S^{\star}_{12} }\leqslant 2 \frac{\norm{\hat{P}_{12}}}{\hat{a}} \; .
  \label{Eq:BoundOnStar12}
\end{equation}
\end{proposition}
\begin{proof}
Let $\cB$ be the ball $\cB = \set{\Xi \in \R^{\kkMUn}, \norm{\Xi} \leqslant 2
\frac{\norm{\hat{P}_{12}}}{\hat{a}}} \subset\R^{\kkMUn}$. We equip the Banach
space  $\real^{\kkMUn}$ with the supremum norm, and define a map $\Phi : \cB
\rightarrow \cB$ by
\begin{equation}
\Phi\pth{ \Xi } =  \frac{1}{\hat{a}} \pth{  \hat{P}_{12} + \hat{P}_{11} \Xi -
\Xi \hat{P}_{21} \Xi } \; .
\end{equation}
Note that 
\begin{equation}
\norm{\hat{P}_{11}} \leqslant 2 b,\ 
\norm{\hat{P}_{12}} \replaced{{}\leqslant{}}{{}={}} b,\ 
\norm{\hat{P}_{21}} =
\hat{a} \; .
\end{equation}
It is then straightforward to check that $\Phi$ is a contraction on $\cB$.
\end{proof}

\begin{remark}
\label{rem:Sstar12} 
It follows from \eqref{Eq:FixedPointEquationMatrixIdMinusP} that
$S^{\star}_{12}$ satisfies
\begin{align}
\nonumber
  S^{\star}_{12}  &= {\pth{\id - \frac{\hat{P}_{11}}{\hat{a}}+
\frac{\hat{P}_{21} S^{\star}_{12}}{\hat{a}}\id  }}^{-1}
   \frac{\hat{P}_{12}}{\hat{a}}    \\
  & = \sum_{k \geqslant 0}  {\pth{\frac{\hat{P}_{11}}{\hat{a}}-
\frac{\hat{P}_{21} S^{\star}_{12} }{\hat{a}}\id }}^k
\frac{\hat{P}_{12}}{\hat{a}}   \; .
\end{align} 
This yields a more precise estimate than the a priori estimate
\eqref{Eq:BoundOnStar12}. In particular, at the first order, it follows
\begin{equation}
\biggnorm{ S^{\star}_{12} - \frac{\hat{P}_{12}}{\hat{a}} } \leqslant \frac{4b/
\hat{a}}{1-4b/ \hat{a}}  \frac{\norm{\hat{P}_{12}} }{\hat{a}}
\end{equation}
\end{remark}

\begin{corollary}\label{CorEigenvalueProp}
For $ 0 \leqslant i \leqslant \kkMUn$, we denote ${\ev}^{\star}_i$ the
eigenvalues of $K^{\star}$ labelled by \replaced{decreasing}{increasing} order.  
Then the smallest in modulus non-zero eigenvalue $\ev^{\star}_{\kkMUn}$ of
$K^{\star}$  is real, simple and satisfies
\begin{equation}
\abs{\ev^{\star}_{\kkMUn} -
\pth{1-\Bigprobin{\mathring{\pi}^{B_{\kk}}_0}{X_{\tau^+_{\cM_{\kk}}} \notin
B_{\kk}}}}  \leqslant 2 \norm{\hat{P}_{12}} \leqslant 2 b\; ,
\end{equation}
Furthermore, for $0 \leqslant i \leqslant \kkMDeux$,
the other non-zero eigenvalue\added{s} of $K^{\star}$ satisfy
\begin{equation}
\abs{1- \ev^{\star}_{i}} \leqslant \norm{T_{11}} \leqslant  4
\norm{\hat{P}_{12}} \leqslant 4 b\; .
\end{equation}
\end{corollary}
\begin{proof}
To each non-zero eigenvalue of $P$  corresponds a non-zero eigenvalue of
$K^{\star}$. From Proposition \ref{Prop:TriangUniqueness},  the biggest non-zero
eigenvalue of $\id-P$, thus the smallest non-zero of $K^{\star}$, is real and
positive and satisfies
\begin{equation}
\Bigl\vert \added{ ( 1- {\ev}^{\star}_{\kkMUn} ) - \Bigprobin{
\mathring{\pi}^{B_{\kk}}_0}{X_{\tau^+_{B_{\kk}}}  \notin B_{\kk} } } \Bigr\vert 
\leqslant 2 \norm{\hat{P}_{12}} \leqslant 2 b  \; .
\end{equation}
\end{proof}

\begin{remark}
Note that  $1- \bigprobin{ \mathring{\pi}^{B_{\kk}}_0}{X_{\tau^+_{B_{\kk}}} 
\notin B_{\kk} } $  is the principal eigenvalue of kernel $K^{\star}_{B_{\kk}}$,
(the process with kernel $K^{\star}$ killed upon leaving $B_{\kk}$). 
\end{remark}

\begin{figure}
\centering
\begin{tikzpicture}[x=2.5cm,y=2.5cm,>=stealth']
\draw [line width=0.1pt,color=black,fill=white, fill opacity=.25] (1.,0.) circle
(2.5cm);
\draw [line width=0.6pt,color=black,fill=white,fill opacity=1]
(0.5076690421225205,0.) circle (0.43721770551038847cm);
(0.2691726053063012cm);
\draw [line width=1pt,color=blue,fill=white,fill opacity=1] (1.,0.) circle
(0.25cm);
\draw[->,color=black] (-0.1,0.) -- (2.2,0.) node[below
left]{$\text{Re}\pth{1-\ev}$};
\draw[->,color=black] (0.,-1.1) -- (0.,1.2) node[below
right]{$\text{Im}\pth{1-\ev}$};
\draw[shift={(1,0)},color=black] (0pt,2pt) -- (0pt,-2pt) 
node[below,yshift=-3pt] {$1$};
\draw[shift={(2,0)},color=black] (0pt,2pt) -- (0pt,-2pt);
\draw[shift={(0,1)},color=black] (2pt,0pt) -- (-2pt,0pt)
node[left,yshift=-1pt] {$1$};
\draw[shift={(0,-1)},color=black] (2pt,0pt) -- (-2pt,0pt);
\draw [shift={(0.5076690421225205,0)},color=brown!70!black] (0pt,2pt) --
(0pt,-20pt) node[below,yshift=5pt] {\footnotesize
$\bigprobin{\mathring{\pi}^{B_{\kk}}_0}{X_{\tau^+_{\mathcal{M}_{\kk}}} \in
\mathcal{M}_{\kkMUn}}$};
\draw [line width=2pt,color=red] (0.4,0.)-- (0.615338084245041,0.);
\draw [fill=black,color=black] (0.5076690421225205,0.) circle (1.5pt);
\draw [shift={(0.,0.)},line width=0.4pt,color=purple,fill=purple,fill
opacity=1.0]  (0,0) -- 
plot[domain=-1.5238347597983815:1.5238347597983815,variable=\t]({
1.*0.09388861496207804*cos(\t r)+0.*0.09388861496207804*sin(\t
r)},{0.*0.09388861496207804*cos(\t r)+1.*0.09388861496207804*sin(\t r)}) --
cycle ;
\end{tikzpicture}
\caption{Sketch of \added{the} location of eigenvalues of $\hat{P}=\id-P$.}
\end{figure}
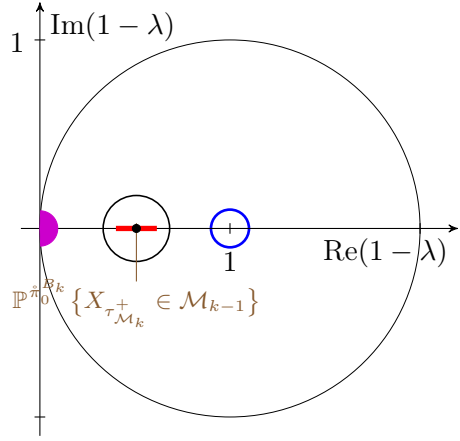

Thanks to the block-triangularisation, we also get an explicit expression for
the eigenfunction associated to the smallest eigenvalue of $K^{\star}$.
\begin{lemma}
\label{lem:phikstar} 
Up to a multiplicative constant, the eigenfunction of $K^{\star}$ corresponding
to the eigenvalue ${\ev}^{\star}_{\kkMUn}$ is given by
\begin{equation}
\phi^{\star}_{\kkMUn}\pth{x} = \sum_{i=1}^{\kkMUn} \ind{x \in B_i}
\pth{S^{\star}_{12}}_i + \ind{x \in B_{\kk}}
\end{equation}
where $ S^{\star}_{12}$ is a unique solution of
\eqref{Eq:FixedPointEquationMatrixIdMinusP}.
\end{lemma}
The proof is immediate since
\begin{equation}
\hat{P} {\begin{pmatrix} S_{12}^{\added{\star} }\\ 1
\end{pmatrix}} =  (1-{\ev}^{\star}_{\kkMUn} )  {\begin{pmatrix} S_{12}^{\added{\star}} \\ 1
\end{pmatrix}}\; .
\end{equation}

\subsection{Resolvent estimate of $K^{\star}$}

We now want to estimate the quantities $C$ and $\gamma$ associated to the real
and simple eigenvalue ${\ev}^{\star}_{\kkMUn}$ of $K^{\star}$, cf.\
\eqref{Eq:DefC} and  \eqref{Eq:DefGamma}. Thus we need an upper bound on the
norm of the resolvent  $\norm{\pth{z \id - K^{\star} }^{-1}} $ when $z$ is close
to ${\ev}^{\star}_{\kkMUn}$. Note that for $z \neq 0$, the resolvent operator
satisfies the  resolvent equation
\begin{equation}\label{Eq:ResolventeOperatorEg}
\pth{z \id - K^{\star}}^{-1} = \frac{1}{z} \pth{\id + K^{\star} \pth{z \id
-K^{\star}}^{-1}}\; .
\end{equation}
Solving the Fredholm linear integral equation of the second kind,
\begin{align}
\nonumber
z  \phi\pth{x} &= \varphi\pth{x} + \pth{K^{\star} \added{\phi}}\pth{x} \\
\nonumber
&= \varphi\pth{x} +
\int_{\meta{\kk}} \sum_{i=1}^{\kk} \ind{x \in B_i} \Bigprobin{
\mathring{\pi}^{B_i}_0}{ X_{\tau^+_{\meta{\kk}}} \in \6y }  \phi\pth{y}\\
&=  \varphi\pth{x} + 
\sum_{i=1}^{\kk} \ind{x \in B_i}  c_i
\label{Eq:NonHomogeneous}
\end{align}
and following the same procedure as for the homogeneous equation
\eqref{Eq:EqValeurPropreKStar}, we obtain the system of linear algebraic
equations given for all $1 \leqslant i \leqslant \kk$ by
\begin{equation}
z c_i - \sum_{i=1}^{\kk} P_{ij} c_j = \varphi_i
:= 
\displaystyle \int_{\meta{\kk}} \Bigprobin{
\mathring{\pi}^{B_i}_0}{X_{\tau^+_{\meta{\kk}}} \in \6y} \varphi\pth{y}\;.
\end{equation}
If $z$ is such that $\det\pth{z \id - K^{\star}} \neq 0$, i.e.\ not an
eigenvalue of $K^{\star}$, the system has the unique solution given by
\begin{equation}
c_i =  \sum_{j=1}^{\kk} \pth{z \id - P}^{-1}_{ij} \varphi_j, \qquad 1 \leqslant
i \leqslant \kk\; .
\end{equation}
Inserting the previous result in \eqref{Eq:NonHomogeneous}, we see that
\begin{equation}
z \phi\pth{x} = \varphi\pth{x} + \int_{\meta{\kk}}   \sum_{i=1}^{\kk}
\sum_{j=1}^{\kk} \ind{x \in B_i} \pth{z \id - P}^{-1}_{ij}
\Bigprobin{\mathring{\pi}^{B_{j}}_0}{X_{\tau^+_{\meta{\kk}}} \in \6y}
\varphi\pth{y}\; ,
\end{equation}
from which it follows that the resolvent operator $\pth{z \id -K^{\star}}^{-1}$
admits a resolvent kernel $R \pth{z; x, \6y}$ given by
\begin{equation}
R \pth{z; x, \6y}= \frac{1}{z} \biggbrak{ \id +\sum_{i =1}^{\kk}
\sum_{j=1}^{\kk} \ind{x \in B_i} \pth{z \id - {P}}^{-1}_{ij}
\Bigprobin{\mathring{\pi}^{B_{j}}_0}{ X_{\tau^+_{\meta{\kk}}} \in \6y }}\; .
\end{equation}
Since $ R  = \frac{1}{z} \braces{\id + \pth{R K^{\star}} }$ and  thanks to
\eqref{Eq:ResolventeOperatorEg}, it follows that the resolvent kernel of the
resolvent operator $\pth{z \id -K^{\star}}^{-1}$ is also given by
\begin{equation}
R \pth{z; x, \6y}=  \sum_{i =1}^{\kk} \sum_{j=1}^{\kk} \ind{x \in B_i} \pth{z
\id - {P}}^{-1}_{ij} \Bigprobin{\mathring{\pi}^{B_{j}}_0}{
X_{\tau^+_{\meta{\kk}}} \in \6y }\; .
\end{equation}
We are now able to bound the resolvent of $K^{\star}$. Let $\cC$ be the contour
defined by  
\begin{equation}
\bigsetsuch{z \in \replaced{\complex}{\cC}}{\abs{z- \ev_k} = r }\; ,
\end{equation}
and let us assume that
\begin{equation}
r < \frac{\hat{a}}{4} < \frac{\hat{a}-6b}{2} \; .
\label{Eq:RCondition}
\end{equation}
Then have the following resolvent estimate.

\begin{proposition}
Recall that we denote ${\ev}^{\star}_{\kkMUn}$ the smallest non-zero eigenvalue
of the kernel $K^{\star}$.  There exists a numerical constant $\cUn>0$,
independent of $\sigma$, such that for all $z \in \cC$
\begin{equation}
\norm{\pth{z \id  - K^{\star}}^{-1}} < \cUn \pth{z -
{\ev}^{\star}_{\kkMUn}}^{-1}\; .
\end{equation}
\end{proposition}
\begin{proof}
Note that we have equality between the supremum norms of the resolvent of the
operator $K^\star$ and of the matrix $P$, i.e.\
\begin{equation}
\norm{\pth{z \id  - K^{\star}}^{-1}} = \norm{\pth{z \id  - P}^{-1}} \; .
\end{equation}
Let us now derive an upper bound on $\norm{\pth{z \id - P}^{-1}}$ for $z \in \cC$.
Thanks to the block-triangularisation \eqref{Eq:BlockTriangDecom}, we get
\begin{align}
\nonumber
\norm{\pth{z \id - P}^{-1}} &= \Bignorm{\pth{\pth{1-z} \id - \hat{P}}^{-1}}\\
& \leqslant \norm{S} \norm{S^{-1}} \norm{ \pth{\pth{1-z} \id - {T}}^{-1} }\; .
\end{align}
Since $\norm{S}=\norm{S^{-1}} = 1+ \norm{S_{12}}$, we can bound  $\norm{S} \norm{S^{-1}}$  thanks to \eqref{Eq:BoundOnStar12}.
We also have an explicit expression for $ \pth{\pth{1-z} \id - {T}}^{-1}$, given by
\begin{multline}
 \bigbrak{(1-z) \id - {T}}^{-1} 
 \\ = 
 \begin{pmatrix}
 \added{\id} & 0 \\
 T_{21}
 (1-{\ev}^{\star}_{\kkMUn})^{-1} & 1
 \end{pmatrix} 
 \begin{pmatrix}
 \bigbrak{ (1-z)\added{\id} - T_{11}}^{-1} &0\\
 0& \bigbrak{(1-z)-(1-{\ev}^{\star}_{\kkMUn})}^{-1}
 \end{pmatrix}\; .
\end{multline}
Since $\abs{1-z}> \norm{T_{11}}$ for all $z \in \cC$, we have the classical bound
\begin{equation}
 \norm{ \pth{\pth{1-z}- T_{11}}^{-1} }  \leqslant \frac{1}{ \abs{1-z} -  \norm{T_{11} }}\; .
\end{equation}
For all $r< \frac{\hat{a}-6b}{2}$, we finally get
\begin{equation}
\norm{\pth{z \id - P}^{-1}} \leqslant \frac{1}{\abs{z - {\ev}^{\star}_{\kkMUn}}}
{\pth{1+2 \frac{b}{\hat{a}}}}^2 \pth{ 1+ \frac{\hat{a}}{\hat{a}-6b-r} }\; .
\end{equation}
The result is then immediate since 
\begin{equation}
\norm{\pth{z \id - P}^{-1}} \leqslant 9{\pth{1+\frac{1}{4}}}^2 {\abs{z -
{\ev}^{\star}_{\kkMUn}}}^{-1}\; .
\end{equation}
\end{proof}

We can now estimate $C$.  Since the resolvent kernel is bounded, the M-L
inequality yields the upper bound
\begin{equation}
C  = \frac{1}{\pi} \int_{\cC}{\norm{\pth{z \id - K^{\star}}^{-1}}^{2} \6z}
\leqslant 2 r \norm{\pth{z \id - K^{\star}}^{-1}}^{2}\; .
\end{equation}
It follows that
\begin{equation}
\epsilon = \min\braces{\frac{1}{2} \gamma , \pth{C+1}^{-1}} \geqslant  \frac{r}{396}\; .
\end{equation}
Thus, if 
\begin{equation}
\norm{K^u - K^{\star} } < \frac{r}{396}\; , 
\end{equation}
we have the desired result for the approximation of the eigenvalues of $K^{u}$, where $K^u$ is seen as an approximation of $K^\star$.
In Section \ref{sec:lastSteps}, we will check that this inequality and
\eqref{Eq:RCondition} are indeed satisfied.

\section{Sample path estimates}
\label{sec:sample_paths} 
\subsection{Stochastic differential equations with coexistence of periodic orbits}

In this section, if $x\in\R^d$ then $\norm{x}$ denotes the Euclidean norm of
$x$. We assume that the deterministic system \eqref{EqODE}
admits $N$ stable periodic orbits and $N^{+}$ unstable periodic orbits, i.e.,
there are periodic functions $\gamma^{-}_{i}: \real \rightarrow \mathcal{D}$ of
respective periods $T_{i}$ such that 
\begin{equation}
\dot\gamma_i\pth{t} = f\pth{ \gamma_i\pth{t}} \qquad \forall t \in \real
\end{equation}
for $1 \leqs i \leqs N$, and there are periodic functions $\gamma^{+}_{i}:
\real \rightarrow \mathcal{D}$ of respective periods $T^+_{i}$ such that 
\begin{equation}
\dot\gamma_i^+\pth{t} = f\pth{ \gamma^+_i\pth{t}}\qquad  \forall t \in \real
\; .
\end{equation}
In what follows, we study the behaviour of the system in the neighbourhood of a
stable or unstable periodic orbit.
Adapting \cite[Proposition 2.1]{berglund2014noise} and \cite[Proposition
3.3]{berglund2014noise} to the multidimensional case, it follows that the SDE
\eqref{eq:SDE}
can be written in polar-type coordinates as shown in the following proposition.
 
\begin{proposition}
\label{Prop:EqStablePeriodicOrbit}
There exists a change of coordinates such that in a neighbourhood of a stable
periodic orbit (i.e., for $\pth{x,\varphi}$ such that 
$\norm{x}$ is small enough), the SDE takes the form
\begin{align}
\nonumber
\dd x_t &= \pth{ - \Lambda x_t + b_x\pth{x_t, \varphi_t} } \dd t + \sigma
g_x\pth{x_t,\varphi_t} \dd W_t\;,\\
\dd \varphi_t& = \pth{ \frac{1}{T_i} + b_\varphi\pth{x_t, \varphi_t} } \dd t +
\sigma g_\varphi\pth{x_t,\varphi_t} \dd W_t\;,
\label{EqSystemeCompletRPhiStable}
\end{align}
where $\Lambda$ is a triangular matrix with positive diagonal elements
corresponding to the Lyapunov exponents of the stable orbit, $b_x$, $b_\varphi$,
$g_x$, $g_\varphi$ are periodic in $\varphi$ with period $1$, and the nonlinear
drift terms satisfy $\norm{b_x\pth{x, \varphi}},\abs{b_{\varphi}\pth{x, \varphi}} = \Order{\norm{x}^2}$.
\end{proposition}

\added{Note that we can choose $\Lambda$ to be in Jordan canonical form, and
that in these variables the $B_i$ can be taken to be balls
$\setsuch{x}{\norm{x}\leqs\delta}$.}
In the neighbourhood of an unstable periodic orbit, we have the following
similar result.

\begin{proposition}
\label{Prop:EqStablePeriodicOrbitUnstable}
There exists a change of coordinates such that in a neighbourhood of an
unstable periodic orbit (i.e., for $\pth{x,\varphi}$ such
that
$\norm{x}$ is small enough), the SDE takes the form
\begin{align}
\nonumber
\dd x_t &= \pth{ \begin{pmatrix}
- \Lambda^- & 0 \\ 
0 & \Lambda^+
\end{pmatrix} 
 x_t + b_x\pth{x_t, \varphi_t} } \dd t + \sigma g_x\pth{x_t,\varphi_t} \dd
W_t\;,\\
\dd \varphi_t& = \pth{ \frac{1}{T_i} + b_\varphi\pth{x_t, \varphi_t} } \dd t +
\sigma g_\varphi\pth{x_t,\varphi_t} \dd W_t\;,
\label{EqSystemeCompletRPhiUnstable}
\end{align}
where $\Lambda^-$ is a triangular 
matrix with positive diagonal elements, $\Lambda^{+}$ is a triangular matrix
with non-negative diagonal elements and at least one strictly positive diagonal
element, 
corresponding to the Lyapunov exponents, $b_x$, $b_\varphi$,
$g_x$, $g_\varphi$ are periodic in $\varphi$ with period $1$, and the nonlinear
drift terms satisfy $\norm{b_x\pth{x, \varphi}},\abs{b_{\varphi}\pth{x, \varphi}} = \Order{\norm{x}^2}$.
\end{proposition}

\begin{proof}[{\sc Proof of Propositions \ref{Prop:EqStablePeriodicOrbit} and
\ref{Prop:EqStablePeriodicOrbitUnstable}}]
Using the time parametrisation proposed in \cite{berglund2014noise}, i.e.,
setting $\Gamma_i\pth{ \varphi} = \gamma_i\pth{T_i \varphi}$ so that $\varphi
\in \real / \mathbb{Z}$ ($\varphi$ parametrises time),
and It\^{o}'s formula, the stochastic differential equation \eqref{eq:SDE}
is equivalent, thanks to  the transformation of Proposition
\ref{Prop:FloquetChangeOfVariable}, to a system of the form 
\begin{align}
\nonumber
\dd x_t &= f_x({x_t, \varphi_t, \sigma}) \dd t + \sigma g_x\pth{x_t, \varphi_t}
\dd W_t \\
\dd \varphi_t &= f_\varphi({x_t, \varphi_t, \sigma}) \dd t + \sigma
g_\varphi\pth{x_t, \varphi_t} \dd W_t \; .
\end{align}
As noted in \cite{berglund2014noise}, a drawback of this system is that the
drift term $f_x$ does not vanish in $x =0$. We use a similar argument as in 
\cite[Proposition 3.3]{berglund2014noise} to obtain the desired form.
\end{proof}

\subsection{General estimates}
\label{Ssection:GeneralEst}

We consider the system in continuous time describing the dynamics near a
periodic orbit in the polar-type coordinates
\eqref{EqSystemeCompletRPhiStable} or \eqref{EqSystemeCompletRPhiUnstable}.
We first recall a result proved in \cite{berglund2014noise} which shows that
$\varphi_t$ does not differ much from ${t}/{T_i}$ on rather long timescales.
Given $T, H>0$, we introduce two stopping times by
\begin{align}
\nonumber
\tilde{\tau}_{H} &= \inf \braces{ t>0:  \norm{x_t } \geqslant H}\;,\\
\tilde{\tau}_\varphi &= \inf \braces{ t>0:  \abs{\varphi_t - \frac{t}{T_i
}} \geqslant M\pth{H^2 t + \sqrt{H^3 T}}}\; .
\end{align}
Then \cite[Proposition 6.3]{berglund2014noise} gives us the following result.

\begin{proposition}[Control of the diffusion along
$\varphi$]\label{PropAlongTheDiffusion}
There is a constant $C_1$, depending
only on the ellipticity constants of the diffusion terms, such that
\begin{equation}
\Prcx{\pth{x,0}}{ \tilde{\tau}_\varphi <  \tilde{\tau}_H \wedge T} \leqslant
\e^{-H/\pth{C_1 \sigma^2}}
\end{equation}
holds for all $T, \sigma, H>0$ \added{and all $x$ with  $\norm{x}<H$}.
\end{proposition}

The following result bounds the probability to escape from one of the metastable
neighbourhood~$B_i$.

\begin{proposition} 
\label{prop:exit_Mk} 
There exist $C>0$ and $\kappa >0$ such that for all $x \in B_i \subset
\meta{k}$, 
\begin{equation}
\Prcx{x}{X_1 \notin \meta{k} } \leqslant C \e^{-\kappa \abs{B_i}/ \sigma^2}\; ,
\end{equation}
where $\abs{B_i}$ denotes the radius of the ball $B_i$.
\end{proposition}
\begin{proof}
We introduce the continuous stopping time $\tilde{\tau}_\Sigma = \inf\setsuch{t
>0}{\varphi_t >1}$ corresponding to the first return time to the Poincar\'{e}
map. Then for any initial condition $(x_0,0)$ with $x_0 \in B_i$,  
\begin{equation}
\added{\Prcx{x_0}{X_1 \notin \meta{k} } 
 = \bigprob{ \norm{x^{(x_0,0)}_{\tilde{\tau}_\Sigma}} 
 > \abs{B_i}}\; .}
 \label{Eq:SamplePathBSigma}
\end{equation}
\added{Introducing a second sample path starting on the $i^{\text{th}}$ stable 
periodic at time 0, i.e. started in $(0,0)$, and driven by the same Brownian
 motion, we can use an upper bound on the probability that the two sample 
paths do not approach each other exponentially fast to
bound~\eqref{Eq:SamplePathBSigma}. Indeed, for
$\varrho\in(0,1)$,}
\begin{align}
\nonumber
\added{ \bigprob{ \norm{x^{(x_0,0)}_{\tilde{\tau}_\Sigma}} > \abs{B_i}}
\leqslant} 
&\ \added{
 \bigprob{ \norm{x^{(0,0)}_{\tilde{\tau}_\Sigma}} > (1-\varrho) \abs{B_i}}}\\
 &{}\added{{}+   \bigprob{ \norm{x^{(x_0,0)}_{\tilde{\tau}_\Sigma} -
x^{(0,0)}_{\tilde{\tau}_\Sigma}  }
 > \varrho \abs{B_i}}\; .}
 \label{eq:DecomposedPbaLeaveBi}
 \end{align}
\added{Adapting~\cite[Proposition 6.12]{berglund2014noise}, we obtain the
existence of $c_0>0$ and $\varrho<1$ such that the second term on the right-hand
side, corresponding to the difference of the two sample paths, is bounded by} 
 \begin{equation}
\added{ \bigprob{ \norm{x^{(x_0,0)}_{\tilde{\tau}_\Sigma} - 
x^{(0,0)}_{\tilde{\tau}_\Sigma}  } > \varrho \abs{B_i}} \leqslant \e^{-c_0
\abs{B_i}/\sigma^2}\; .}
 \end{equation}
\added{In order to apply Proposition \ref{PropAlongTheDiffusion} to bound the 
first term on the right-hand side of~\eqref{eq:DecomposedPbaLeaveBi}, we
decompose}
\begin{align}
\nonumber
\added{\bigprob{
\norm{x^{(0,0)}_{\tilde{\tau}_\Sigma}} > (1- \varrho) \abs{B_i} } \leqslant}&\
\added{ \bigprob{
\norm{x^{(0,0)}_{\tilde{\tau}_\Sigma}} > (1- \varrho) \abs{B_i} , \tilde{\tau}_\varphi > 
\tilde{\tau}_{H} \wedge T}}\\
 &{}\added{{}+ \Prcx{\pth{0,0}}{ \tilde{\tau}_\varphi < 
\tilde{\tau}_{H} \wedge  T}\;,}
\label{Eq:BoundExitBi}
\end{align}
\added{where we will choose $T = 2 T_i$ and $H = (1- \varrho) \abs{B_i}$.  }
\added{Note that the solution of \eqref{EqSystemeCompletRPhiStable} with initial
condition $(0,0)$ can be written as}
\begin{equation}
\added{
x^{(0,0)}_t = \int_0^{t} \e^{-\Lambda \pth{t -s}} b_x\bigpar{x^{(0,0)}_s,
\varphi^{(0,0)}_s} \6s + \sigma \int_0^t \e^{-\Lambda \pth{t-s}}
g_x\bigpar{x^{(0,0)}_s, \varphi^{(0,0)}_s} \6W_s\; .}
\end{equation}
\added{To bound the first term on the right-hand side of~\eqref{Eq:BoundExitBi},
observe that on $\set{\tilde{\tau}_\varphi > \tilde{\tau}_{H} \wedge T}$ we
have $\added{\tilde{\tau}_{\Sigma}}< 2 T_i$ and thus }
\begin{equation}
\added{\bigprob{ \norm{x^{(0,0)}_{\tilde{\tau}_\Sigma}} >
 (1- \varrho) \abs{B_i} , \tilde{\tau}_\varphi
>  \tilde{\tau}_{H} \wedge T}  \leqslant
\biggprob{\supSur{0 \leqslant s \leqslant 2 T_i}{ \norm{x^{(0,0)}_s }
> (1- \varrho) \abs{B_i} }}\;.}
\end{equation}
\added{Using a Bernstein inequality and a partition of the interval $\brak{0,2
T_i}$, as in~\cite[Theorem~5.1.18]{Berglund_Gentz_book}
or~\cite[Proposition~3.3]{Berglund_Gentz_Kuehn_2015}, 
we can show that there exist
  $C_0, \kappa_0 >0$ such that}
\begin{equation}
\added{\biggprob{\supSur{0 \leqslant s \leqslant 2 T_i}{ \norm{x^{(0,0)}_s }
\geqslant (1- \varrho) \abs{B_i}   }} \leqslant  C_0 \e^{-\kappa_0  \abs{B_i} / \sigma^2}\; .}
\end{equation}
Using Proposition \ref{PropAlongTheDiffusion} to bound the second term
on the right-hand side of \eqref{Eq:BoundExitBi}\added{,} we obtain the result.
\end{proof}

We also need to bound the probability of staying close to an unstable periodic
orbit. Let \added{ $\cU \subset \Sigma$ be a union of neighbourhoods of size $\delta$
of the unstable periodic orbits, with $\delta$ of order $1$,  and 
let $\cS \subset \cU$ be a union of  neighbourhoods of size $h=\sigma^{3/4}$ of the
unstable periodic orbits on the Poincar\'{e} section.}

\begin{proposition}
Let $h = \sigma^{3/4}$ and  
$ \tilde{\tau}_{\cS^c} = \inf \braces{t>0 : \norm{x_t} = h }$. 
There exists a constant $C_2$ such that for any $x$ such that $\norm{x} < h$ 
and $0<T \leqslant1/h$, 
\begin{equation}
\Prcx{\pth{x,0}}{ \tilde{\tau}_{\cS^c} >T  ,  \tilde{\tau}_\varphi > 
\tilde{\tau}_{\added{\cS^c}} \wedge T}  \leqslant C_2 \sigma^{1/2}\; .
\label{Eq:ProbaContLow}
\end{equation}
\end{proposition}
\begin{proof}
We introduce  the stopping time
\begin{equation}
\tilde{\tau}_{h^+} = \inf \braces{t>0 : \norm{x^+_t} = h }
 \end{equation}			
 where $x^+$ corresponds to the coordinates with positive Lyapunov exponents. Note that
 \begin{equation}
\Prcx{\pth{x,0}}{ \tilde{\tau}_{\cS^c} >T  ,  \tilde{\tau}_\varphi > 
\tilde{\tau}_{\added{\cS^c}}\wedge T}  \leqslant \Prcx{\pth{x,0}}{ \tilde{\tau}_{h^+} >T  , 
\tilde{\tau}_\varphi >  \tilde{\tau}_{\added{\cS^c}} \wedge T}  \; .
\end{equation}
On $\braces{\tilde{\tau}_\varphi >  \tilde{\tau}_{\cS^c} \wedge T}$, $\varphi_t$
is close to $t/T_i$, hence the equation for $x^+_t$ can be written 
\begin{equation}
\6 x^+_t = \pth{ \Lambda^+ x^+_t + b_{x^+}\pth{x_t, \varphi_t} } \6t + \sigma
\pth{ g_0(t) +  g_1\pth{x_t,\varphi_t,t} }\6W_t
\end{equation}
where $g_0(t) =  g_{x^+}\pth{0,t/T_i}$ and $g_1 = \Order{ \norm{x} + h}$. The solution can be expressed as
\begin{multline}
x^+_t = \e^{\Lambda^+ t} \biggl\{ \sigma \int_0^t \e^{-\Lambda^+ s} g_0(s) \6
W_s + \sigma \int_0^t \e^{-\Lambda^+ s} g_1\pth{x_s,\varphi_s,s} \6W_s \\+
\int_{0}^t  \e^{-\Lambda^+ s}  b_{x^+}\pth{x_s, \varphi_s}  \6s \biggr\} \; .
\end{multline}
The proof is then similar to \cite[Theorem 3.2.2]{Berglund_Gentz_book}.
\end{proof}

The following proposition will allow us to extend the previous estimate to an
exit from the larger set $\cU$. We denote $\cU \bs \cS$ by $\cK$. 

\begin{proposition}
\label{prop:tauKc} 
Let $\tilde{\tau}_{\cK^c} = \inf\setsuch{t>0}{x_t \notin \cK}$. There exists
a constant $\kappa_2>0$ such that for any initial condition $\pth{x,0} \in
\cK$, 
\begin{equation}
\label{eq:tauKc} 
\Prcx{(x,0)}{ \tilde{\tau}_{\cK^c} > t } \leqslant \e^{-\kappa_2 t / 
\log(\sigma^{-1})} \; .
\end{equation}
Furthermore, if $\tilde{\tau}_{\cU^c} = \inf\setsuch{t>0}{x_t \notin \cU}$
and $\tilde{\tau}_{\cS} = \inf\setsuch{t>0}{x_t \in \cS}$, for any $T_0>0$ 
there exists a constant $\kappa_3>0$ such that 
\begin{equation}
\label{eq:prob_tauSU} 
 \Prcx{(x,0)}{ T_0 \leqs \tilde{\tau}_{\cS} < \tilde\tau_{\cU^c} }
 \leqs \e^{-\kappa_3/\sigma^{1/2}}\;.
\end{equation} 
\end{proposition}
\begin{proof}
First, note that $\braces{\tilde{\tau}_{\cK^c} > t}\subset \braces{\norm{x^+_T}
< \delta}$.
Assume that the unstable periodic orbit admits $m^+$ positive Lyapunov exponents. We introduce the Lyapunov function
\begin{equation}
U_t = \sum_{i=1}^{m^+} (x^+_{t,i})^2 \; .
\end{equation}
Applying It\^o's formula  we obtain
\begin{equation}
\6U_t = \Bigset{ \sum_{i=1}^{m^+} \lambda^+_{i} (x^+_{t,i})^2 + \beta\pth{x_t,
\varphi_t}} \6t + \sigma  \sum_{i=1}^{m^+} g_{x,i}(x_t, \varphi_t) \6W_t^{i}\;
\end{equation}
where $\beta\pth{x_t, \varphi_t} \leqslant M \pth{ (U_t)^{3/2} + \sigma^2}$.
The proof is then similar to the proof of \cite[Proposition
D.4]{BerglundGentzKuehn2010}. Indeed, the drift term is bounded below by a
constant times $U_t$, and  $\braces{\norm{x^+_T} < \delta} \subset \braces{ U_T 
< m^+ \delta^2 /2 } $.
Using an endpoint estimate and the Markov property to restart the process at times
which are multiples of $\log(\sigma^{-1})$, we obtain~\eqref{eq:tauKc}. The
estimate~\eqref{eq:prob_tauSU} is obtained by bounding the probability that
$U_t$ leaves a neighbourhood of size $\sigma^{3/4}$ around an exponentially
growing term, similarly to \cite[Proposition D.7]{BerglundGentzKuehn2010}.
\end{proof}

\begin{proposition}
\label{prop:committor_U_MN} 
For all  $x \in {\pth{\cU \cup \cM_N}}^c$, there exist constants $C_1,
\kappa_1>0$ such that
\begin{equation}
\Prcx{x}{\tau^+_{\cU} < \tau^+_{\cM_N}} \leqslant C_1 \e^{-\kappa_1/\sigma^2}\;.
\end{equation}
\end{proposition}
\begin{proof}
Consider a deterministic solution $z_t^\text{det}= \pth{x_t^\text{det},
\varphi_t^{\text{det}}}$ with initial condition $z_0 =(x, 0) $. Since
$\partial\cU$ is at distance of order $1$ of any unstable periodic orbit, and
because the stable periodic orbits are the only attractive limit sets
(Assumption \ref{ass:limit_sets}), $z_t^\text{det}$ will reach a neighbourhood 
of a stable periodic orbit  in a time $T$ of order 1. Using \cite[Theorem
5.1.18]{Berglund_Gentz_book},
it follows that for $t\geqslant 0$,
\begin{equation}
\Prcx{\pth{x,0}}{ \supSur{0\leqslant s \leqslant t}{\norm{z_{\added{s}} -
z_{\added{s}}^{\text{det}}}} > h_0} \leqslant C_0 (1+t) \e^{-\kappa_0 h_0^2 / \sigma^2}
\label{Eq:DiffDeterministic}
\end{equation}
for some constants $C_0, \kappa_0>0$. Note that the estimate holds for $h_0
\leqslant h_1 / \chi(t)$, where $h_1$ is another constant and $\chi(t)$ is
related to the local Lyapunov exponent of $z_t^\text{det}$. Since
$z_t^{\text{det}}$ is attracted by the stable orbit, there exists $M_0>0$ such
that $\chi(T) \leqslant 1+ M_0 T$. Applying \eqref{Eq:DiffDeterministic} with
$h_0 = \abs{B_i}/2$, we find that any sample path which does not reach $\cU$
before time $T$  will hit $B_i$ with high probability.
\end{proof}

\subsection{Mean return time estimates}
\label{ssec:mean_return_time} 

The following two lemmas are useful to bound expectations of first return
times.

\begin{lemma}
\label{Lemma:GenExpEst}
For any $A \subset \Sigma$, $n_0 \in \N$ and $x \in \Sigma$, the
expectation of the first return time to $A$ satisfies
\begin{equation}
\Espc{x}{\tau^{+}_{A}} \leqslant \frac{n_0 \Prcx{x}{\tau^{+}_{A} \geqslant n_0}}{1- \Prcx{A^c}{\tau^{+}_{A}\geqslant n_0}}\; .
\end{equation}
\end{lemma}

\begin{proof}
Using the Markov property, we decompose the expectation as
\begin{align}
\nonumber
\Espc{x}{\tau^{+}_{A} } &= \sum_{i \geqslant 0} \sum_{n=1}^{n_0} \Prcx{x}{\tau^{+}_{A} \geqslant i n_0 + n}\\
\nonumber
&\leqslant n_0  \sum_{i \geqslant 0} \Prcx{x}{\tau^{+}_{A} \geqslant (i+1) n_0} 
\\
&\leqslant n_0 \sum_{i \geqslant 0} \Prcx{x}{\tau^{+}_{A} \geqslant n_0}
{\pth{\Prcx{A^c}{\tau^{+}_{A} \geqslant n_0}}}^{i} \;
\end{align}
which gives the result by summing a geometric series.
\end{proof}

The next lemma is inspired by results in \cite{betz2016multi}.

\begin{lemma}
\label{Lemma:ExpSplittingSet}
For any $A,B,C \subset \Sigma$, 
\begin{equation}
\Espc{A}{\tau^{+}_B} \leqslant \Espc{A}{\tau^{+}_{B \cup C}} + \Prcx{A}{\tau^{+}_C < \tau^{+}_B} \Espc{C}{\tau^{+}_B}\; .
\end{equation}
\end{lemma}

\begin{proof}
Splitting the expectation according to the event $ \braces{\tau^{+}_B <
\tau^{+}_C}$ or $\braces{\tau^{+}_C < \tau^{+}_B}$ and then using the strong
Markov property, we obtain 
\begin{align}
\nonumber
\Espc{x}{\tau^{+}_B} &=  \Espc{x}{\tau^{+}_B \ind{\tau^{+}_B < \tau^{+}_C}}  +
\Espc{x}{\tau^{+}_B \ind{\tau^{+}_C < \tau^{+}_B}} \\
\nonumber
&=  \Espc{x}{\tau^{+}_B \ind{\tau^{+}_B < \tau^{+}_C}}  +
\Espc{x}{\added{\bigl[} (\tau^{+}_B-\tau^{+}_C) + \tau^{+}_C  \added{\bigr]}
\ind{\tau^{+}_C < \tau^{+}_B}} \\
\nonumber
&= \Espc{x}{\tau^{+}_{B \cup C}} + \Espc{x}{(\tau^{+}_B-\tau^{+}_C) 
\ind{\tau^{+}_C < \tau^{+}_B}}\\
& \leqslant \Espc{x}{\tau^{+}_{B \cup C}} + \Prcx{A}{\tau^{+}_C < \tau^{+}_B} \Espc{C}{\tau^{+}_B}\; ,
\end{align}
which gives the result by taking the supremum for $ x \in A$.
\end{proof}

\begin{corollary}
\label{Corrolary:ExpectedValueMN}
For $\cU$ as defined in Section~\ref{Ssection:GeneralEst},
\begin{equation}
\label{eq:cor_expectation} 
\bigexpecin{\cM_N^c}{\tau^{+}_{\cM_N} } \leqslant \frac{
\Espc{\cU}{\tau^{+}_{\cU^c}} + \bigexpecin{{\pth{\cU \cup  \cM_N}
}^c}{\tau^{+}_{\cU \cup \cM_N}} }{1-\bigprobin{{\pth{\cU \cup
\cM_N}}^c}{\tau^{+}_{\cU} < \tau^{+}_{\cM_{N}}}}\; .
\end{equation}
\end{corollary}
\begin{proof}
For all $x \in \cM_N^c$, 
\begin{equation}
\bigexpecin{x}{\tau^{+}_{\cM_N}} \leqslant \max\Bigset{
\bigexpecin{\cU}{\tau^{+}_{{\cM}_N}} , \bigexpecin{\added{(\cU \cup
\cM_N)}^c}{\tau^{+}_{{\cM}_N}} }\; .
\end{equation}
Applying Lemma \ref{Lemma:ExpSplittingSet} with $A = \cU$, $B =\cM_{N}$ and $C
= (\cU \cup \cM_{N})^c$,  we obtain
\begin{equation}
\bigexpecin{\cU}{\tau^{+}_{{\cM}_N}}  \leqslant \bigexpecin{\cU}{
\tau^{+}_{\cU^c}  } + 
 \bigexpecin{{\pth{\cU \cup \cM_N}}^c}{\tau^{+}_{{\cM}_N}}\; ,
\end{equation}
whereas taking $A = (\cU\cup\cM_N)^c$, $B =\cM_{N}$ and $C
= \cU$,  we get
\begin{equation}
\bigexpecin{{\pth{\cU \cup \cM_N}}^c }{\tau^{+}_{{\cM}_N } }  \leqslant
\bigexpecin{{\pth{\cU \cup \cM_N}}^c }{ \tau^{+}_{\cM_N \cup \ \cU}  } +
\bigprobin{ {\pth{\cU \cup \cM_N}}^c }{{ \tau^{+}_{\cU}} < \tau^{+}_{\cM_N}  }
\bigexpecin{\cU}{\tau^{+}_{{\cM}_N}}\; .
\end{equation}
Combining these two bounds, we obtain
\begin{equation}
\bigexpecin{\cU} {\tau^{+}_{{\cM}_N}}  \leqslant \bigexpecin{\cU}{
\tau^{+}_{\cU^c}  } + 
\bigexpecin{{\pth{\cU \cup \cM_N}}^c }{ \tau^{+}_{\cM_N \cup \ \cU}  } +
\bigprobin{ {\pth{\cU \cup \cM_N}}^c }{{ \tau^{+}_{\cU}} < \tau^{+}_{\cM_N}  }
\bigexpecin{\cU}{\tau^{+}_{{\cM}_N}}\; ,
\end{equation}
which yields~\eqref{eq:cor_expectation}.
\end{proof}

In order to bound the expected value $\Espc{\cU}{\tau^{+}_{\cU^c}} $, we will
again use Lemma \ref{Lemma:ExpSplittingSet} with two neighbourhoods of an
unstable periodic orbit. First, we show that the sample paths are likely to
leave the small neighbourhood $\cS$ of the unstable periodic orbit (of size
$h=\sigma^{3/4}$), then as soon as paths have left $\cS$, the drift term will
make it easier to escape from the larger neighbourhood $\cU$. Using similar
arguments as in the proof of Corollary \ref{Corrolary:ExpectedValueMN}, we
obtain the following result.

\begin{lemma}
For $\cS\subset\cU$ and $\cK=\cU\setminus\cS$, as defined in
Section~\ref{Ssection:GeneralEst},
%
and all $x \in \cU$,
\begin{equation}
\Espc{x}{ \tau^{+}_{\cU^c}} \leqslant \frac{\Espc{\cK  }{\tau^{+}_{{\cK}^c}} +
\Espc{\cS }{\tau^{+}_{\cS^c}} }{1- \Prcx{\cK}{ \tau^{+}_{\cS}  < 
\tau^{+}_{\cU^c}}}\;.
\label{Eq:ExpectedValueLeave}
\end{equation}
\end{lemma}

The different expected values involved in \eqref{Eq:ExpectedValueLeave} will be
bounded using Lemma \ref{Lemma:GenExpEst} and results from
Section~\ref{Ssection:GeneralEst}.

\begin{proposition}
There exist constants $M_1, \kappa>0$ such that 
\begin{align}
\nonumber
\Espc{\cS}{\tau^{+}_{\cS^c}} &\leqslant M_1 \sigma^{1/2}\;,\\
\nonumber
\Espc{\cK}{\tau^{+}_{\cK^c}} &\leqslant M_1 \log(\sigma^{-1})\;,\\
\Prcx{\cK}{\tau^{+}_{\cS} < \tau^{+}_{\cU^c}} &\leqslant
\e^{-\kappa/\sigma^{1/2}}\;.
\end{align}
\end{proposition}
\begin{proof}
Recall that 
$\tilde{\tau}_{\cS^c} = \braces{\inf t > 0 : \norm{x_t} \geqslant h}$, 
and let $n_0> T/T_i+\sigma^{3/4}$, with $0<T \leqslant {1}/{\sigma^{3/4}}$.
For all $x \in \cS$,
\begin{align}
\nonumber
\Prcx{x}{\tau^{+}_{\cS^c} > n_0} 
\leqslant{}&  \Prcx{(x,0)}{{ \tilde{\tau}_\varphi
<\tilde{\tau}_{\cS^c} \wedge T}   } +  \Prcx{(x,0)}{ \tilde{\tau}_\varphi
>\tilde{\tau}_{\cS^c} \wedge T, \tau^{+}_{\cS^c} > n_0 }\\
\nonumber
\leqslant{}& \Prcx{(x,0)}{{ \tilde{\tau}_\varphi
<\tilde{\tau}_{\cS^c} \wedge T}   } +  \Prcx{(x,0)}{\tilde{\tau}_\varphi
>\tilde{\tau}_{\cS^c} \wedge T ,
\tau^{+}_{\cS^c} > n_0 , \tilde{\tau}_{\cS^c} < T } \\
&{}+ \Prcx{(x,0)}{\tilde{\tau}_\varphi >\tilde{\tau}_{\cS^c} \wedge T,
\tilde{\tau}_{\cS^c}>T} \;.
\end{align}
However since $n_0> T/T_i+\sigma^{3/4}$,
\begin{align}
\added{\Prcx{(x,0)}{\tilde{\tau}_\varphi
>\tilde{\tau}_{\cS^c} \wedge T ,
\tau^{+}_{\cS^c} > n_0 , \tilde{\tau}_{\cS^c} < T }}=0 \; .
\end{align}
We obtain the bound on $\bigexpecin{\cS}{\tau^{+}_{\cS^c}}$
using~\eqref{Eq:ProbaContLow} and applying Lemma
\ref{Lemma:GenExpEst}. The two other bounds follow in a similar way, using
Proposition~\ref{prop:tauKc}. 
\end{proof}

Combining the last three results with Proposition~\ref{prop:committor_U_MN}, we
immediately get:

\begin{corollary}
\label{cor:expecMN} 
There exists a constant $M_2>0$ such that 
\begin{equation}
\label{eq:expec_tauMN} 
\bigexpecin{\cM_N^c}{\tau^{+}_{\cM_N}} \leqslant M_2 \log(\sigma^{-1})\;.
\end{equation} 
\end{corollary}

We are now going to estimate the mean return time
$\bigexpecin{x}{\tau^+_{\meta{k}}}$ for $x \in \meta{k}$. This estimate is
needed to bound the norm of the difference between $K^u$  and  $K^\star$  in
Proposition~\ref{Prop:NormKuKZIterated}.
By decreasing induction on $k$ ($1\leqslant k \leqslant N$), we can prove that
for all $x \in \cM_{k}$, the expectation  $\bigexpecin{x}{\tau^+_{\meta{k}}}$ is
exponentially close to one. We start by estimating the expectation of the first
return time to $\cM_N$.
\begin{lemma}
For all $x \in \cM_{N}$, 
\begin{equation}
\bigexpecin{x}{\tau^{+}_{\cM_N}-1} \leqslant\Prcx{x}{X_1 \notin \cM_{N}} 
\bigexpecin{\cM_N^c}{\tau^{+}_{\cM_N}} 
\; .
\end{equation} 
\end{lemma}
\begin{proof}
Splitting the expectation according to the location of $X_1$, we have 
\begin{equation}
\bigexpecin{x}{\tau^{+}_{\cM_N} } \leqslant 1+  \Prcx{x}{X_1 \notin \cM_{N}} 
\bigexpecin{\cM_N^c}{\tau^{+}_{\cM_N}}  \; .
\end{equation}
\end{proof}

\begin{lemma}
For all $k <N$, for all $x \in B_i \subset \cM_{k}$,
\begin{equation}
\bigexpecin{x}{\tau^{+}_{\cM_k}} \leqslant \bigexpecin{x}{\tau^{+}_{\cM_{k+1}}}
+ \bigprobin{x}{\tau^{+}_{B_{k+1}} < \tau^{+}_{\cM_{k}}}  \bigexpecin{B_{k+1}}{
\tau^{+}_{\cM_{k}}} \; .
\end{equation}
\end{lemma}
\begin{proof}
The proof is a direct application of Lemma \ref{Lemma:ExpSplittingSet} with $A =
B_i$, $B = \cM_k $ and $C  =  B_{k+1}$. 
\end{proof}

Combining the last two lemmas with Corollary~\ref{cor:expecMN} and
Proposition~\ref{prop:exit_Mk} shows that, as announced,
$\bigexpecin{\cM_k}{\tau^{+}_{\cM_k}} = 1+\Order{\e^{-\kappa/\sigma^2}}$ for
all $k$, where $\kappa\added{{}>0}$ is proportional to the size of the
neighbourhood $B_i$.

\subsection{Coupling argument}
\label{ssec:distance_trajectories} 

In order to apply the coupling argument in Proposition~\ref{prop:coupling}, we
need to estimate the probability that two trajectories $\pth{X_n^{x_1}}_n$ and
$\pth{X_n^{x_2}}_n$ driven by the same realization of the Brownian motion drift
apart, i.e., their difference leaves a contracting \lq\lq layer\rq\rq. 

\begin{proposition}[{\cite[Proposition~6.12]{berglund2014noise}}]
There exist $C, \kappa>0$ and $\varrho<1$, independent of
$\sigma$ such that for $x_1, x_2 \in B_i$, 
\begin{equation}
\Pba{ \norm{X_n^{x_1} - X_n^{x_2}}  > \varrho^n \norm{x_1 - x_2} }\leqslant C
\e^{- \kappa/\sigma^2}\; .
\end{equation}
\end{proposition}

The proof is a straightforward generalisation of the proof of \cite[Proposition
6.12]{berglund2014noise} to the multidimensional case. As explained
in~\cite[Section~6.3]{berglund2014noise}, it follows that the stopping time $N$
introduced in~\eqref{eq:def_N} satisfies $\prob{N>n_0} \leqs
n_0\e^{-\kappa/\sigma^2}$ for an $n_0$ of order $\log(\sigma^{-1})$. Using the
Markov property at multiple times of $n_0$, if follows that 
\begin{equation}
 \rho_{kn_0} = 
 \prob{N>kn_0} \leqs \bigpar{M\log(\sigma^{-1})\e^{-\kappa/\sigma^2}}^k\;.
\end{equation} 
Choosing $k$ such that $k\kappa > C+1$ in~\eqref{eq:rho_n}, we obtain a
constant $L(n)$ close to $1$. 


\subsection{Miscellaneous a priori bounds}
\label{ssec:apriori}

\begin{proof}[{\sc Proof of Proposition~\ref{prop:tau_sigma}}]
 If the initial condition $z$ lies within the basin of attraction of one of the
stable periodic orbits, the same argument as in
Proposition~\ref{prop:committor_U_MN} shows that $Z_t$ will reach $\Sigma$ in a
time of order $1$ with high probability, so that $\bigprobin{z}{\tau_\Sigma >
2T}$ is exponentially small. If $z$ belongs to the neighbourhood of an unstable
periodic orbit, the results from Section~\ref{ssec:mean_return_time} show that
$Z_t$ will leave this neighbourhood in a mean time of order
$\log(\sigma^{-1})$. A similar result holds \added{if} $z$ belongs to the neighbourhood
of an unstable equilibrium point, as shown
in~\added{\cite{Kifer81,Bakhtin_2008,Almada_Bakhtin_2011}}. Combining this with
the strong Markov property \added{and~\eqref{eq:tau_D}} yields the result. 
\end{proof}

\begin{proof}[{\sc Proof of Proposition~\ref{Prop:LargeDeviation}}]
For two points $x, y\in\Sigma$, the continuous-time large-deviation principle
naturally induces a discrete-time large-deviation principle with rate function
\begin{equation}
 J(x,y) = \inf_{T>0} \inf_{\gamma:(x,0)\to(y,1)} I_{[0,T]}(\gamma)\;,
\end{equation} 
where the notation $\gamma:(x,0)\to(y,1)$ implies that we consider trajectories
visiting $\Sigma'$ between the points $x$ and $y$ (this can be viewed as an
instance of the contraction principle). More generally, for any sequence
$(x_0,\dots,x_n)$ of points in $\Sigma$, the rate function is given by 
$ J(x_0,\dots,x_n) = \sum_{j=0}^{n-1} J(x_j,x_{j+1})$.
The fact that $V(x^\star_i,x^\star_j) = H(i,j)$ implies that for any $\eta>0$,
there exists a $T>0$ and a continuous-time trajectory $\gamma$ connecting the
two periodic orbits in time $T$ such that 
\begin{equation}
 I_{[0,T]}(\gamma) \leqs H(i,j) + \frac{\eta}2\;.
\end{equation} 
Enlarging $T$ if needed, one can assume that $\gamma$ starts and ends on
$\Sigma$, since one can follow the deterministic flow at zero cost. Furthermore,
there exists $\delta>0$ such that if the neighbourhood $B_i$, $B_j$ have radius
$\delta$, they can be connected by a trajectory $\gamma$ such that
$I_{[0,T]}(\gamma) \leqs H(i,j) + \eta$. We may assume that $\gamma$ intersects
$B_i\added{{}\cup{}} B_j$ only at its endpoints, for otherwise there would exist a cheaper
way to connect the neighbourhoods. Therefore, there exists $n\geqs1$ and points $x_0\in
B_i, x_1, \dots, x_{n-1} \notin B_i\cup B_j, x_n\in B_j$, defined by the
successive intersections of $\gamma$ with $\Sigma$, such that 
\begin{equation}
 J(x_0,\dots,x_n) \leqs H(i,j) + \eta\;.
\end{equation} 
On the other hand, for any $\eta>0$, there exists a neighbourhood of radius $\delta>0$ such
that for any $x\in B_i$ and $y\in B_j$, $V(x,y) \geqs H(i,j) - \eta$. A
similar argument as above shows that any discrete-time trajectory connecting
the neighbourhoods must also have a cost larger than $H(i,j) - \eta$. 
\end{proof}

\section{Last steps of the proofs}
\label{sec:lastSteps} 
\subsection{Proof of Theorem~\ref{thm:eigenvalues}}

Fix a small constant $\eta>0$. We start by estimating the $k^{\text{th}}$
eigenvalue $\lambda_{k-1}$ of $K$, by showing that it is close to the
$k^{\text{th}}$ eigenvalue $\lambda^\star_{k-1}$
of the finite rank kernel $K^\star$, estimated in
Corollary~\ref{CorEigenvalueProp}.

As discussed in Section~\ref{ssec:distance_trajectories}, we can find an $n$ of
order $\log(\sigma^{-1})$ such that each kernel $K^0_{B_i}$ satisfies the
uniform positivity condition~\eqref{eq:uniform_positivity}, with $L(n)-1$ an
arbitrary positive constant of order $1$. Then
Proposition~\ref{prop:spectral_gap} shows the existence of a constant $c_0>0$
such that 
\begin{equation}
 \bigabs{\mathring{\ev}^{B_i}_1} \leqs \e^{-c_0/\log(\sigma^{-1})}\;.
\end{equation} 
Proposition~\ref{prop:oscillation_phi0}, \eqref{eq:lambda0_qed} and the
large-deviation estimate in Proposition~\ref{Prop:LargeDeviation} yield the
bound 
\begin{equation}
 \norm{\mathring{\phi}^{B_i}_0 - 1} \leqs M_0\log(\sigma^{-1}) 
 \e^{-[H(i,M_k\setminus \set{i}) - \eta]/\sigma^2}
\end{equation} 
on the oscillation of the principal eigenfunction. Plugging this into
Proposition~\ref{Prop:NormK0KStark} and using 
Assumption~\ref{ass:hierarchy} to compare the various $H(i,j)$ yields 
\begin{equation}
 \norm{(K^0)^m - (K^\star)^m} 
 \leqs 2\e^{-mc_0/\log(\sigma^{-1})} 
 + \bigbrak{M_0\log(\sigma^{-1}) + m^2 \e^{-H'_k/\sigma^2}}\e^{-H'_k/\sigma^2}
\end{equation} 
where $H'_k = H(k,M_{k-1}) - \eta$. Combining this with
Proposition~\ref{Prop:NormKuKZIteratedk} and the mean return time estimates in
Section~\ref{ssec:mean_return_time} shows that $\norm{(K^u)^m -
(K^\star)^m}$ is bounded by 
\begin{align}
\nonumber
 \Delta_m ={}& 2\e^{-mc_0/\log(\sigma^{-1})} 
+ \bigbrak{M_0\log(\sigma^{-1}) + m^2 \e^{-H'_k/\sigma^2}}\e^{-H'_k/\sigma^2}\\
 &{}+ \bigpar{1 + 2(1-\e^{-u})\e^{-\theta'/\sigma^2}}^m - 1\;,
\end{align} 
provided $(1-\e^{-u})\e^{[H(k+1,M_k)+\eta]/\sigma^2} \leqs 1/2$. 
The argument given in Section~\ref{ssec:contour_general} shows that $(\Ku)^m$
admits a unique eigenvalue $\lambda_{k-1}^m$ inside the contour $\cC$ of radius
$c_2\Delta_m$ centred in $(\lambda^\star_{k-1})^m$ (for a $c_2$ of order $1$),
and that 
\begin{equation}
 \frac{1-\lambda_{k-1}^m}{1-(\lambda^\star_{k-1})^m}
 = 1 + \biggOrder{\frac{\Delta_m}{1-(\lambda^\star_{k-1})^m}}\;.
\end{equation} 
Note that this eigenvalue is necessarily real, since $(\Ku)^m$ is real and has
exactly one eigenvalue inside $\cC$. Using the fact that for any
$x\in(0,1)$ such that $m(1-x) < 2$, one has 
\begin{equation}
 (1-x) \Bigbrak{1 - \frac12 m(1-x)} \leqs \frac{1-x^m}{m} \leqs 1-x\;,
\end{equation} 
we obtain 
\begin{equation}
 \frac{1-\lambda_{k-1}}{1-\lambda^\star_{k-1}}
 = 1 + \bigOrder{m(1-\lambda^\star_{k-1})} + 
 \biggOrder{\frac{\Delta_m}{1-(\lambda^\star_{k-1})^m}}\;.
\end{equation} 
The optimal error term is obtained for
$m=\log(\sigma^{-1})\e^{(2\eta+\delta)/\sigma^2}$, with $\delta =
H(k,M_{k-1})/2$. Together with Corollary~\ref{CorEigenvalueProp}, this shows
that $\lambda_{k-1}$ satisfies~\eqref{eq:lambdak}. 

As discussed in Section~\ref{subSection:ChoiceSetA}, applying this argument 
to the kernels $K^{u,(k)}$ for $k=1,\dots N$ shows that $K^{u,(N)}$ has exactly
$N$ eigenvalues outside some disc centred in the origin. The
system~\eqref{eq:coupled} can then be used to show that the original kernel $K$
also has exactly $N$ eigenvalues outside this disc, satisfying the same
asymptotics. 

\begin{remark}
Strictly speaking, to justify this argument, we have to make sure that the
eigenvalues of the $K^{u,(k)}$ vary sufficiently slowly as functions of $u$.
This, however, is easy to obtain. Indeed, a standard perturbation argument shows
that if $K(u)$ is a family of linear operators depending differentiably on $u$,
and $\lambda$ is an isolated simple eigenvalue of $K(u_0)$ with left and right
eigenfunctions $\pi$ and $\phi$, then 
\begin{equation}
 \frac{\6\lambda}{\6u}(u_0) = \pi \frac{\6K}{\6u}(u_0)\phi\;.
\end{equation} 
In our case, the relevant derivative is given by 
\begin{equation}
 \frac{\6}{\6u} K^u(x,\6y) = \biggexpecin{x}{ \bigpar{\tau^+_{\cM_k}-1}
\e^{u (\tau^+_{\cM_{k}}-1)} \bigind{X_{\tau^+_{\cM_{k}}} \in \6y}}\;.
\end{equation} 
Proceeding as in Proposition~\ref{Prop:NormKuK0}, it is not hard to check that
the norm of this operator is of order $\bigexpecin{\cM_k}{\tau^+_{\cM_k}-1}$ for
$u$ as in the above computation. 
\end{remark}

To prove the spectral gap estimate~\eqref{eq:gap_lambdaN}, one can use the fact
that 
\begin{equation}
 \bigprobin{\cM_N^c}{X_m \in \cM_N^c} \leqs \frac12
\end{equation} 
for $m$ of order $\log(\sigma^{-1})$, as a consequence
of~\eqref{eq:expec_tauMN}, Proposition~\ref{prop:exit_Mk} and Markov's
inequality. If $(\tilde X_n)_{n\geqslant0} = (X_{mn})_{n\geqslant0}$ denotes the process diluted by a factor
$m$, then the Laplace transform of the first time $\tilde X_n$ hits $\cM_N$
exists for all $u$ such that $\abs{\e^{-u}} \geqs 1/2$. Therefore, by the above
argument, $K^m$ has exactly $N$ eigenvalues outside a disc of radius $1/2$,
which implies that $K$ has exactly $N$ eigenvalues outside a disc of radius
$\e^{-c_0/\log(\sigma)^{-1}}$. 

Finally, the result~\eqref{eq:principal_eigenvalue} on the principal eigenvalue
follows from the fact that the principal eigenfunction of the process killed
when hitting $\cM_{k-1}$ satisfies 
\begin{equation}
 \phi_0^{\cM_{k-1}^c}(x) =
\Bigexpecin{x}{\e^{u\tau_{B_k}}\phi_0^{\cM_{k-1}^c} \bigpar{X_{\tau_{B_k}}}}\;.
\end{equation} 
Therefore, it is also an eigenfunction of the kernel  
\begin{equation}
 K^u_{B_k}(x,\6y) = \biggexpecin{x}{\e^{u (\tau^+_{B_k}-1)}
\bigind{X_{\tau^+_{B_k}}  \in \6y, \tau^+_{B_k} < \tau^+_{\cM_{k-1}}}}\;.
\end{equation} 
This kernel can be approximated by 
\begin{equation}
K^{\star}_{B_k}\pth{x, \6y} =  
\int_{B_k} \mathring{\pi}^{B_i}_0\pth{z} \ K^0_{B_k}\pth{z, \6y} \6z
= \bigprobin{ \mathring{\pi}^{B_k}_0}{X_{\tau^+_{B_k}}  \in \6y, \tau^+_{B_k} <
\tau^+_{\cM_{k-1}}}\;,
\end{equation}
which is a rank $1$ operator, whose single nonzero eigenvalue is 
$\bigprobin{ \mathring{\pi}^{B_k}_0}{\tau^+_{B_k} <
\tau^+_{\cM_{k-1}}}$. The approximation arguments applied to $K^u$ and
$K^\star$ apply in this case as well, because the norm of the difference
$K^u_{B_k}-K^\star_{B_k}$ is trivially bounded above by the norm of the
difference $K^u - K^\star$.
\qed

\subsection{Proof of Theorem~\ref{thm:right_eigenfunctions}}

Recall that the $k^{\text{th}}$ eigenfunction $\phi^\star_{k-1}$ of $K^\star$
has been obtained in Lemma~\ref{lem:phikstar}, and that
$\norm{\phi^\star_{k-1}}=1$. In order to bound the difference between
$\phi_{k-1}$ and $\phi^\star_{k-1}$, we choose a countour $\cC$ around
$\lambda_{k-1}$ and consider the associated Riesz projector $\Pi_{\uGs}(\Ku)$
(cf.~\eqref{eq:riesz}). Since $\Pi_{\uGs}(\Ku)$ projects on the subspace
associated with $\lambda_{k-1}$, $\phi_{k-1}$ is given, up to multiplication by
a constant, by
\begin{equation}
 \phi_{k-1} = \Pi_{\uGs}(K^u)\phi^\star_{k-1}\;.
\end{equation} 
We also have the relation 
\begin{equation}
 \phi^\star_{k-1} = \Pi_{\uGs}(K^\star)\phi^\star_{k-1}\;,
\end{equation} 
where the Riesz projector $\Pi_{\uGs}(K^\star)$ is defined with the same
contour $\cC$. Taking the difference, it follows from
Proposition~\ref{Prop:InverseOperator} that 
\begin{equation}
 \norm{\phi_{k-1} - \phi^\star_{k-1}}
 \leqs C \norm{\Ku - K^\star}\;,
\end{equation} 
where $C$ is defined in~\eqref{Eq:DefC}, provided $\norm{\Ku - K^\star} <
\gamma/2$, cf.~\eqref{Eq:DefGamma}. An analogous bound holds for the
iterates $(\Ku)^m$ and $(K^\star)^m$, with a coutour around $\lambda_{k-1}^m$. 
Choosing $m$ as in the previous section, and a circular contour of radius
$(1-\lambda_{k-1}^m)/2$, one obtains 
\begin{equation}
 \norm{\phi_{k-1} - \phi^\star_{k-1}}
 = \Order{\e^{-\theta_{k-1}/\sigma^2}}\;,
\end{equation} 
where $\theta_{k-1}$ is $\eta$-close to $H(k,M_{k-1})/2$. 

Applying the Feynman--Kac relation of Proposition~\ref{prop:Feynman-Kac} with 
$\e^{-u} = \lambda_{k-1}$, we obtain 
\begin{equation}
 \e^{-u} \phi_{k-1}(x) 
 = \bigexpecin{x}{\phi_{k-1}(X_{\tau_{\cM_k}})} 
 + \bigexpecin{x}{(\e^{u(\tau_{\cM_k}-1)}-1) \phi_{k-1}(X_{\tau_{\cM_k}})}\;.
\end{equation}
By Proposition~\ref{Prop:NormKuK0}, the second term on the right-hand side has
order $\e^{-(H(k,M_{k-1})+\theta'-\eta)/\sigma^2}$. As for the first term, it
can be rewritten (recall that $\phi^\star_{k-1}$ is constant on each $B_j$)
\begin{equation}
 \sum_{j=1}^k \Bigexpecin{x}{\bigind{X_{\tau_{\cM_k}}\in B_j}
\phi_{k-1}\bigpar{X_{\tau_{\cM_k}}}}
= \sum_{j=1}^k \bigprobin{x}{\tau_{B_j} < \tau_{\cM_k\setminus B_j}}
\phi^\star_{k-1}(x^\star_j) + \Order{\e^{-\theta_{k-1}/\sigma^2}}\;.
\end{equation} 
To lowest order, \added{using Lemma~\ref{lem:phikstar} and Remark~\ref{rem:Sstar12}}, we have $\phi^\star_{k-1}(x^\star_j) = \delta_{jk} +
\smash{\Order{\e^{-\theta^-/\sigma^2}}}$, which
yields~\eqref{eq:right_eigenfunction}. The more precise
expression~\eqref{eq:right_eigenfunction_iterated} is based on the fact that
\begin{equation}
 \phi^\star(x^\star_j) = 
 - \frac{\bigprobin{\mathring{\pi}^{B_j}_0}{\tau^+_{B_k} <
 \tau^+_{\cM_{k-1}}}}{\bigprobin{\mathring{\pi}^{B_k}_0}{\tau^+_{\cM_{k-1}} <
 \tau^+_{B_k}}}
 + \Order{\e^{-2\theta^-/\sigma^2}}\;,
\end{equation}
as a consequence of Remark~\ref{rem:Sstar12}. 
As for the principal eigenfunction $\phi_0^{\cM_{k-1}^c}$, it satisfies 
\begin{align}
\nonumber
 \e^{-u} \phi_0^{\cM_{k-1}^c}(x) 
 ={}& \Bigexpecin{x}{\phi_0^{\cM_{k-1}^c} \bigpar{X_{\tau_{B_k}}}
 \normalind{\tau_{B_k} < \tau_{\cM_{k-1}}}} \\
 &{}+ \Bigexpecin{x}{\bigpar{\e^{u(\tau_{\cM_k}-1)}-1} \phi_0^{\cM_{k-1}^c} 
 \bigpar{X_{\tau_{B_k}}} \normalind{\tau_{B_k} < \tau_{\cM_{k-1}}}}\;,
\end{align}
where $\e^{-u} = \lambda_0^{\cM_{k-1}^c}$. 
The first term on the right-hand side is equal to 
\begin{equation}
\bigprobin{x}{\tau_{B_k} < \tau_{\cM_{k-1}}} \bigpar{ 1 +
\Order{\e^{-\theta_{k-1}/\sigma^2}}}\;,
\end{equation} 
while the second one can be bounded as above by
$\Order{\e^{-(H(k,M_{k-1})+\theta'-\eta)/\sigma^2}}$. 

\subsection{Proof of Theorem~\ref{thm:left_eigenfunctions}}

Using Proposition~\ref{prop:pi0} with $A_1=B_1$ and $A_2=\cM_N\setminus B_1$
and the large-deviation a priori bounds of
Proposition~\ref{Prop:LargeDeviation} shows that $\pi_0(\cM_N\setminus B_1)
\leqs \e^{-\theta^-/\sigma^2} \pi_0(B_1)$. Together with~\eqref{eq:pi0_MN}, this
proves~\eqref{eq:left_ef_1}. 

The bound~\eqref{eq:left_ef_2} can be proved by reasoning on the stationary
distribution of the Doob-transformed process $\bar X_{\cM_k^c}$ and using the
relation~\eqref{eq:Doob_eigenfunctions} between the left eigenfunctions of both
processes. 

In order to prove the first relation in~\eqref{eq:left_ef_3}, we use
Lemma~\ref{lem:left_eigenfunction}, showing that $\pi_{k-1}$ is a left
eigenfunction of the kernel $\Ku$, cf.~\eqref{eq:defKu}. Therefore we expect
$\pi_{k-1}$ to be close to the left eigenfunction $\pi^\star_{k-1}$ of
$K^\star$. Using the block-triangularisation of
Section~\ref{SubSect:LocaEigenvalues}, one easily obtains that 
\begin{equation}
 \pi^\star_{k-1} = (\hat\pi, 1-\hat\pi S^\star_{12})
 \qquad \text{where} \qquad 
 \hat\pi = (\alpha\id - T_{11})^{-1}\hat P_{21}\;,
\end{equation} 
which implies 
\begin{align}
\nonumber 
\pi^\star_{k-1}(B_k) &= 1 + \Order{\e^{-\theta^-/\sigma^2}}\;, \\
\pi^\star_{k-1}(B_j) &=
 -\frac{\bigprobin{\mathring{\pi}^{B_k}_0}{\tau^+_{B_j} <
 \tau^+_{\cM_k\setminus B_j}}}
 {\bigprobin{\mathring{\pi}^{B_k}_0}{\tau^+_ {\cM_{k-1}} <
 \tau^+_{B_k}}} \bigbrak{1 + \Order{\e^{-\theta^-/\sigma^2}}}
 && \text{for $1\leqs j\leqs k-1$\;.}
\end{align}
To compare $\pi_{k-1}$ and $\pi^\star_{k-1}$, it suffices to realise that the
$L^1$-operator norm of a kernel $K$, acting on signed measures, can be bounded
by $\sup_{x\in\cM_k} K(x,\cM_k)$. Therefore, the same bounds on
$\norm{\Ku-K^\star}$ and their iterates apply for the action of these operators
on signed measures, so that one can repeat the argument of the previous section
showing that 
\begin{equation}
 \bigabs{\pi_{k-1}(B_j) - \pi^\star_{k-1}(B_j)} =
\Order{\e^{-\theta_{k-1}/\sigma^2}}\;.
\end{equation} 
Finally, the second relation in~\eqref{eq:left_ef_3} is obained by comparing
the original and killed process monitored while \added{visiting} $\cM_j$. The
kernel of the original process can be approximated by a kernel $K^\star$ of
rank $j$, while the killed process is described by the restriction of this
kernel to $B_k\cup\dots\cup B_j$. Using a similar block-triangularisation as
in Section~\ref{SubSect:LocaEigenvalues}, with blocks of size $k-1$ and
$j-k+1$, the result follows easily. 
\qed

\subsection{Proof of Theorem~\ref{thm:expectations}}

The result will be proved if we manage to control the oscillation of
$\smash{\bigexpecin{x}{\tau^+_{\cM_{k-1}}}}$ when $x$ varies in $B_k$. To this
end, consider the process $(\hat X_n)_n$, killed when hitting $\cM_{k-1}$ and
monitored only while visiting $\cM_k$, whose kernel is $K^0_{B_k}$. If
$\hat\tau_{\cM_{k-1}}$ denotes the killing time of $\smash{\hat X_n}$, then we
have 
\begin{equation}
 \bigexpecin{x}{\hat\tau_{\cM_{k-1}}} 
 \leqs \bigexpecin{x}{\tau^+_{\cM_{k-1}}} 
 = \biggexpecin{x}{\sum_{n=0}^{\hat\tau-1} \bigexpecin{\hat X_n}{\tau_{\cM_k}}}
 \leqs \bigexpecin{x}{\hat\tau_{\cM_{k-1}}} \bigexpecin{B_k}{\tau_{\cM_k}}\;,
\end{equation} 
so that 
\begin{equation}
 1 \leqs
\frac{\bigexpecin{x}{\tau^+_{\cM_{k-1}}}}{\bigexpecin{x}{\hat\tau_{\cM_{k-1}}}}
 \leqs \bigexpecin{B_k}{\tau_{\cM_k}}\;.
\end{equation} 
It follows that 
\begin{equation}
\label{eq:oscillation_expectation} 
 \frac{\bigexpecin{B_k}{\tau^+_{\cM_{k-1}}}}
 {\added{\infSur{x \in B_k}{ \bigexpecin{x}{\tau^+_{\cM_{k-1}}}}}}
 \leqs  \frac{\bigexpecin{B_k}{\hat\tau_{\cM_{k-1}}}}
 {\added{\infSur{x \in B_k}{ \bigexpecin{x}{\hat\tau_{\cM_{k-1}}}}}}
\bigexpecin{B_k}{\tau_{\cM_k}}\;.
\end{equation} 
To control the oscillation of $\hat\tau_{\cM_{k-1}}$, we note that the spectral
decomposition~\eqref{Eq:SpectralDecompositionKZBI} yields 
\begin{align}
\nonumber
 \bigexpecin{x}{\hat\tau_{\cM_{k-1}}} 
 &= \sum_{n\geqs0} \bigpar{K^0_{B_k}}^n(x,B_k) \\
 &= \sum_{n\geqs0} \bigpar{\mathring{\ev}^{B_k}_0}^n \braces{ 
 \mathring{\phi}^{B_k}_0\pth{x}  +
\biggpar{\frac{\mathring{\ev}^{B_k}_1}{\mathring{\ev}^{B_k}_0}}^n
g^n\pth{x,B_k}}\;.
\end{align} 
We know that the kernel $K^0_{B_k}$ satisfies the uniform positivity
condition~\eqref{eq:uniform_positivity} with an $n_0$ of order
$\log(\sigma^{-1})$. It follows that 
\begin{equation}
 \bigexpecin{x}{\hat\tau_{\cM_{k-1}}}
 = \frac{1}{1-\mathring{\ev}^{B_k}_0}\mathring{\phi}^{B_k}_0\pth{x} 
 + \biggOrder{\frac{1}{1-\varrho^{1/n_0}\mathring{\ev}^{B_k}_0}}\;.
\end{equation} 
Together with Proposition~\ref{prop:oscillation_phi0}, this shows that
the oscillation of $\bigexpecin{x}{\hat\tau_{\cM_{k-1}}}$ is bounded by a term
of order $\smash{\log(\sigma^{-1})\e^{-(H(k,M_{k-1})-\eta)/\sigma^2}}$. Combined
with~\eqref{eq:oscillation_expectation}, this completes the proof. 
\qed

\appendix
\section{Doob's $h$-transform}
\label{app:Doob} 
Consider a Markov process $(X_n)_{n\geqs0}$ with state space $\Sigma$ and
transition kernel having density $k(x,y)$. Given a subset $A\subset\Sigma$, the
process conditioned on remaining in $A$ can be constructed using the functions 
\begin{equation}
 h_n(x) = \probin{x}{\tau_{A^c} > n}\;,
\end{equation} 
where $\tau_{A^c} = \inf\setsuch{n>0}{X_n \in A^c}$ denotes the first-exit time
from $A$. Indeed, \added{assuming  $h_n(x) > 0$ for all $x \in A$,} then
\added{for $y \in A$} we have 
\begin{equation}
 \bigpcondin{x}{X_1\in\6y}{\tau_{A^c} > n} 
 = \frac{1}{h_n(x)} \Bigexpecin{x}{\ind{X_1\in\6y}
\probin{y}{\tau_{A^c} > n-1}}
= \frac{h_{n-1}(y)}{h_n(x)} \probin{x}{X_1\in\6y}\;.
\end{equation}
This shows that the kernel
\begin{equation}
 \bar k_A(x,y;n) = \frac{h_{n-1}(y)}{h_n(x)} k(x,y) \ind{x\in A, y\in A}
\end{equation} 
describes the process conditioned to stay in $A$ up to time $n$. 
Thus if 
\begin{equation}
 \bar k_A(x,y) = \lim_{n\to\infty} \bar k_A(x,y;n)
\end{equation} 
exists, it will describe the process conditioned on staying in $A$ forever. 

Let $k_A(x,y) = k(x,y)\ind{x\in A,y\in A}$ denote the kernel of the process
killed upon leaving $A$, and write \added{$\lambda^A_i$} for its eigenvalues ordered by
decreasing module, $\pi^A_i$ for its left eigenfunctions and $\phi^A_i$ for its
right eigenfunctions. Recall that the principal eigenvalue $\lambda^A_0$ is real and
positive, and that $\pi^A_0(x)$ and $\phi^A_0(x)$ can be chosen real and
positive as well. \added{We also choose to normalise the eigenfunctions in such a way that
}
\begin{equation}
\added{ 
 \int_A \pi^A_i(x)\phi^A_j(x)\6x = \delta_{ij}\;.
}\end{equation}

\begin{lemma}
Under the spectral gap condition $\abs{\lambda^A_1} < \lambda^A_0$, we have 
\begin{equation}
 \lim_{n\to\infty} \frac{h_{n-1}(y)}{h_n(x)} 
 = \frac{1}{\lambda^A_0} \frac{\phi^A_0(y)}{\phi^A_0(x)}\;.
\end{equation} 
\end{lemma}
\begin{proof}
We can write 
\begin{equation}
 k_A(x,y) = \lambda^A_0 \Pi_0(x,y) + k_\perp(x,y)\;,
\end{equation} 
where $\Pi_0(x,y) = \phi^A_0(x)\pi^A_0(y)$ is the projector on the subspace of
$\lambda^A_0$, and the remainder $k_\perp$ satisfies $\Pi_0 k_\perp = 0$, $k_\perp
\Pi_0 = 0$. Furthermore, $k_\perp$ has spectral radius $\abs{\lambda^A_1}$.
Therefore 
\begin{equation}
 k^n_A(x,y) = (\lambda^A_0)^n \Pi_0(x,y) + k_\perp(x,y)^n\;,
\end{equation} 
and thus 
\begin{equation}
 h_n(x) = \int_A k^n_A(x,y) \6y 
 = (\lambda^A_0)^n \replaced{\phi_0^A(x)}{\phi_0^A(y)} +
\bigOrder{\abs{\lambda^A_1}^n}\;. 
\end{equation} 
The result follows at once from the spectral-gap assumption. 
\end{proof}

We have thus obtained 
\begin{equation}
 \bar k_A(x,y) = \frac{1}{\lambda^A_0} \frac{\phi^A_0(y)}{\phi^A_0(x)} k_A(x,y)\;.
\end{equation} 

\begin{corollary}
The eigenvalues and eigenfunctions of $\bar K_A$ are given by 
\begin{equation}
 \bar\lambda^A_n = \frac{\lambda^A_n}{\lambda^A_0}\;, 
 \qquad
 \bar\pi^A_n(x) = \pi^A_n(x)\phi^A_0(x) 
 \qquad\text{and}\qquad 
 \bar\phi^A_n(x) = \frac{\phi^A_n(x)}{\phi^A_0(x)}\;.
\end{equation} 
\end{corollary}
\begin{proof}
A direct computation shows that 
$K_A \phi^A_n = \lambda^A_n \phi^A_n \Leftrightarrow \bar K_A \bar \phi^A_n = \bar
\lambda^A_n \bar\phi^A_n$, and similarly for the left eigenfunctions. 
\end{proof}

\section{Floquet theory}
\label{app:Floquet} 
\added{Floquet theory and its application to the stability of periodic orbits
is explained in many standard text books, such as~\cite[Chapters~III
and~VI]{Hale}. Here we briefly recall some important facts and notations used
in the present work.}

Consider a $d+1$-dimensional deterministic ODE
\begin{equation}
\dot{z} = f\pth{z}\; ,
\label{Eq:ODEFloquet}
\end{equation}
where $f \in \cC^2(\cD_0,\real^{d+1})$. We assume that this system admits a
periodic solution $\gamma$ \added{of period $T$} with associated orbit $\Gamma$.
We introduce the variable $\varphi \in \real / \mathbb{Z}$ and set $\Gamma\pth{
\varphi} = \gamma\pth{T \varphi}$. Note that
\begin{equation}
\dfrac{\dd }{\dd \varphi} \Gamma\pth{\varphi} = T f\pth{\Gamma\pth{\varphi}}\;
,\label{EqSystCoordEquaDiffGamma}
\end{equation} 
so that $\dot{\varphi}=1/T$ is constant on the periodic orbit.
In order to analyse the dynamics near $\Gamma$,  we start by linearising the
equation. Let $A\pth{\varphi}= \partial_z f\pth{ \Gamma\pth{\varphi}} $
 be the Jacobian matrix of $f$ at $\Gamma(\varphi)$. The linearisation around
the periodic orbit is given by
\begin{equation}
\added{
\dfrac{\dd }{\dd \varphi} \zeta = T A\pth{\varphi} \zeta\;.}
\end{equation}
Therefore \added{$\zeta\pth{\varphi} =
U\pth{\varphi,\varphi_0}\zeta\pth{\varphi_0}$}, where the principal solution
$U\pth{\varphi,\varphi_0}$ satisfies
\begin{equation}
\added{\partial_\varphi} U\pth{\varphi,\varphi_0}= T \replaced{A\pth{\varphi}}{A\pth{\varphi, \varphi_0}} U\pth{\varphi,
\varphi_0}, \qquad U\pth{\varphi\added{_0},\varphi_0} = \id \; .
\end{equation}
Since \replaced{$A(\varphi)=A(\varphi+1)$}{$A\pth{\varphi, \varphi_0}=A\pth{\varphi+1, \varphi_0}$} for all $\varphi$,
Floquet's theorem allows us to decompose the principal solution as 
\begin{equation}
\label{Eq:FloquetTheorem}
U\pth{\varphi, \varphi_0} = P\pth{\varphi, \varphi_0}
\e^{T\pth{\varphi-\varphi_0} B\pth{\varphi_0}}\; ,
\end{equation}
where
 $P\pth{\cdot , \varphi_0}$ is periodic with same period as 
$A\pth{\cdot}$, i.e.\ $1$, and $P$ satisfies $P\pth{\varphi_0,\varphi_0}= \id$,
 and $B\pth{\varphi_0}$ is a constant matrix \added{which can always be chosen to be real
 even if it means taking $P(\cdot , \varphi_0)$ to be $2-$periodic}.
\added{Note that $P$ satisfies}
\begin{equation}
\label{eq:ODE_for_P} 
\added{\dfrac{\dd }{\dd \varphi} 
P(\varphi, \varphi_0) = T\bigbrak{A(\varphi) P(\varphi , \varphi_0)  
- P(\varphi,\varphi_0)  B(\varphi_0)}\;.}
\end{equation}
The asymptotic behaviour  of $\Gamma\pth{\varphi}$ only depends on the
eigenvalues of $T B\pth{\varphi_0}$, which are called characteristic exponents
(or Floquet exponents) of
$\Gamma$. The matrix $U\pth{1+\varphi_0, \varphi_0}=\exp\pth{T B\pth{\varphi_0}
}$ is called the monodromy matrix in $\varphi_0$, and its eigenvalues are called
the characteristic multipliers.
\added{Note that  Floquet multipliers do not depend on $\varphi_0$. Indeed, one 
can show that all monodromy matrices are similar and thus have the same eigenvalues.} 
Differentiating \eqref{EqSystCoordEquaDiffGamma} with respect to $\varphi$,  we
observe that
\begin{equation}
\dfrac{\dd}{\dd \varphi} \Gamma '\pth{\varphi} = T \dfrac{\dd}{\dd \varphi}
f\pth{\Gamma\pth{\varphi}} = \added{T} A\pth{\varphi} \Gamma '\pth{\varphi}\; .
\end{equation}
\added{Thus, owing to periodicity, we have}
\begin{equation}
\Gamma '\pth{\varphi}= \Gamma '\pth{\varphi +1 }= U\pth{\varphi +1 ,\varphi }
\Gamma '\pth{\varphi}\; ,
\label{EqValeurPropre1PrA}
\end{equation}
\added{showing that} $1$ is an eigenvalue of  $U\pth{1+\varphi, \varphi}$ with
eigenvector $\Gamma'\pth{\varphi}$.
\begin{proposition}
\label{Prop:FloquetChangeOfVariable}
There exist $L>0$ and a $d \times d$ triangular matrix $\Lambda$  such that
system \eqref{Eq:ODEFloquet} is equivalent for $\norm{x}< L$ to
\begin{align}
\nonumber
\dot{x} &= \Lambda x + \Order{{\norm{x}}^2}\\
\dot{\varphi} &= \frac{1}{T} + \Order{{\norm{x}}^2}\; .
\end{align}
\end{proposition}

\begin{proof}
We are going to define the change of coordinates explicitly but we first
introduce some notations. \added{Let $\hat{\Lambda}=S^{-1} B
S=\text{diag}(0,\Lambda)$ be the Jordan canonical form of the constant matrix
$B$ defined in~\eqref{Eq:FloquetTheorem}, where $\Lambda\in\R^{d\times d}$. We
also write $P(\varphi,\varphi_0)S = 
 \crochets{u\pth{\varphi},R\pth{\varphi}}$, 
where $u$ is a column vector of dimension $d+1$ 
and $R$ is a matrix of dimension $(d+1 )\times d$.
It follows from~\eqref{eq:ODE_for_P} that the vector $u$ and the matrix $R$
satisfy the equations}
\begin{align}
\nonumber
u'(\varphi) &= T A(\varphi) u(\varphi)\; ,\\
R'\pth{\varphi} &= T \pth{A(\varphi) R\pth{\varphi}- R\pth{\varphi} \Lambda }\; .
\label{EqMatriceRDerivee}
\end{align}
Note that we can choose the matrix $S$ such that $u(\varphi) =
\Gamma'(\varphi)$.

We now introduce the transformation
\begin{equation} \label{EqTransfCgmtVar}
 z =  \Gamma\pth{\varphi} + R\pth{\varphi} x\; .
\end{equation}
We can first check that this transformation is well defined in a neighbourhood
of $\Gamma$. Indeed,
if $F\pth{z,x,\varphi} = \Gamma\pth{\varphi} + R\pth{\varphi} \added{x}-z$, the
\replaced{partial}{paritial}
derivatives of $F$ with respect of $x$ and $\varphi$ are
\begin{align}
\nonumber
\dfrac{\partial F}{\partial \varphi} &= \Gamma'\pth{\varphi} + R'\pth{\varphi}
x\;,\\
\dfrac{\partial F}{\partial x} &= R\pth{\varphi}\;.
\end{align}
For $x=0$, \added{we have $\det\crochets{\partial_\varphi F,
\partial_x F} \neq 0$ for all $\varphi$, since
$\crochets{\Gamma'\pth{\varphi}, R\pth{\varphi}}$ is the matrix
$P(\varphi,\varphi_0)S$ which is invertible.}

If $z(t) = \Gamma\pth{\varphi\pth{t}} + R\pth{\varphi\pth{t}} x\pth{t}$
satisfies $\dot{z} = f\pth{z}$ then
\begin{equation}
f\pth{ \Gamma\pth{\varphi} + R\pth{\varphi} x }= \dot{\varphi}
\Gamma'\pth{\varphi} + \dot{\varphi} R'\pth{\varphi} x + R\pth{\varphi} 
\dot{x}\; .
\end{equation}  
Performing a Taylor expansion of the left-hand side and
using~\eqref{EqMatriceRDerivee}, we obtain
\begin{equation}
  \Order{\norm{x}^2}= \bigpar{\dot{\varphi}-\frac{1}{T}} 
  \bigbrak{ \Gamma'(\varphi)  + T A(\varphi) R(\varphi)x}+ R(\varphi)(
\dot{x} - \dot{\varphi} T \Lambda x)\;.
\end{equation}
The result follows by projecting on a normal vector to the space generated by
the column vectors of $R$.
\end{proof}

\bibliographystyle{plain}
\bibliography{ref}

\tableofcontents

\vfill

\bigskip\bigskip\noindent
{\small
Universit\'e d'Orl\'eans, Laboratoire {\sc Mapmo} \\
{\sc CNRS, UMR 7349} \\
F\'ed\'eration Denis Poisson, FR 2964 \\
B\^atiment de Math\'ematiques, B.P. 6759\\
45067~Orl\'eans Cedex 2, France \\
{\it E-mail addresses: }
{\tt manon.baudel@etu.univ-orleans.fr},
{\tt nils.berglund@univ-orleans.fr}

\end{document}